\documentclass[11pt,letterpaper,leqno]{article}
\usepackage[T1]{fontenc}
\usepackage{times}
\usepackage[scaled=.9]{helvet}
\usepackage{textcomp}
\usepackage[margin=1in]{geometry}
\usepackage{xcolor}
\usepackage{authblk}
\usepackage{microtype}
\usepackage{etoolbox}
\PassOptionsToPackage{cmex10}{amsmath}
\RequirePackage{amsthm,amsmath,amsfonts,amssymb,mathtools}
\RequirePackage[numbers]{natbib}
\RequirePackage{aliascnt}
\RequirePackage{enumitem,booktabs,array}
\RequirePackage[colorlinks,citecolor=blue,urlcolor=blue,linkcolor=blue]{hyperref}
\RequirePackage[nameinlink,noabbrev]{cleveref}

\usepackage{mathrsfs}

\numberwithin{equation}{section}
\allowdisplaybreaks[3]
\theoremstyle{plain}
\newtheorem{theorem}{Theorem}[section]
\newaliascnt{proposition}{theorem}
\newtheorem{proposition}[proposition]{Proposition}
\aliascntresetthe{proposition}
\newaliascnt{lemma}{theorem}
\newtheorem{lemma}[lemma]{Lemma}
\aliascntresetthe{lemma}
\newaliascnt{corollary}{theorem}
\newtheorem{corollary}[corollary]{Corollary}
\aliascntresetthe{corollary}
\newaliascnt{claim}{theorem}

\aliascntresetthe{claim}
\theoremstyle{definition}
\newaliascnt{definition}{theorem}

\aliascntresetthe{definition}
\newaliascnt{assumption}{theorem}

\aliascntresetthe{assumption}
\newaliascnt{remark}{theorem}
\newtheorem{remark}[remark]{Remark}
\aliascntresetthe{remark}

\newcommand{\E}{\mathbb{E}}
\renewcommand{\P}{\mathbb{P}}
\newcommand{\R}{\mathbb{R}}

\hypersetup{linktoc=page}
\makeatletter
\patchcmd{\l@section}{1.0em \@plus\p@}{0.5em \@plus\p@}{}%
  {\PackageError{arxiv-layout}{Could not adjust contents spacing}{}}
\makeatother
\date{}
\title{Berry--Esseen bounds for products of random matrices under a \texorpdfstring{$2+\delta$}{2+delta} polynomial moment}
\hypersetup{pdftitle={Berry--Esseen bounds for products of random matrices under a 2+delta polynomial moment},pdfauthor={Zhi-Jun Cai, Qi-Man Shao, Hui Xiao, Zhuo-Song Zhang}}

\author[1]{Zhi-Jun Cai}
\author[2]{Qi-Man Shao}
\author[3]{Hui Xiao}
\author[2]{Zhuo-Song Zhang}
\affil[1]{Department of Statistics and Data Science,
The Chinese University of Hong Kong, Hong Kong, China\\
\href{mailto:zhijuncai@cuhk.edu.hk}
{\textnormal{zhijuncai@cuhk.edu.hk}}}
\affil[2]{Department of Statistics and Data Science and
Shenzhen International Center for Mathematics,
Southern University of Science and Technology, Shenzhen 518055, China\\
\href{mailto:shaoqm@sustech.edu.cn}{\textnormal{shaoqm@sustech.edu.cn}}, \href{mailto:zhangzs3@sustech.edu.cn}{\textnormal{zhangzs3@sustech.edu.cn}}}
\affil[3]{State Key Laboratory of Mathematical Sciences,
Academy of Mathematics and Systems Science,
Chinese Academy of Sciences,  Beijing 100190, China\\
\href{mailto:xiaohui@amss.ac.cn}{\textnormal{xiaohui@amss.ac.cn}}}
\begin{document}
\maketitle
\begin{abstract}
Let $(g_k)_{k \geq 1}$ be a sequence of independent and identically distributed matrices with law $\mu$ on $\mathrm{GL}(d,\mathbb{R})$, where $d\geq2$ is an integer. Suppose that the closed semigroup generated by the support of $\mu$ is strongly irreducible and proximal and that the polynomial moment of order $2+\delta$ is finite, where $0<\delta\leq 1$. We establish a Berry--Esseen bound of order $n^{-\delta/2}$ for $\log \| g_n \cdots g_1 v\|$, where $v$ is a nonzero vector in $\mathbb R^d$ and $\|\cdot\|$ is an arbitrary norm on $\mathbb R^d$. As an application, we also obtain a Berry--Esseen bound of order $n^{-\delta/2}$ for $\log \| g_n \cdots g_1\|$, where $\|\cdot\|$ is any matrix norm. The proof combines an $m$-dependent approximation with Stein's method. The argument uses identities for exchangeable pairs constructed by resampling individual matrices, together with a concentration inequality for sums of the approximating variables.
\end{abstract}


\tableofcontents

\section{Introduction}
\label{sec-introduction}

We study Berry--Esseen bounds for products of independent and
identically distributed random matrices under polynomial logarithmic
moment conditions. Let $G = \mathrm{GL}(d,\mathbb{R})$ be the group of all invertible
$d\times d$ real matrices, where $d\geq 2$ is an integer.  Let $(g_k)_{k\geq1}$ be a sequence of independent random
matrices with common law $\mu$ on $G$. Fix a Euclidean norm
$\|\cdot\|$ on $\mathbb{R}^d$ and denote its induced operator norm by
the same symbol. The quantities of interest are the logarithmic
growth along a deterministic direction, $\log\|g_n \cdots g_1 v\|$ for
$\|v\|=1$, and the logarithm of the operator norm, $\log\|g_n \cdots g_1\|$.
For $g\in G$, write $N(g)=\max\{\|g\|,\|g^{-1}\|\}$. We say that
$\mu$ has a finite polynomial moment of order $q$ if
\begin{align}\label{eq-log-moment}
\int_{G}[\log N(g)]^q\,\mu(\mathrm{d}g)<\infty.
\end{align}
Under a finite first moment condition, Furstenberg and Kesten \cite{FK1960} established the following strong law of large numbers: 
$$\frac{1}{n} \log \|g_n \cdots g_1\| \xrightarrow[n\to\infty]{} \lambda_\mu \quad \text{almost surely}, $$ where
$\lambda_\mu=\lim_{n\to\infty} \frac{1}{n} \mathbb{E}\log\|g_n \cdots g_1\|$ is called the
top Lyapunov exponent of $\mu$
and it is independent of the choice of the norm $\|\cdot \|$. Throughout this paper, the closed semigroup
generated by the support of $\mu$ is assumed to be strongly
irreducible and proximal; the precise definitions will be given in
 \Cref{sec-main-results}.  Under these assumptions, Benoist and Quint \cite{BQ2016}
proved that, if $\mu$ has a finite second moment,
then there exists a constant $s>0$ such that
\begin{align*}
	\sup_{\|v\|=1}\sup_{z\in\mathbb{R}}
	\Big|
	\mathbb{P}\Big(
	\frac{\log\|g_n \cdots g_1 v\|-n\lambda_\mu}{s\sqrt{n}}\leq z
	\Big)-\Phi(z)
\Big|
	& \xrightarrow[n\to\infty]{} 0,\\
	\sup_{z\in\mathbb{R}}
\Big|
	\mathbb{P}\Big(
	\frac{\log\|g_n \cdots g_1 \|-n\lambda_\mu}{s\sqrt{n}}\leq z
	\Big)-\Phi(z)
	\Big|
	&\xrightarrow[n\to\infty]{} 0,
\end{align*}
where $\Phi$ is the standard normal distribution function on $\mathbb R$.
A natural quantitative question is how fast these errors tend
to zero under stronger polynomial moment conditions.

Under an exponential moment condition on $\mu$,
namely $\int_{G} N(g)^{\tau}\,\mu(\mathrm{d}g)<\infty$ for some constant $\tau>0$,
Le Page \cite{LePage1982} used spectral properties of the
transfer operators associated with the projective action
to establish the Berry--Esseen rate $O(n^{-1/2})$ for
$\log\|g_n \cdots g_1 v\|$, uniformly over unit vectors $v$.
In the same setting, Xiao, Grama and Liu \cite{XGL2022, XGL2025b}
extended this bound to expectations weighted by H\"older
continuous functions of the terminal projective state,
and also obtained an Edgeworth expansion and precise
moderate deviation expansions.
For the logarithm of the operator norm, the same authors
\cite{XGL2021} proved a bound of order $n^{-1/2}\log n$.

With polynomial moment assumptions, Cuny,
Dedecker and Jan \cite{CDJ2017} developed weighted coupling
estimates and obtained quantitative limit theorems,
including rates in Wasserstein distances.
For a finite polynomial moment of order $q\in(2,3]$,
their Wasserstein bounds imply Kolmogorov bounds of order
$n^{-(q-2)/(2(q-1))}$ when $2<q<3$ and
$n^{-1/4}\sqrt{\log n}$ when $q=3$, uniformly over
deterministic initial directions.
The rate $O(n^{-1/2})$ was obtained by Jirak
\cite{Jirak2023} for the stationary norm cocycle under
the stronger assumption of a finite polynomial moment
of some order $q>8$, as an application of his general
Berry--Esseen bounds for weakly dependent sequences.

Sharper Kolmogorov bounds in the range $2<q\leq3$
were obtained by Cuny, Dedecker, Merlev\`ede and Peligrad
\cite{CDMP2023}.
Their Theorem~2.1 gives the rate
$O((\log n/n)^{q/2-1})$ under a finite polynomial moment
of order $q$, for the stationary norm cocycle, for
deterministic initial directions uniformly over unit
vectors, and for the logarithm of the operator norm.
Their Theorem~2.2 also reduces the moment requirement
for the rate $O(n^{-1/2})$ to a finite fourth logarithmic
moment, for all these quantities.
At $q=3$, however, the bound in Theorem~2.1 still contains
an additional factor $\sqrt{\log n}$.
They explicitly ask whether the conclusion of
Theorem~2.2 remains valid under a finite third
polynomial moment alone.

The rate $O(n^{-1/2})$ under a finite third logarithmic
moment is known in several special settings.
Dinh, Kaufmann and Wu \cite{DKW2023} established this rate
by a spectral method for the norm cocycle and the logarithms
of the absolute values of matrix coefficients in the setting
of non-elementary random walks on $\mathrm{SL}(2,\mathbb{C})$.
For products of nonnegative allowable matrices, Cuny,
Dedecker and Merlev\`ede \cite{CDM2024} used exponentially
decaying coupling coefficients to establish the Berry--Esseen
rate $O(n^{1-q/2})$ for the norm cocycle and the logarithm
of the matrix norm, assuming strict contraction,
a polynomial moment condition of order $q\in(2,3]$,
and a positive limiting variance.
With an additional integrability condition, their
Theorem~8.9 also yields this rate for the logarithms
of matrix coefficients and the spectral radius.
Xiao, Grama and Liu \cite{XGL2025} obtained the rate
$O(n^{-1/2})$ for the logarithms of matrix coefficients,
the matrix norm and the spectral radius, assuming their
condition A2, a positive limiting variance, and a finite
third absolute moment of the logarithm of the sum of
the matrix entries.

In this paper, we prove the rate $n^{-\delta/2}$ under a finite
polynomial moment of order $2+\delta$, for each fixed
$0<\delta\leq1$. In particular, there is a constant $C$, independent
of $n$ and $v$, such that
\[
\sup_{\|v\|=1}\sup_{z\in\mathbb{R}}
\Big|
\mathbb{P}\Big(
\frac{\log\|g_n \cdots g_1 v\|-n\lambda_\mu}{s\sqrt{n}}\leq z
\Big)-\Phi(z)
\Big|
\leq Cn^{-\delta/2}.
\]
Bounds
of the same order hold for the stationary norm cocycle, its
average over initial directions with respect to the stationary
measure, and the logarithm of the operator norm; see
\Cref{thm-main}.

The significance of the result lies in the relationship it establishes between the moment condition and the rate of convergence.  The exponent $\delta/2$ is the classical exponent in
the Berry--Esseen bound for sums of independent random variables
with a finite moment of order $2+\delta$; see, for example,
Petrov \cite{Petrov1975}. Thus the dependence induced by the
projective dynamics causes no loss in this rate under strong
irreducibility and proximality. With $q=2+\delta$, our bounds
remove the factor $(\log n)^{\delta/2}$ from those in
\cite{CDMP2023}. At $\delta=1$, they give an affirmative answer to
the third-moment question posed there, in every fixed dimension
$d\geq2$. The same proof framework treats the entire range
$0<\delta\leq1$ and yields all four normal approximations.

Our approach combines an $m$-dependent approximation with Stein's method. Introduced by Stein \cite{Stein1972}, the method provides a flexible alternative to approaches based on characteristic functions and often simplifies error estimates for dependent random variables. An important technique in Stein's method is the use of exchangeable pairs, developed systematically by Stein \cite{Stein1986}. This approach has also proved effective in nonnormal approximation; for example,
\cite{ChatterjeeShao2011,ShaoZhang2019}.  For a systematic treatment of normal approximation by Stein's method, we refer to Chen, Goldstein and Shao \cite{ChenGoldsteinShao2011}.

The main difficulty in applying this approach to random matrix products is that the independence of the matrices does not imply local dependence of the cocycle increments: replacing one matrix can alter every subsequent projective state. We first censor the centered cocycle and take conditional expectations of the resulting increments given the most recent $m$ matrices. We then construct exchangeable pairs by resampling individual matrices, a coordinate replacement technique related to Wu's physical dependence coefficients \cite{Wu2005} and Chatterjee's normal approximation method \cite{Chatterjee2008}.

A key ingredient is a conditional concentration inequality for sums of these approximating variables. Concentration inequalities have been used in Stein's method to control errors in distribution functions; see, for example, \cite{ChenShao2007,ChenFang2015,FangLuoShao2020}. We prove the required inequality by bounding the concentration probability for a block in terms of that for a shorter block and then iterating the resulting bound. For related uses of recursive arguments in Stein's method, see \cite{Gotze1991,FangRollin2015,LiuZhang2023}. The inequality controls the short-interval probabilities arising in Stein's identity and the errors in distribution functions when passing from the $m$-dependent approximation to the original sums.

Throughout this paper, we use the following notation.
Let $\Phi$ and $\varphi$ denote the standard normal distribution
function and density, respectively.  For a real-valued random variable $X$ and $p\geq1$, write
$\lVert X\rVert_{p}:=(\mathbb{E}|X|^{p})^{1/p}$. For a measure $\mu$ and a $\mu$-integrable function $f$, write
$\mu(f):=\int f(x)\,\mu(\mathrm{d}x)$. For real-valued random
variables $Y$ and $Z$, write
$d_{\mathrm{K}}(\mathcal{L}(Y),\mathcal{L}(Z))
=\sup_{z\in\mathbb{R}}|\mathbb{P}(Y\leq z)-\mathbb{P}(Z\leq z)|$
for their Kolmogorov distance. Throughout the paper, $C$ and $c$ denote positive finite constants
that may change from line to line. Unless stated otherwise, they
may depend on $\mu$, $\delta$, $d$, and the fixed norm, but not on
$n$, the censoring level $a$, $m$, or the block length.

The rest of this paper is organized as follows.
\Cref{sec-main-results} introduces the notation and states
the main theorem.
\Cref{sec-preliminaries} constructs the $m$-dependent approximation and its martingale decomposition
and establishes the estimates needed in the subsequent proofs.
\Cref{sec-concentration} proves a conditional concentration inequality
for the sums of the approximating variables.
\Cref{sec-proof} combines Stein's method with the preceding approximation
arguments to prove the four bounds in \Cref{thm-main}.
The proofs of several auxiliary lemmas are deferred to 
\Cref{sec-app-a}.

\section{Main results}
\label{sec-main-results}

We use the notation \(G\), \(\mu\) and \(N(g)\) introduced in \Cref{sec-introduction}. Set \(A_{0}=I_{d}\) and $A_n = g_n \cdots g_1$ for $n \geq 1$.  Let $\Gamma_\mu$ be the closed
semigroup generated by $\operatorname{supp}\mu$. We say that $\Gamma_\mu$ is strongly irreducible 
if no finite union of nonzero proper subspaces of $\mathbb{R}^d$ is $\Gamma_{\mu}$-invariant.
We say that $\Gamma_\mu$ is proximal if $\Gamma_\mu$ contains a matrix with an
algebraically simple eigenvalue whose modulus is strictly larger
than the moduli of all its other eigenvalues. These two assumptions
are in force throughout the paper.

Let $X=\mathbb{P}(\mathbb{R}^d)$, and write $[v]$ for the
projective direction of a nonzero vector $v$. The action of $G$ on
$X$ is $g[v]=[gv]$. Its Markov operator is
$Pf(x)=\int_G f(gx)\,\mu(\mathrm{d}g)$ for $x \in X$.  Under the above
assumptions, there is a unique $\mu$-stationary probability measure
$\nu$ on $X$; see \cite{BQ2016,GR1985}. Thus
$\int_X Pf\,\mathrm{d}\nu=\int_X f\,\mathrm{d}\nu$ for every
bounded Borel function $f$.

The norm cocycle is defined by
\begin{equation}
	\sigma(g,[v])=\log\frac{\|gv\|}{\|v\|},
	\qquad g\in G,\quad v\in\mathbb{R}^d\setminus\{0\}.
	\label{eq-aa-03}
\end{equation}
It satisfies the cocycle identity
$\sigma(gh,x)=\sigma(g,hx)+\sigma(h,x)$ for all \(g,h\in G\) and \(x\in X\), and hence
\begin{equation}
	\sigma(A_n,x)
	=\sum_{k=1}^n\sigma(g_k,A_{k-1}x).
	\label{eq-aa-05}
\end{equation}
Under a finite first polynomial moment,
the Furstenberg formula \cite{BL1985,Furstenberg1963} gives
\begin{equation*}
	\lambda_\mu
	=\int_X\int_G\sigma(g,x)\,\mu(\mathrm{d}g)\nu(\mathrm{d}x).
\end{equation*}

Let $W_0$ have law $\nu$, independently of
$(g_{k})_{k\geq1}$, and put $W_k=A_kW_0$. Define  $\mathscr{F}_{0}=\sigma(W_{0})$ and
$\mathscr{F}_{k}=\sigma(W_0,g_1,\ldots,g_{k})$
for $k\geq1$. The chain
$(W_k)_{k\geq0}$ is stationary with transition operator $P$.
Define
\begin{align*}
	S_{n,x}
	&=\sigma(A_n,x)-n\lambda_\mu,
	\quad x\in X,\qquad
	S_{n}
	=\sigma(A_n,W_0)-n\lambda_\mu,
\\
	\overline{S}_{n}
	&=\int_X\sigma(A_n,x)\,\nu(\mathrm{d}x)-n\lambda_\mu.
\end{align*}
By \eqref{eq-aa-05}, the stationary sum also has the representation
$S_n=\sum_{k=1}^n\{\sigma(g_k,W_{k-1})-\lambda_\mu\}$.
Under a finite second polynomial moment, Benoist and Quint
\cite{BQ2016} proved that
$s^2:=\lim_{n\to\infty}n^{-1}\mathbb{E}S_n^2$ exists and is
strictly positive. We write $s=\sqrt{s^2}$.

The main result of this paper is stated below.
\begin{theorem}	\label{thm-main}
	Let $0<\delta\leq1$. Assume that $\Gamma_\mu$ is strongly
	irreducible and proximal and that
	\begin{equation}
		\int_G\bigl[\log N(g)\bigr]^{2+\delta}\,\mu(\mathrm{d}g)<\infty.
		\label{eq-aa-08}
	\end{equation}
	Then  there exists a finite constant $C>0$, depending
	only on $\mu$, $\delta$, $d$, and the fixed norm on
	$\mathbb{R}^d$, such that, for every $n\geq1$,
	\begin{align}
		\sup_{z\in\mathbb{R}}
		\Big|
		\mathbb{P}\Big(\frac{S_n}{s\sqrt{n}}\leq z\Big)-\Phi(z)
	\Big|
		&\leq Cn^{-\delta/2},
		\label{eq-tt-01}\\
		\sup_{x\in X}\sup_{z\in\mathbb{R}}
		\Big|
		\mathbb{P}\Big(\frac{S_{n,x}}{s\sqrt{n}}\leq z\Big)-\Phi(z)
		\Big|
		&\leq Cn^{-\delta/2},
		\label{eq-tt-02}\\
		\sup_{z\in\mathbb{R}}
	\Big|
		\mathbb{P}\Big(
		\frac{\overline{S}_n}{s\sqrt{n}}\leq z
	\Big)-\Phi(z)
		\Big|
		&\leq Cn^{-\delta/2},
		\label{eq-tt-03}\\
		\sup_{z\in\mathbb{R}}
	\Big|
		\mathbb{P}\Big(
		\frac{\log\|A_n\|-n\lambda_\mu}{s\sqrt{n}}\leq z
		\Big)-\Phi(z)
	\Big|
		&\leq Cn^{-\delta/2}.
		\label{eq-tt-04}
	\end{align}
\end{theorem}

\begin{remark}
Although \Cref{thm-main} is stated for the Euclidean norm $\|\cdot\|$ of vectors in $\mathbb{R}^d$ and the induced operator norm 
$\|\cdot\|$ for matrices in $\mathrm{GL}(d,\mathbb{R})$, the Berry--Esseen bounds still hold for arbitrary vector norms and matrix norms. 
This follows from the fact that all norms on a finite-dimensional space are comparable with each other. 
\end{remark}

\begin{remark}
\Cref{thm-main} is stated for products of random matrices in the group $\mathrm{GL}(d, \mathbb R)$ 
over the field of real numbers $\mathbb R$. 
However, they can be extended to the setting of $\mathrm{GL}(d, \mathbb C)$ or $\mathrm{GL}(d, \mathbb K)$, 
where $\mathbb K$ is a non-Archimedean local field. 
\end{remark}

\section{Preliminaries}\label{sec-preliminaries}
Throughout this section, we assume that the conditions of
\Cref{thm-main} hold for a fixed $0<\delta\leq 1$,
unless a different moment assumption is explicitly stated.  We establish the estimates used in the subsequent proofs. 
\subsection{Martingale decomposition and censoring}
For $[v],[w]\in X$, we use the projective distance
\begin{equation}
	\rho([v],[w])
	=
	\frac{\lVert v\wedge w\rVert}
	{\lVert v\rVert\lVert w\rVert},
	\label{eq-ab-02}
\end{equation}
where $v$ and $w$ are nonzero representatives of $[v]$ and $[w]$,
respectively, $v\wedge w$ denotes their exterior product, and
we still use $\lVert\cdot\rVert$ to denote the induced norm on
$\wedge^{2}\mathbb{R}^{d}$.
We first collect some estimates for the norm cocycle and the
projective action.
The weighted summability estimate for projective contraction
follows from \cite[Lemma~6]{CDJ2017}.
The remaining bounds are elementary.  For completeness, we provide a proof in \Cref{sec-app-a}.
\begin{lemma}\label{lem-A1}
	The following statements hold.
	\begin{itemize}
		\item [(i)] There exists a constant $c>0$ such that
		\begin{equation}
			\sum_{r=1}^{\infty}
			r^{\delta}
			\sup_{x,y\in X}
			\P\bigl(\rho(A_{r}x,A_{r}y)>e^{-cr}\bigr)
			<\infty.
			\label{eq-ab-05}
		\end{equation}
		
		\item [(ii)] For every invertible $g$ and all $x,y\in X$, we have
		\begin{align}
			|\sigma(g,x)-\sigma(g,y)|
			&\leq
			\min\bigl\{
			2\log N(g),
			\sqrt{2} N(g)^{2}\rho(x,y)
			\bigr\},
			\label{eq-ab-06}\\
			\rho(gx,gy)
			&\leq
			N(g)^{2}\rho(x,y).
			\label{eq-ab-07}
		\end{align}
	\end{itemize}
\end{lemma}

 For any $k\geq 1$ and $x\in X$, let $$X_{k,x}=\sigma(g_{k}, A_{k-1}x)-\lambda_{\mu}.$$
For any $p\geq 1$ and $k\geq 1$, define
\[
\delta_{p,\infty}^{p}(k)= \sup_{x,y\in X}\E|X_{k,x}-X_{k,y}|^{p}.
\]

Proposition~3 of \cite{CDJ2017} gives a weighted summability
result for the coupling coefficients $\delta_{p,\infty}^{p}(k)$
under a finite polynomial moment assumption.
For a precise statement, see the following lemma.
\begin{lemma}\label{lem-1}
Assume that $\Gamma_\mu$ is strongly irreducible and proximal
and that \eqref{eq-log-moment} holds for some $q>1$.
Then, 
	\begin{align}\label{eq-ww-01}
		\sum_{k\geq 1} k^{q-p-1}\delta_{p,\infty}^{p}(k) < \infty,
		\qquad \forall\, p\in[1,q).
	\end{align}
\end{lemma}

Recall that $Pf(x)=\int_G f(gx)\,\mu(\mathrm{d}g)$ for $x \in X.$ Define 
\begin{align}\label{def-phi-x}
\phi(x)=\int_G (\sigma(g,x)-\lambda_{\mu})\,\mu(\mathrm{d}g), \quad  x \in X. 
\end{align}
Benoist and Quint \cite[Proposition~4.9]{BQ2016}
proved that the norm cocycle can be centered.
The following lemma gives a uniformly convergent
series representation of a solution to the corresponding
Poisson equation.
We include a short proof for completeness, since
the same argument will be used for the censored observable.
\begin{lemma}\label{lem-centering}
	The Poisson equation $(I-P)\psi=\phi$ admits a bounded
	solution $\psi$ on $X$ with $\nu$-mean zero, given by
	\[
	\psi(x)=\sum_{j=0}^{\infty}P^{j}\phi(x),  \quad  x \in X, 
	\]
	where the series converges uniformly on $X$.
	Moreover, the increment
	\begin{equation}\label{eq-ab-12}
		\sigma_{0}(g,x)=\sigma(g,x)-\lambda_{\mu}+\psi(gx)-\psi(x)
	\end{equation}
	satisfies
	\begin{equation}\label{eq-abb-12}
	\int_G \sigma_{0}(g,x)\,\mu(\mathrm{d}g)=0
	\qquad \text{for every } x\in X.
\end{equation}
\end{lemma}

\begin{proof}
	For every $j\geq0$ and $x \in X$, the Markov property gives
	$P^{j}\phi(x)=\E X_{j+1,x}$, and the stationarity of $\nu$ gives
	$\int_{X}P^{j}\phi(y)\,\nu(\mathrm{d}y)=0$.
	Hence, by the definition of $\delta_{1,\infty}$, we have
	\begin{equation}\label{eq-ab-13}
		|P^{j}\phi(x)|
		=
		\Bigl|
		\int_{X}
		\bigl(P^{j}\phi(x)-P^{j}\phi(y)\bigr)
		\,\nu(\mathrm{d}y)
		\Bigr|
		\leq\delta_{1,\infty}(j+1).
	\end{equation}
	Taking $q = 2$ and $p=1$ in \Cref{lem-1},
	we get $\sum_{k\geq1}\delta_{1,\infty}(k)<\infty$.
	Thus the series defining $\psi$ converges uniformly
	to a bounded function satisfying
	$\lVert\psi\rVert_{\infty}
	\leq\sum_{k\geq1}\delta_{1,\infty}(k)$
	and $\int_{X}\psi(x)\,\nu(\mathrm{d}x)=0$.
	Since $P$ is continuous in the supremum norm,
	\[
	P\psi=\sum_{j=1}^{\infty}P^{j}\phi=\psi-\phi.
	\]
	Integrating \eqref{eq-ab-12} with respect to $\mu$ gives
	\[
	\int_{G}\sigma_{0}(g,x)\,\mu(\mathrm{d}g)
	=\phi(x)+P\psi(x)-\psi(x)=0.
	\]
\end{proof}

To obtain the convergence rate under the $2+\delta$ moment condition,
we use a censoring argument (see also \cite{LeungShaoZhang2024} for the terminology ``censoring''). To this end, for any $a\geq 1$, define
the scalar censoring map by
\begin{equation}\label{def-censoring map}
	\tau^{[a]}(t)=
	\begin{cases}
		-a, & t<-a,\\
		t, & -a\leq t\leq a,\\
		a, & t>a.
	\end{cases}
\end{equation}
Define
\begin{align}
		c^{[a]}
	&=\int_{{X}} \int_G 
	\tau^{[a]}\bigl(\sigma(g,x)-\lambda_{\mu}\bigr)\,\mu(\mathrm{d}{g})\nu(\mathrm{d}{x}), \label{eq-ll-01}\\
	H^{[a]}(g,x)
	&=\tau^{[a]}\bigl(\sigma(g,x)-\lambda_{\mu}\bigr)-c^{[a]}, \label{eq-ll-01-Ha} \\
	\phi^{[a]}(x)
&	=\int_G H^{[a]}(g,x)\,\mu(\mathrm{d}{g}).  \label{eq-ll-01-phi-a}
\end{align}
Thus $H^{[a]}$ has mean zero under $\mu\otimes\nu$, and
$\nu(\phi^{[a]})=0$. 
Similarly to \Cref{lem-centering}, the following result gives a uniformly convergent
series representation of a solution to the corresponding 
Poisson equation associated with the censored
observable. In particular, this implies that the censored increment $\sigma_{0}^{[a]}$ remains centered. 

\begin{lemma}\label{lem-centering-02}
Let $a \geq 1$. The Poisson equation $(I-P)\psi^{[a]}=\phi^{[a]}$ admits a bounded
solution $\psi^{[a]}$ on $X$ with $\nu$-mean zero, given by
\begin{align}\label{def-psi-a}
\psi^{[a]}(x)=\sum_{j=0}^{\infty}P^{j}\phi^{[a]}(x),  \quad  x \in X, 
\end{align}
where the series converges uniformly on $X$.
Moreover, the increment
\begin{equation}\label{eq-ab-12-a}
\sigma_{0}^{[a]}(g,x) = H^{[a]}(g,x)+\psi^{[a]}(gx)-\psi^{[a]}(x)
\end{equation}
satisfies
\begin{equation*}
	\int_G \sigma_{0}^{[a]}(g,x)\,\mu(\mathrm{d}g)=0
	\qquad \text{for every } x\in X.
\end{equation*}
\end{lemma}


\begin{proof}
Note that $\tau^{[a]}$ is $1$-Lipschitz.  Hence, for every $g,x,y$,
\begin{align}\label{Lipschitz Ha cocycle}
\lvert H^{[a]}(g,x)-H^{[a]}(g,y)\rvert 
\leq \lvert \sigma(g,x)-\sigma(g,y) \rvert.
\end{align}
Applying this inequality, we have, for any $x,y\in X$,
\begin{align*}
\lvert P^{j}\phi^{[a]}(x) -P^{j}\phi^{[a]}(y)\rvert 
& \leq  \int_{G^{j+1}}  | \sigma(g_{j+1}, g_j \cdots g_1 x) - \sigma(g_{j+1}, g_j \cdots g_1 y) |  \mu(dg_1) \cdots \mu(dg_{j+1}) \notag\\
& \leq   \sup_{x,y \in X}\E|X_{j+1, x} - X_{j+1, y}| =  \delta_{1,\infty}(j+1),\qquad j\geq 0.
\end{align*}
As $\nu$ is $P$-invariant, we have
$\int_X P^{j}\phi^{[a]}(x)\,\nu(\mathrm{d}x) = \nu(\phi^{[a]}) =0$
for every $j\geq0$.  Consequently, we have  
\begin{equation}\label{eq-ac-01}
	\lVert P^{j}\phi^{[a]}\rVert _{\infty}
	\leq \delta_{1,\infty}(j+1).
\end{equation}
Since $\sum_{j=0}^{\infty} \delta_{1,\infty}(j+1)<\infty$ (taking $q = 2$ and $p=1$ in \Cref{lem-1}), 
it follows that the series
$\sum_{j=0}^{\infty}P^{j}\phi^{[a]}(x)$
converges uniformly in $x \in X$, and
\[
\sup_{a \geq 1}\lVert \psi^{[a]}\rVert _{\infty}
\leq  \sup_{a \geq 1} \sum_{j=0}^{\infty} \lVert P^{j}\phi^{[a]}\rVert _{\infty} 
\leq\sum_{j=0}^{\infty}\delta_{1,\infty}(j+1)<\infty.
\]
Since $P$ is continuous in the supremum norm, we may apply it
term by term to the series defining $\psi^{[a]}$. Hence
$P\psi^{[a]}=\sum_{j=1}^{\infty}P^{j}\phi^{[a]}$, and therefore
$(I-P)\psi^{[a]}=\phi^{[a]}$. Using this identity and \eqref{eq-ll-01-phi-a}, we have, for any $a \geq 1$ and $x \in X$, 
\begin{align*}
\int_G \sigma_{0}^{[a]}(g,x)\,\mu(\mathrm{d}{g})
&=\phi^{[a]}(x)+P\psi^{[a]}(x)-\psi^{[a]}(x)=0, 
\end{align*}
as desired. 
\end{proof}

Let $$L(g)=1+|\lambda_{\mu}|+\log N(g).$$
The following lemma quantifies the effect of censoring on the centering constant, the
Poisson solution, and the variance.

\begin{lemma}
\label{lem-clipvariance}
Under the stationary law of
$(g_1,W_{0})$, we have
\begin{equation}\label{eq-ac-07}
\lvert c^{[a]}\rvert \leq C a^{-1-\delta},\qquad
\lVert \psi^{[a]}-\psi\rVert _{\infty}\leq C a^{-\delta},\qquad
\big| \mathbb{E}\bigl[\sigma_{0}^{[a]}(g_1,W_{0})^{2}\bigr]-s^{2}\big|
\leq C a^{-\delta},
\end{equation}
where $C>0$ is a constant depending only on $\mu$, $\delta$,
$d$ and the fixed norm $\| \cdot \|$ on $\mathbb R^d$.
\end{lemma}

\begin{proof}
Since $\sigma(g,x)-\lambda_{\mu}$ has stationary mean zero and
$\lvert\sigma(g,x)-\lambda_{\mu}\rvert \leq L(g)$, it follows that
\begin{align}\label{bound c-a}
	\begin{aligned}
	\lvert c^{[a]}\rvert 
	&=\biggl\lvert  \int_{{X}} \int_G 
	\bigl[\tau^{[a]}\bigl(\sigma(g,x)-\lambda_{\mu}\bigr)-\bigl(\sigma(g,x)-\lambda_{\mu}\bigr)\bigr]
	\,\mu(\mathrm{d}{g})\nu(\mathrm{d}{x})\biggr\rvert \\
	&\leq \mathbb{E}\bigl[L(g_1)\mathbf{1}_{\{L(g_1)>a\}}\bigr] \\
	&\leq a^{-1-\delta}\mathbb{E} (L(g_1)^{2+\delta}).
	\end{aligned}
\end{align}
This proves the first inequality in \eqref{eq-ac-07}. 

The same censoring estimate gives
\begin{align*}
\sup_{x \in X} \biggl\lvert  \int_G 
 \bigl[\tau^{[a]}\bigl(\sigma(g,x)-\lambda_{\mu}\bigr)-\bigl(\sigma(g,x)-\lambda_{\mu}\bigr)\bigr]
 \,\mu(\mathrm{d}{g}) \biggr\rvert 
\leq \mathbb{E}\bigl[L(g_1)\mathbf{1}_{\{L(g_1)>a\}}\bigr]
\end{align*}
and hence, by \eqref{def-phi-x} and \eqref{eq-ll-01-phi-a}, 
\begin{align*}
\lVert \phi^{[a]}-\phi\rVert _{\infty}
& = \sup_{x \in X}  
\left| \int_G  \bigl[ \tau^{[a]} \bigl( \sigma(g,x)-\lambda_{\mu} \bigr) - c^{[a]} \bigr]  \mu(\mathrm{d}{g}) 
 - \int_G  \bigl(\sigma(g,x)-\lambda_{\mu}\bigr)  \mu(\mathrm{d}{g})  \right|  \notag\\
 & \leq  \sup_{x \in X} \biggl\lvert  \int_G 
 \bigl[\tau^{[a]}\bigl(\sigma(g,x)-\lambda_{\mu}\bigr)-\bigl(\sigma(g,x)-\lambda_{\mu}\bigr)\bigr]
 \,\mu(\mathrm{d}{g}) \biggr\rvert  +  \lvert c^{[a]}\rvert  \notag\\ 
& \leq\mathbb{E}\bigl[L(g_1)\mathbf{1}_{\{L(g_1)>a\}}\bigr]
 +\lvert c^{[a]}\rvert 
\leq C a^{-1-\delta}.
\end{align*}
Combining this with \eqref{eq-ab-13} and \eqref{eq-ac-01}, we get 
\begin{align*}
\lVert \psi^{[a]}-\psi\rVert _{\infty}
&\leq\sum_{j=0}^{\lceil a\rceil-1}
 \lVert P^{j}(\phi^{[a]}-\phi)\rVert _{\infty}
 +2\sum_{j=\lceil a\rceil}^{\infty}\delta_{1,\infty}(j+1)\\
&\leq C\lceil a\rceil a^{-1-\delta}
 +2\lceil a \rceil^{-\delta}
 \sum_{j=\lceil a\rceil}^{\infty}(j+1)^{\delta}\delta_{1,\infty}(j+1)\\
&\leq C a^{-\delta},
\end{align*}
where in the last inequality we used \Cref{lem-1} with $q = 2 + \delta$ and $p=1$. 
This proves the second inequality in \eqref{eq-ac-07}. 

For the last inequality in \eqref{eq-ac-07},  it follows from \eqref{eq-ab-12-a} and \eqref{eq-ab-12} that
\begin{align*}
\sigma_{0}^{[a]}(g,x)-\sigma_{0}(g,x)
={}&\tau^{[a]}\bigl(\sigma(g,x)-\lambda_{\mu}\bigr)-\bigl(\sigma(g,x)-\lambda_{\mu}\bigr)-c^{[a]}\\
&+(\psi^{[a]}-\psi)(gx)-(\psi^{[a]}-\psi)(x).
\end{align*}
The preceding estimates imply the pointwise bound
\begin{align}\label{eq-ii-01}
\lvert \sigma_{0}^{[a]}(g,x)-\sigma_{0}(g,x)\rvert 
\leq L(g)\mathbf{1}_{\{L(g)>a\}}+C a^{-\delta}.
\end{align}
Since  $\lvert \sigma_{0}(g,x)\rvert \leq C L(g)$, it follows that
\begin{equation}\label{eq-clipvariance-cross}
	\begin{aligned}
		\mathbb{E}\bigl[
		\lvert\sigma_{0}(g_{1},W_{0})\rvert
		\lvert\sigma_{0}^{[a]}(g_{1},W_{0})
		-\sigma_{0}(g_{1},W_{0})\rvert
		\bigr]
		&\leq C\mathbb{E}\bigl[
		L(g_{1})^{2}\mathbf{1}_{\{L(g_{1})>a\}}
		\bigr]
		+C a^{-\delta}\mathbb{E}L(g_{1})\\
		&\leq C a^{-\delta}.
	\end{aligned}
\end{equation}
Moreover, by \eqref{eq-ii-01} and  Markov's inequality, we have
\begin{equation}\label{eq-clipvariance-square}
	\mathbb{E}\bigl(
	\lvert\sigma_{0}^{[a]}(g_{1},W_{0})
	-\sigma_{0}(g_{1},W_{0})\rvert
	\bigr)^{2}
	\leq C a^{-\delta}+C a^{-2\delta}
	\leq C a^{-\delta}.
\end{equation}
To identify the limiting variance, set
\[
M_n:=\sum_{k=1}^n \sigma_{0}(g_k,W_{k-1})
=S_n+\psi(W_n)-\psi(W_0), 
\]
where $W_n = A_n W_0$.  By \eqref{eq-abb-12} and  stationarity, we have
\begin{align*}
	\frac{1}{n}\mathbb{E}M_n^2&=\frac{1}{n}\sum_{k=1}^{n}  \E [\sigma_{0}(g_k,W_{k-1})]^{2}+\frac{1}{n}\sum_{1\leq k\neq  l}^{n}\E \sigma_{0}(g_k,W_{k-1})\sigma_{0}(g_{l},W_{l-1})\nonumber\\
	&=\frac{1}{n}\sum_{k=1}^{n}\E [\sigma_{0}(g_k,W_{k-1})]^{2} \\
	&=\mathbb{E}\sigma_{0}(g_1,W_0)^2.
\end{align*}
Since $\lVert M_n-S_n\rVert_\infty\leq2\lVert\psi\rVert_\infty$,
the Cauchy--Schwarz inequality gives
\[
\left|
\frac{\mathbb{E}M_n^2}{n}
-\frac{\mathbb{E}S_n^2}{n}
\right|
\leq
\frac{4\lVert\psi\rVert_\infty}{\sqrt{n}}
\left(\frac{\mathbb{E}S_n^2}{n}\right)^{1/2}
+\frac{4\lVert\psi\rVert_\infty^2}{n}
\longrightarrow0,
\]
where we used $n^{-1}\mathbb{E}S_n^2\to s^2$.
Hence $\mathbb{E}\sigma_0(g_1,W_0)^2=s^2$.
Finally, expanding
$(\sigma_{0}^{[a]})^{2}-\sigma_{0}^{2}
=2\sigma_{0}(\sigma_{0}^{[a]}-\sigma_{0})
+(\sigma_{0}^{[a]}-\sigma_{0})^{2}$
and applying \eqref{eq-clipvariance-cross} and
\eqref{eq-clipvariance-square}, we get the last inequality
in \eqref{eq-ac-07}.
\end{proof}

\subsection{An \texorpdfstring{$m$}{m}-dependent approximation and its properties}
In this subsection, we construct an $m$-dependent martingale
approximation for the censored cocycle. We then establish
variance estimates and moment bounds for the effects of
replacing individual matrices.
The construction begins by taking conditional expectations of the censored variables given the most recent $m$ matrices (see also \cite{Jirak2016} for a similar approximation).
Throughout this subsection, the parameter $m$ is an integer with
$m\geq2$, and $a\geq1$.
For $m\geq 2$,  let
\[
\mathscr{F}_{k,m}=\sigma(g_{k-m+1},\ldots,g_{k}),\qquad k\geq m.
\]
Define 
\[
x_{a,m}(g_{1},\ldots,g_{m})
=\int_X H^{[a]}(g_{m},g_{m-1}\cdots g_{1}x)\,\nu(\mathrm{d}{x}),
\qquad
X^{[a]}_{k,m}=x_{a,m}(g_{k-m+1},\ldots,g_{k}).
\]
Note that $W_{k-m}$ has the law $\nu$ and is independent of
$\mathscr{F}_{k,m}$.  Moreover, since $W_n = A_n W_0$, we have 
$
W_{k-1}=g_{k-1}\cdots g_{k-m+1}W_{k-m}.
$
Conditioning on $\mathscr{F}_{k,m}$ therefore yields the
exact projection identity
\[
X^{[a]}_{k,m}
=\mathbb{E}\bigl[H^{[a]}(g_{k},W_{k-1})\mid\mathscr{F}_{k,m}\bigr].
\]
In particular, since $H^{[a]}(g_{k}, W_{k-1})$ has mean zero under $\mu\otimes\nu$, we have $\mathbb{E} X^{[a]}_{k,m}=0$. For $m\geq 2$, denote
\[
U_{k}=(g_{k-m+2},\ldots,g_{k}).
\]
Then the chain $(U_k)$ has the transition operator $P_m$ given as follows: 
for any $u_1,\ldots,u_{m-1}\in G$ and any nonnegative measurable function $F:G^{m-1}\to\mathbb{R}$, 
\begin{align}\label{def-operator-Pm}
(P_mF)(u_1,\ldots,u_{m-1})
=
\int_G
F(u_2,\ldots,u_{m-1},g)\,\mu(\mathrm{d}g).
\end{align}
Thus, at each step, we remove the first coordinate and add
an independent matrix with law $\mu$ at the end.
For any $u_1,\ldots,u_{m-1}\in G$, define
\begin{align}
	\label{eq-vv-02}
	\begin{aligned}
		\phi^{[a]}_{m}(u_{1},\ldots,u_{m-1})
	&=
	\int_G  x_{a,m}(u_{1},\ldots,u_{m-1},g)\,\mu(\mathrm{d}g)
	\\
	&=
	\int_X
	\phi^{[a]}(u_{m-1}\cdots u_{1}x)\,\nu(\mathrm{d}x), 
	\end{aligned}
\end{align}  
where the last identity holds due to \eqref{eq-ll-01-phi-a}. 
Since $\nu(\phi^{[a]})=0$, $\phi_{m}^{[a]}$ has mean zero under the stationary law
$\mu^{\otimes(m-1)}$.  Applying $P_m$ a total of $m-1$ times amounts to integrating
$\phi_m^{[a]}$ with respect to $\mu$ in each of its
$m-1$ variables. Thus
$P_{m}^{m-1}\phi^{[a]}_{m}=0$ pointwise. 
For any $u_1,\ldots,u_{m-1}\in G$, define 
\[
\psi^{[a]}_{m}(u_{1},\ldots,u_{m-1}) = \sum_{j=0}^{m-2}P_{m}^{j}\phi^{[a]}_{m} (u_{1},\ldots,u_{m-1}).
\]
Subtracting $P_m\psi_m^{[a]}$ from $\psi_m^{[a]}$
gives the Poisson equation
\begin{align}\label{eq-vv-01}
	(I-P_{m})\psi_{m}^{[a]}
	=\phi^{[a]}_{m}-P_{m}^{m-1}\phi^{[a]}_{m}
	=\phi^{[a]}_{m}.
\end{align}

For $k\geq m$, define
\begin{equation}\label{eq-ad-01}
	Y_{k,m}^{[a]}
	=X^{[a]}_{k,m}
	+\psi_m^{[a]}(U_k)-\psi_m^{[a]}(U_{k-1}).
\end{equation}
The following lemma gives a martingale decomposition of
$(X_{k,m}^{[a]})_{k\geq m}$ and bounds for
$\psi_m^{[a]}$ and $Y_{k,m}^{[a]}$.

\begin{lemma}\label{lem-finiteG}
	For any $a\geq1$ and $m\geq2$,
	the following statements hold.
	\begin{enumerate}
		\item[(i)]
		The sequence $(Y_{k,m}^{[a]})_{k\geq m}$ is a stationary
		martingale difference sequence with respect to
		$(\mathscr{F}_k)_{k\geq0}$, where $\mathscr{F}_{k}=\sigma(W_0,g_1,\ldots,g_{k})$. 
		\item[(ii)]
		The function $\psi_m^{[a]}$ has the representation
		\begin{equation}\label{eq-ad-02}
			\psi_m^{[a]}(u_1,\ldots,u_{m-1})
			= \sum_{j=0}^{m-2} \int_X
			(P^j\phi^{[a]})
			(u_{m-1}\cdots u_{j+1}x)\,\nu(\mathrm{d}x).
		\end{equation}
		\item[(iii)]
There exists a constant	$C>0$ depending only on $\mu$, $\delta$,
	$d$ and the fixed norm, such that
		\begin{equation}\label{eq-ad-04}
			\begin{aligned}
				\sup_{a\geq1,\;m\geq2}
				\lVert\psi_m^{[a]}\rVert_\infty
				\leq C,\quad
				\lvert Y_{k,m}^{[a]}\rvert
				&\leq C \big( 1+\min\{\log N(g_k),a\} \big),
				\qquad k\geq m.
			\end{aligned}
		\end{equation}
		
		\item[(iv)]
		For all integers $j, n$ with $m \leq j \leq n$,
		\begin{equation}\label{eq-ad-06}
			\sum_{k=j}^{n}Y_{k,m}^{[a]}
			-\sum_{k=j}^{n}X_{k,m}^{[a]}
			=\psi_m^{[a]}(U_n)-\psi_m^{[a]}(U_{j-1}).
		\end{equation}
	\end{enumerate}
\end{lemma}

\begin{proof}
	We first prove (i).
	For fixed $a$ and $m$, each $Y_{k,m}^{[a]}$ is bounded
	and $\mathscr{F}_k$-measurable.
	The sequence is stationary because the matrices $g_k$
	are i.i.d. and each $Y_{k,m}^{[a]}$ is the same measurable
	function of $(g_{k-m+1},\ldots,g_k)$.
By the definitions of $x_{a,m}$ and $\psi_{m}^{[a]}$ and the finite Poisson equation \eqref{eq-vv-01}, 
we have, for any $u_{1},\ldots,u_{m-1} \in G$,  
\begin{align}\label{eq-hh-01}
	\begin{aligned}
	&\int_G \bigl[
	x_{a,m}(u_{1},\ldots,u_{m-1},g)
	+\psi_{m}^{[a]}(u_{2},\ldots,u_{m-1},g)
	-\psi_{m}^{[a]}(u_{1},\ldots,u_{m-1})
	\bigr]\,\mu(\mathrm{d}{g})\\
	&\qquad=
	\phi^{[a]}_{m}(u_{1},\ldots,u_{m-1})
	+P_{m}\psi_{m}^{[a]}(u_{1},\ldots,u_{m-1})
	-\psi_{m}^{[a]}(u_{1},\ldots,u_{m-1})\\
	&\qquad=0.
	\end{aligned}
\end{align}
Since $U_{k-1}$ is $\mathscr{F}_{k-1}$-measurable and
$g_k$ is independent of $\mathscr{F}_{k-1}$
with law $\mu$, the equality \eqref{eq-hh-01} gives
$
\mathbb{E}(Y_{k,m}^{[a]}\mid\mathscr{F}_{k-1})=0 
$
for any $k \geq m$. 
This proves (i).

We next prove (ii).
For $j=0$, the formula follows from the definition of $\phi^{[a]}_{m}$ in \eqref{eq-vv-02}.
Using \eqref{def-operator-Pm} and \eqref{eq-vv-02}, we get 
\begin{align*}
P_{m} \phi^{[a]}_{m}(u_{1},\ldots,u_{m-1})
& = \int_G  \phi^{[a]}_{m}(u_{2},\ldots,u_{m-1}, g) \mu(dg)  \\
&=\int_X \int_G 
 \phi^{[a]}(g u_{m-1}\cdots u_{2}x)
 \,\mu(\mathrm{d}{g})\nu(\mathrm{d}{x})\\
&=\int_X (P\phi^{[a]})(u_{m-1}\cdots u_{2}x)
 \,\nu(\mathrm{d}{x}).
\end{align*}
By induction, $j$ applications of $P_m$ remove the first $j$
coordinates and integrate over $j$ independent matrices
with law $\mu$.  Consequently, we have
\[
P_{m}^{j}\phi^{[a]}_{m}(u_{1},\ldots,u_{m-1})
=\int(P^{j}\phi^{[a]})
 (u_{m-1}\cdots u_{j+1}x)\,\nu(\mathrm{d}{x}),
\qquad 0\leq j\leq m-2.
\]
Summing these identities, we obtain  \eqref{eq-ad-02}.

For (iii), taking supremum norms in \eqref{eq-ad-02}
and using \eqref{eq-ac-01}, we obtain
\[
\lVert \psi_{m}^{[a]}\rVert _{\infty}
\leq\sum_{j=0}^{m-2}\delta_{1,\infty}(j+1)
\leq\sum_{r=1}^{\infty}\delta_{1,\infty}(r) \leq C,
\]
which proves the first inequality in \eqref{eq-ad-04}. 
For the second inequality in \eqref{eq-ad-04}, 
since $L(g)=1+|\lambda_\mu|+\log N(g)$, 
by \eqref{def-censoring map}, we have
$\lvert \tau^{[a]}(\sigma(g,x)-\lambda_{\mu})\rvert \leq\min\{L(g),a\}$.
Moreover, $\lvert c^{[a]}\rvert \leq\mathbb{E} L(g_1)$ (see \eqref{bound c-a}).
Combining the bounds for $\tau^{[a]}$ and $c^{[a]}$
with the definition of $x_{a,m}$ and \eqref{eq-ad-01},
we obtain
\[
\lvert Y_{k,m}^{[a]}\rvert 
\leq \min\{L(g_{k}),a\}
 +\mathbb{E} L(g_1)+2\sup_{a\geq1,\;m\geq2}\lVert \psi_m^{[a]}\rVert _{\infty}
\leq C\big( 1+\min\{L(g_{k}),a\} \big).
\]

Finally, summing \eqref{eq-ad-01} over $k = j, \ldots, n$ yields \eqref{eq-ad-06}.
\end{proof}

 The following lemma quantifies the local oscillations of
$P^{j}\phi^{[a]}$ with respect to the projective distance $\rho$
and provides the corresponding summability and decay estimates. We defer its proof to \Cref{sec-app-a}.
\begin{lemma}\label{lem-absmod}
	For $a\geq1$, $j\geq 0$ and real $r\geq 1$, let
	\[
	\omega_{j,a}(r)
	=\sup_{\rho(x,y)\leq e^{-r}}
	\lvert P^{j}\phi^{[a]}(x)
	-P^{j}\phi^{[a]}(y)\rvert .
	\]
	There is a constant $C>0$ independent of $m,a$ such that
	\begin{align}
		\sum_{r=1}^{\infty}r^{\delta-1}
		\sum_{j=0}^{\infty}\omega_{j,a}(r)
		&\leq C,\quad \text{and}\quad
		\sum_{j=0}^{m-2}\sum_{r=1}^{\infty}\omega_{j,a}(r)
		\leq C m^{1-\delta},\qquad m\geq 2.
		\label{eq-ac-03}
	\end{align}
	For every real $r\geq 1$, one also has
	\begin{equation}\label{eq-ac-04}
		\sum_{j=0}^{\infty}\omega_{j,a}(r)
		\leq C r^{-\delta}.
	\end{equation}
\end{lemma}

The next lemma gives the bounds of  the difference between $\psi^{[a]}(W_k)$ and
$\psi_m^{[a]}(U_k)$ in  $L^1$ norms.
\begin{lemma}
\label{lem-weightedprojection}
For all  $a\geq1$, $m\geq2$ and $k\geq m$, we have
\begin{align}
\mathbb{E}\lvert \psi_{m}^{[a]}(U_{k})-\psi^{[a]}(W_{k})\rvert 
&\leq C m^{-\delta},\label{eq-ad-07}\\
\mathbb{E}\bigl[L(g_{k})
 \lvert \psi_{m}^{[a]}(U_{k})-\psi^{[a]}(W_{k})\rvert \bigr]
&\leq C m^{-\delta},
\label{eq-ad-08}
\end{align}
where we recall that $L(g)=1+|\lambda_\mu|+\log N(g)$.
\end{lemma}

\begin{proof}
	For each summation index $j$, we write $h_{j}:=m-1-j$. Let \(\widetilde W_{k-h_{j}}\) be a random variable with
	distribution \(\nu\), independent of
	\(W_{k-h_{j}}, g_{k-h_{j}+1} ,\ldots,g_{k}\). 
By \eqref{eq-ad-02} and \eqref{def-psi-a}, the definition of $\psi^{[a]}$
and the triangle inequality, we have
	\begin{align}\label{eq-vv-09}
		\begin{aligned}
			\E	\lvert \psi_{m}^{[a]}(U_{k})-\psi^{[a]}(W_{k})\rvert
		& \leq \sum_{j=0}^{m-2} \E
		\big|(P^{j}\phi^{[a]})
		(g_{k}\cdots  g_{k-h_{j}+1} \widetilde{W}_{k-h_{j}})-	(P^{j}\phi^{[a]})
		(W_{k})\big|\\
		&\quad+\sum_{j=m-1}^{\infty}
		\E|(P^{j}\phi^{[a]})
		(W_{k})|.
		\end{aligned}
	\end{align}
For the first term of \eqref{eq-vv-09}, note that
$P^{j}\phi^{[a]}(x)=\E H^{[a]}(\widetilde{g}_{j+1}, \widetilde{g}_{j} \widetilde{g}_{j-1} \cdots \widetilde{g}_{1}x)$,  where $(\widetilde{g}_k)_{k\geq1}$ is an i.i.d. sequence with common distribution \(\mu\), independent of all other random variables.
Then, 
\begin{align}\label{eq-ww-08}
	\begin{aligned}
		&\E\big[	|(P^{j}\phi^{[a]})
		(g_{k}\cdots g_{k-m+j+2}\widetilde{W}_{k-h_{j}})-	(P^{j}\phi^{[a]})
		(W_{k})|\mid W_{k-h_{j}},\widetilde{W}_{k-h_{j}}\big]\\
		&\quad\leq \sup_{x,y}\E\big| H^{[a]}(\widetilde{g}_{j+1}, \widetilde{g}_{j} \cdots \widetilde{g}_{1} g_{k}\cdots  g_{k-h_{j}+1} x)- H^{[a]}(\widetilde{g}_{j+1}, \widetilde{g}_{j} \cdots \widetilde{g}_{1}g_{k}\cdots  g_{k-h_{j}+1} y)\big|\\
		&\quad\leq \sup_{x,y}\E|X_{m,x}-X_{m,y}|\leq \delta_{1,\infty}(m),
	\end{aligned}
\end{align}
where we used \eqref{Lipschitz Ha cocycle} in the second inequality.  
Combining \eqref{eq-ac-01}, \eqref{eq-vv-09} and \eqref{eq-ww-08}, we have
\begin{align}\label{eq-vv-9.1}
\mathbb{E}\big| \psi_{m}^{[a]}(U_{k})-\psi^{[a]}(W_{k})\big| 
&\leq(m-1)\delta_{1,\infty}(m)+\sum_{j=m-1}^{\infty}\delta_{1,\infty}(j+1).
\end{align}
By the definition of $\delta_{p,\infty}$ and the i.i.d.
property of $(g_k)$, for $p\geq1$ and $r\geq1$,
we have
\[
\begin{aligned}
	\delta_{p,\infty}^{p}(r+1)
	&=\sup_{x,y\in X}\int_G
	\mathbb{E}|X_{r,gx}-X_{r,gy}|^p\,\mu(\mathrm{d}g)\leq\sup_{u,v\in X}\mathbb{E}|X_{r,u}-X_{r,v}|^p=\delta_{p,\infty}^{p}(r).
\end{aligned}
\]
Taking $p=1$ gives
$\delta_{1,\infty}(r+1)\leq\delta_{1,\infty}(r)$. Hence,
\begin{align}\label{eq-vv-10}
m^{1+\delta}\delta_{1,\infty}(m)
\leq C\sum_{r=\lceil m/2\rceil}^{m}r^{\delta}\delta_{1,\infty}(r)
\leq C.
\end{align}
By \Cref{lem-1}, we have
$\sum_{j=m-1}^{\infty}\delta_{1,\infty}(j+1)\leq m^{-\delta}\sum_{r\geq m}r^{\delta}\delta_{1,\infty}(r)\leq Cm^{-\delta}.$  Substituting the bounds for $(m-1)\delta_{1,\infty}(m)$
and $\sum_{j=m-1}^{\infty}\delta_{1,\infty}(j+1)$
into \eqref{eq-vv-9.1} yields \eqref{eq-ad-07}.

We next prove \eqref{eq-ad-08}. Let \(c_{0}>0\) be a constant provided by \Cref{lem-A1}(i) for which \eqref{eq-ab-05} holds.
 Without loss of generality, we assume that  $m\geq\max\{12,6/c_0\}$;
 otherwise, \eqref{eq-ad-08} follows from
 $\mathbb{E}L(g_1)<\infty$ and the uniform bounds
 for $\psi_m^{[a]}$ and $\psi^{[a]}$, after enlarging $C$. By \eqref{eq-ad-02} and the definition of $\psi^{[a]}$, we have
\begin{align}\label{eq-hh-1.1}
	\mathbb{E}\bigl[L(g_{k})
	\lvert \psi_{m}^{[a]}(U_{k})-\psi^{[a]}(W_{k})\rvert \bigr]\leq I_{1}+I_{2},
\end{align}
where 
\begin{align*}
	I_{1}&=\sum_{j=0}^{\lfloor m/2\rfloor} \E\big[ L(g_{k})
	\big|(P^{j}\phi^{[a]})
	(g_{k}\cdots  g_{k-h_{j}+1} \widetilde{W}_{k-h_{j}})-	(P^{j}\phi^{[a]})
	(W_{k})\big|\big],\\
	I_2
	&=\sum_{j=\lfloor m/2\rfloor+1}^{m-2}
	\mathbb{E}\Bigl[
	L(g_k)
	\bigl|
	(P^j\phi^{[a]})
	(g_k\cdots g_{k-h_j+1}
	\widetilde W_{k-h_j})
	-(P^j\phi^{[a]})(W_k)
	\bigr|
	\Bigr]\\
	&\quad+\sum_{j=m-1}^{\infty}
	\mathbb{E}\Bigl[
	L(g_k)
	\bigl|(P^j\phi^{[a]})(W_k)\bigr|
	\Bigr].
\end{align*}
We first bound \(I_{2}\). By \eqref{eq-ac-01} and \Cref{lem-1}, we have
\begin{align}\label{eq-hh-02}
	I_{2}&\leq C\E L(g_{k})\sum_{j=\lfloor m/2\rfloor+1}^{\infty}\delta_{1,\infty}(j+1)\leq Cm^{-\delta}.
\end{align}
It remains to  bound  $I_{1}$.  
Set
\[
\mathcal A_{h_{j}}
:=
\biggl\{
\rho\bigl(
g_{k-1}\cdots g_{k-h_{j}+1} W_{k-h_{j}},
g_{k-1}\cdots g_{k-h_{j}+1} \widetilde W_{k-h_{j}}
\bigr)
>
e^{-c_{0}(h_{j}-1)}
\biggr\}
\]
and
\[
\mathcal B_{h_{j}}
:=
\biggl\{
\log N(g_{k})
>
\frac{c_{0}(h_{j}-1)}{4}
\biggr\}.
\]
We further decompose \(I_{1}\) according to $1
\leq
\mathbf 1_{\mathcal A_{h_{j}}}
+
\mathbf 1_{\mathcal B_{h_{j}}}
+
\mathbf 1_{\mathcal A_{h_{j}}^{c}\cap\mathcal B_{h_{j}}^{c}}$ and write
\[
I_{1}\leq I_{11}+I_{12}+I_{13},
\]
where 
\begin{align*}
I_{11}&=\sum_{j=0}^{\lfloor m/2\rfloor} \E\big[ L(g_{k})
\big|(P^{j}\phi^{[a]})
(g_{k}\cdots  g_{k-h_{j}+1} \widetilde{W}_{k-h_{j}})-	(P^{j}\phi^{[a]})
(W_{k})\big| \mathbf 1_{\mathcal A_{h_{j}}}\big].\\
I_{12}&=\sum_{j=0}^{\lfloor m/2\rfloor} \E\big[ L(g_{k})
\big|(P^{j}\phi^{[a]})
(g_{k}\cdots  g_{k-h_{j}+1} \widetilde{W}_{k-h_{j}})-	(P^{j}\phi^{[a]})
(W_{k})\big| \mathbf 1_{\mathcal B_{h_{j}}}\big],\\
I_{13}&=\sum_{j=0}^{\lfloor m/2\rfloor} \E\big[ L(g_{k})
\big|(P^{j}\phi^{[a]})
(g_{k}\cdots  g_{k-h_{j}+1} \widetilde{W}_{k-h_{j}})-	(P^{j}\phi^{[a]})
(W_{k})\big| \mathbf 1_{\mathcal A_{h_{j}}^{c}\cap\mathcal B_{h_{j}}^{c}}\big].
\end{align*}
We estimate $I_{11}$, $I_{12}$ and $I_{13}$ separately.
Note that $\mathcal A_{h_j}$ is measurable with respect to
the $\sigma$-algebra
$
\sigma(
W_{k-h_j},
\widetilde W_{k-h_j},
g_{k-h_j+1},
\ldots,
g_{k-1}
),
$
which is independent of $g_k$.
For every $0\leq j\leq\lfloor m/2\rfloor$, we have
$h_j-1\geq m/2-2\geq m/3$.
Thus \eqref{eq-ab-05}, conditioned on the two initial states, gives
$\P(\mathcal A_{h_j})\leq C(h_j-1)^{-\delta}\leq Cm^{-\delta}$.
Combining the bound for $\P(\mathcal A_{h_j})$ with
\eqref{eq-ac-01} and \Cref{lem-1}, we obtain
\begin{align}\label{eq-hh-03}
	I_{11}&\leq C\sum_{j=0}^{\lfloor m/2\rfloor} \delta_{1,\infty}(j+1)
	 \E L(g_{k})
	 \P(\mathcal A_{h_j})\leq Cm^{-\delta}\sum_{j=0}^{\lfloor m/2\rfloor}
	\delta_{1,\infty}(j+1)\leq Cm^{-\delta}.
\end{align}

On \(\mathcal B_{h_{j}}\), since $h_{j}-1\geq m/3$ for $m\geq 12$, it follows that $\mathcal B_{h_{j}}
\subset
\bigl\{
L(g_{k})>c_{1}m
\bigr\},$ where $c_{1}=c_{0}/12.$  Combining the inclusion for $\mathcal B_{h_j}$ with
\eqref{eq-ac-01} and \Cref{lem-1}, we obtain
\begin{align}\label{eq-hh-04}
	I_{12}&\leq C\sum_{j=0}^{\lfloor m/2\rfloor} \delta_{1,\infty}(j+1)\cdot  \E\bigl[
	L(g_{k})
	\mathbf 1_{\{L(g_{k})>c_{1}m\}}
	\bigr]\leq Cm^{-\delta}.
\end{align}

Finally, on
\(\mathcal A_{h_{j}}^{c}\cap\mathcal B_{h_{j}}^{c}\), by \eqref{eq-ab-07}, we have
\[
\begin{aligned}
	\rho\bigl(
	g_{k}\cdots g_{k-h_{j}+1} W_{k-h_{j}},
	g_{k}\cdots g_{k-h_{j}+1} \widetilde W_{k-h_{j}}
	\bigr)&\leq
	\exp\bigl(
	2\log N(g_{k})-c_{0}(h_{j}-1)
	\bigr)\\
	&\leq
	e^{-c_{0}(h_{j}-1)/2}\leq
	e^{-c_{2}m},
\end{aligned}
\]
where $c_2=c_0/6$.
 Hence, by the definition of
\(\omega_{j,a}\) in \Cref{lem-absmod}, we have
\[
|(P^{j}\phi^{[a]})
(g_{k}\cdots  g_{k-h_{j}+1} \widetilde{W}_{k-h_{j}})-	(P^{j}\phi^{[a]})
(W_{k})|\leq
\omega_{j,a}(c_{2}m)
\]
on \(\mathcal A_{h_{j}}^{c}\cap\mathcal B_{h_{j}}^{c}\). Therefore, we have
\begin{align}\label{eq-hh-05}
I_{13}&\leq  \E L(g_1)
 \sum_{j=0}^{\lfloor m/2\rfloor} \omega_{j,a}(c_{2}m)\leq Cm^{-\delta},
\end{align}
where the last inequality follows from \eqref{eq-ac-04}. Substituting \eqref{eq-hh-02}, \eqref{eq-hh-03}, \eqref{eq-hh-04} and \eqref{eq-hh-05} into \eqref{eq-hh-1.1}, we obtain   \eqref{eq-ad-08}. 
\end{proof}

Let
$s_{a,m}^{2}=\mathbb{E}[(Y^{[a]}_{m,m})^{2}]$. 
The following lemma provides a quantitative approximation of
\(s_{a,m}^{2}\) by \(s^{2}\), with an error of order
\(a^{-\delta}+m^{-\delta}\).
\begin{lemma}\label{prop-variance}
   There exists a positive constant $C$ such that
\begin{equation}\label{eq-ad-09}
\lvert s_{a,m}^{2}-s^{2}\rvert 
\leq C\bigl(a^{-\delta}+m^{-\delta}\bigr).
\end{equation}
In particular, $s_{m,m}^{2}$ is bounded above and bounded away
from zero for every sufficiently large $m$.
\end{lemma}

\begin{proof}
	By the triangle inequality and \Cref{lem-clipvariance}, we have
	\begin{align}\label{eq-ww-05}
		\begin{aligned}
		| s_{a,m}^{2}-s^{2}| &\leq \big|s_{a,m}^{2}-\mathbb{E}\bigl[\sigma_{0}^{[a]}(g_1,W_{0})^{2}\bigr]\big|+\big| \mathbb{E}\bigl[\sigma_{0}^{[a]}(g_1,W_{0})^{2}\bigr]-s^{2}\big|\\
		&\leq \big|s_{a,m}^{2}-\mathbb{E}\bigl[\sigma_{0}^{[a]}(g_1,W_{0})^{2}\bigr]\big|+C a^{-\delta}.
		\end{aligned}
	\end{align}
Fix $k\geq m$. To estimate
$|s_{a,m}^{2}-\mathbb{E}[\sigma_{0}^{[a]}(g_1,W_{0})^{2}]|$,
we first expand the second moments of
$Y_{k,m}^{[a]}$ and $\sigma_{0}^{[a]}(g_k,W_{k-1})$
using the definitions. By
	\eqref{eq-ad-01}, we have
	\begin{align}\label{eq-vv-20}
		\begin{aligned}
					\mathbb{E}\bigl[(Y_{k,m}^{[a]})^{2}\bigr]
			={}&\mathbb{E}\bigl[(X^{[a]}_{k,m})^{2}\bigr]
			+2\mathbb{E}\bigl[X^{[a]}_{k,m}\psi_{m}^{[a]}(U_{k})\bigr]
			-2\mathbb{E}\bigl[X^{[a]}_{k,m}\psi_{m}^{[a]}(U_{k-1})\bigr]\\
			&+\mathbb{E}\bigl[\psi_{m}^{[a]}(U_{k})^{2}\bigr]
			+\mathbb{E}\bigl[\psi_{m}^{[a]}(U_{k-1})^{2}\bigr]
			-2\mathbb{E}\bigl[\psi_{m}^{[a]}(U_{k})\psi_{m}^{[a]}(U_{k-1})\bigr].
		\end{aligned}
	\end{align}
	For $\mathbb{E}\bigl[X^{[a]}_{k,m}\psi_{m}^{[a]}(U_{k-1})\bigr]$, by the definitions of $X_{k,m}^{[a]}$ and $\phi_{m}^{[a]}$, we have 
	\begin{align*}
		\E[X_{k,m}^{[a]}\mid U_{k-1}]=\E\Big[\int H^{[a]}(g_{k},  g_{k-1}\cdots g_{k-m+1}x)\nu(dx) \mid U_{k-1}\Big]=\phi_{m}^{[a]}(U_{k-1}).
	\end{align*}
	Consequently, we have
	\begin{align}\label{eq-vv-20a}
		\begin{aligned}
				\mathbb{E}\bigl[X^{[a]}_{k,m}\psi_{m}^{[a]}(U_{k-1})\bigr]
		&=\mathbb{E}\bigl[\phi^{[a]}_{m}(U_{k-1})\psi_{m}^{[a]}(U_{k-1})\bigr]\\
		&=\mathbb{E}\bigl[\psi_{m}^{[a]}(U_{k-1})^{2}\bigr]
		-\mathbb{E}\bigl[(P_{m}\psi_m^{[a]})(U_{k-1})
		\psi_{m}^{[a]}(U_{k-1})\bigr],
		\end{aligned}
	\end{align}
	where the second identity uses the fact that
	$\phi^{[a]}_{m}=(I-P_{m})\psi_{m}^{[a]}$.  Similarly, conditioning on \(U_{k-1}\) and using the definition
	of \(P_{m}\), we have
\begin{align}\label{eq-vv-20b}
	\begin{aligned}
	\mathbb{E}\bigl[
	\psi_{m}^{[a]}(U_{k})\psi_{m}^{[a]}(U_{k-1})
	\bigr]
	&=
	\mathbb{E}\Bigl[
	\mathbb{E}\bigl[\psi_{m}^{[a]}(U_{k})\mid U_{k-1}\bigr]
	\psi_{m}^{[a]}(U_{k-1})
	\Bigr]\\
	&=
	\mathbb{E}\bigl[
	(P_{m}\psi_{m}^{[a]})(U_{k-1})
	\psi_{m}^{[a]}(U_{k-1})
	\bigr].
\end{aligned}
\end{align}
By stationarity, we have
$
	\mathbb{E}\psi_{m}^{[a]}(U_{k})^{2}
	=\mathbb{E}\psi_{m}^{[a]}(U_{k-1})^{2}.
$
Substituting \eqref{eq-vv-20a} and \eqref{eq-vv-20b}
into \eqref{eq-vv-20} and using stationarity, we obtain
\begin{align}\label{eq-vv-21}
	\mathbb{E}\bigl[(Y_{k,m}^{[a]})^{2}\bigr]
=\mathbb{E}\bigl[(X^{[a]}_{k,m})^{2}\bigr]
+2\mathbb{E}\bigl[X^{[a]}_{k,m}\psi_{m}^{[a]}(U_{k})\bigr].
\end{align}
	For $\mathbb{E}\bigl[\sigma_{0}^{[a]}(g_1,W_{0})^{2}]$,  by the definition of $\sigma_{0}^{[a]}(g_{k}, W_{k-1})$, we have
	\begin{align*}
		\mathbb{E}\bigl[\sigma_{0}^{[a]}(g_{k},W_{k-1})^{2}\bigr]
		={}&
		\mathbb{E}\bigl[H^{[a]}(g_{k},W_{k-1})^{2}\bigr]
		+\mathbb{E}\bigl[\psi^{[a]}(W_{k})^{2}\bigr]
		+\mathbb{E}\bigl[\psi^{[a]}(W_{k-1})^{2}\bigr]\\
		&+
		2\mathbb{E}\bigl[
		H^{[a]}(g_{k},W_{k-1})
		\psi^{[a]}(W_{k})
		\bigr]-
		2\mathbb{E}\bigl[
		H^{[a]}(g_{k},W_{k-1})
		\psi^{[a]}(W_{k-1})
		\bigr]\\
		&-
		2\mathbb{E}\bigl[
		\psi^{[a]}(W_{k})
		\psi^{[a]}(W_{k-1})
		\bigr].
	\end{align*}
By	conditioning on 
	$W_{k-1}$  and using 	$(I-P)\psi^{[a]}=\phi^{[a]}$, we have
	\begin{align*}
		\mathbb{E}\bigl[H^{[a]}(g_{k},W_{k-1})\psi^{[a]}(W_{k-1})\bigr]
		&=
		\mathbb{E}\bigl[\phi^{[a]}(W_{k-1})\psi^{[a]}(W_{k-1})\bigr]\\
		&=
		\mathbb{E}\bigl[\psi^{[a]}(W_{k-1})^{2}\bigr]
		-\mathbb{E}\bigl[(P\psi^{[a]})(W_{k-1})\psi^{[a]}(W_{k-1})\bigr],
	\end{align*}
	and 
	\[
	\mathbb{E}\bigl[\psi^{[a]}(W_{k})\psi^{[a]}(W_{k-1})\bigr]=
	\mathbb{E}\bigl[(P\psi^{[a]})(W_{k-1})\psi^{[a]}(W_{k-1})\bigr].
	\]
Also, 
 stationarity gives
$
	\mathbb{E}\psi^{[a]}(W_{k})^{2}=\mathbb{E}\psi^{[a]}(W_{k-1})^{2}.
$
Therefore, 
\begin{align}\label{eq-vv-22}
	\mathbb{E}\bigl[\sigma_{0}^{[a]}(g_{k},W_{k-1})^{2}\bigr]
	=\mathbb{E}\bigl[H^{[a]}(g_{k},W_{k-1})^{2}\bigr]
	+2\mathbb{E}\bigl[H^{[a]}(g_{k},W_{k-1})\psi^{[a]}(W_{k})\bigr].
\end{align}
Combining \eqref{eq-vv-21} and \eqref{eq-vv-22}, and using stationarity, we have
\begin{align}
	\big|s_{a,m}^{2}-\mathbb{E}\bigl[\sigma_{0}^{[a]}(g_1,W_{0})^{2}\bigr]\big|\leq I_{1}+I_{2},
\end{align}
where 
\begin{align*}
	I_{1}&= \big| \mathbb{E}\bigl[(X^{[a]}_{k,m})^{2}\bigr] -\mathbb{E}\bigl[H^{[a]}(g_{k},W_{k-1})^{2}\bigr]\big|,\\
	I_{2}&=2\big| \mathbb{E}\bigl[X^{[a]}_{k,m}\psi_{m}^{[a]}(U_{k})\bigr]-\mathbb{E}\bigl[H^{[a]}(g_{k},W_{k-1})\psi^{[a]}(W_{k})\bigr] \big|.
\end{align*}
For $I_{1}$, the conditional projection identity
$
X_{k,m}^{[a]}=\E[H^{[a]}(g_{k},W_{k-1})\mid\mathscr{F}_{k,m}]
$
gives
\[
\mathbb{E}\bigl[X^{[a]}_{k,m}\bigl(H^{[a]}(g_{k},W_{k-1})-X^{[a]}_{k,m}\bigr)\bigr]=0.
\]
Hence, 
\begin{align*}
	I_{1}&\leq  \mathbb{E}\lvert X^{[a]}_{k,m}-H^{[a]}(g_{k},W_{k-1})\rvert ^{2}\\
	&\leq \mathbb{E}\big\{\lvert X^{[a]}_{k,m}-H^{[a]}(g_{k},W_{k-1})\rvert ^{2}\mathbf{1}_{\{L(g_{k})\leq m\}}\big\}
+\mathbb{E}\big\{\lvert X^{[a]}_{k,m}-H^{[a]}(g_{k},W_{k-1})\rvert ^{2}\mathbf{1}_{\{L(g_{k})> m\}}\big\}.
\end{align*}
Using
$
W_{k-1}
=g_{k-1}\cdots
g_{k-m+1}W_{k-m},
$
the definition of $X_{k,m}^{[a]}$ gives
\[
\begin{aligned}
	H^{[a]}(g_{k},W_{k-1})-X_{k,m}^{[a]}=
	\int_{X}\Bigl\{
	H^{[a]}(g_{k},W_{k-1})
	-
	H^{[a]}\bigl(
	g_{k},
	g_{k-1}\cdots g_{k-m+1}y
	\bigr)
	\Bigr\}\,\nu(\mathrm{d}y).
\end{aligned}
\]
Since $\tau^{[a]}$ is
$1$-Lipschitz and $W_{k-m}$ is independent of
$g_{k-m+1},\ldots,g_{k}$,
the definition of $\delta_{1,\infty}(m)$ yields
\begin{equation}\label{eq-ai-03}
	\mathbb{E}\bigl|
	H^{[a]}(g_{k},W_{k-1})
	-X_{k,m}^{[a]}
	\bigr|
	\leq\delta_{1,\infty}(m).
\end{equation}
Combining this with the inequality $|X_{k,m}^{[a]}|+|H^{[a]}(g_{k},W_{k-1})|\leq CL(g_{k})$ and \eqref{eq-vv-10}, we have
\begin{align}\label{eq-vv-23}
	I_{1}\leq Cm \delta_{1,\infty}(m)
	+C\mathbb{E}\bigl[L(g_{k})^{2}
	\mathbf{1}_{\{L(g_{k})>m\}}\bigr]\leq Cm^{-\delta}.
\end{align}

For $I_{2}$, noting that $\psi_{m}^{[a]}(U_{k})$ is
$\mathscr{F}_{k,m}$-measurable, it follows that
\[
\mathbb{E}\bigl[X^{[a]}_{k,m}\psi_{m}^{[a]}(U_{k})\bigr]
=\mathbb{E}\bigl[H^{[a]}(g_{k},W_{k-1})\psi_{m}^{[a]}(U_{k})\bigr].
\]
Using the above identity together with the inequality $\lvert H^{[a]}(g_{k},W_{k-1})\rvert \leq C L(g_{k})$, and then applying \Cref{lem-weightedprojection}, we obtain
\begin{align}\label{eq-vv-24}
I_{2}\leq C\mathbb{E}\bigl[L(g_{k})
\lvert \psi_{m}^{[a]}(U_{k})-\psi^{[a]}(W_{k})\rvert \bigr]
\leq C m^{-\delta}.
\end{align}
Combining  \eqref{eq-ww-05}, \eqref{eq-vv-23} and \eqref{eq-vv-24},  we  complete  the proof of \eqref{eq-ad-09} and hence the lemma. 
\end{proof}

 For notational
simplicity, when no confusion can arise,  we simply write $Y_{k}=Y_{k,m}^{[a]}$ and  $X_{k,m}=X_{k,m}^{[a]}.$
For $m\geq2$ and $h\in G=\mathrm{GL}(d,\R)$,  let
$T_{h}:G^{m-1}\to G^{m-1}$ be the shift map given by
\[
T_{h}(u_{1},\ldots,u_{m-1})
:=(u_{2},\ldots,u_{m-1},h).
\]
For $u=(u_{1},\ldots,u_{m-1})\in G^{m-1}$, write
\[
x_{a,m}(u,h)
:=x_{a,m}(u_{1},\ldots,u_{m-1},h).
\]
With this notation, define
$F_{a,m}:G^{m-1}\times G\to\mathbb{R}$ by
\[
F_{a,m}(u,h)
:=x_{a,m}(u,h)
+\psi_{m}^{[a]}(T_{h}u)-\psi_{m}^{[a]}(u).
\]

Recall that $(g_{k})_{k\geq1}$ is a sequence of independent
random matrices with common law $\mu$.
For every deterministic initial state $u\in G^{m-1}$, define
\begin{align}\label{eq-defi-Uu}
	U_{0}(u):=u,\qquad
	U_{r}(u):=T_{g_{r}}\bigl(U_{r-1}(u)\bigr),
	\qquad r\geq1.
\end{align}
For every $r\geq1$, $u\in G^{m-1}$ and $h\in G$, define
\[
\begin{aligned}
	X_{r}(u;h)&:=x_{a,m}\bigl(U_{r-1}(u),h\bigr),\quad
	Y_{r}(u;h):=F_{a,m}\bigl(U_{r-1}(u),h\bigr).
\end{aligned}
\]
 For fixed $u$ and $h$, these random
variables depend only on $g_{1},\ldots,g_{r-1}$.
The corresponding random variables at step $r$ are defined by
\begin{align}\label{eq-defi-Yr}
	X_{r}(u):=X_{r}(u;g_{r}),\qquad
	Y_{r}(u):=Y_{r}(u;g_{r}).
\end{align}
Thus
\begin{align}\label{eq-ww-06}
	Y_{r}(u)
	=X_{r}(u)+\psi_{m}^{[a]}\bigl(U_{r}(u)\bigr)
	-\psi_{m}^{[a]}\bigl(U_{r-1}(u)\bigr).
\end{align}

Fix two deterministic initial states $u,v\in G^{m-1}$.
For each fixed $r\geq1$, write $g:=g_{r}$ and let $g'$ have
law $\mu$, independently of the entire sequence
$(g_{j})_{j\geq1}$. Define
\[
\Gamma_{r}(u,v)
:=Y_{r}(u;g)-Y_{r}(v;g)
-Y_{r}(u;g')+Y_{r}(v;g').
\]

After $m-1$ steps, all coordinates of the initial states
have been replaced.  Hence, for every $r\geq m$,
\begin{align}\label{eq-ww-09}
	U_{r-1}(u)=U_{r-1}(v).
\end{align}
Consequently,
\[
Y_{r}(u;h)=Y_{r}(v;h)
\quad\text{for every }h\in G,
\qquad
\Gamma_{r}(u,v)=0.
\]

To prove the concentration inequality and \Cref{thm-main},
we need moment bounds on the changes in $Y_{r}$ and $Y_{r}^{2}$
caused by changing the initial state.
\begin{lemma}
\label{prop-budgets}
Let $\ell_{a}(g)=1+\min\{L(g), a\}$, where we recall that
$L(g)=1+|\lambda_{\mu}|+\log N(g)$. Then
there exist nonnegative deterministic coefficients
$\beta_{r,1},\beta_{r,2},\beta_{r,3}$, all zero for $r\geq m$, such
that
\begin{align}
\sup_{u,v}\mathbb{E}\lvert Y_{r}(u)-Y_{r}(v)\rvert ^{2}
&\leq\beta_{r,1},\label{eq-ae-02}\\
\sup_{u,v}\mathbb{E}\lvert (Y_{r}(u))^{2}-(Y_{r}(v))^{2}\rvert 
&\leq \beta_{r,2},\label{eq-ae-03}\\
\sup_{u,v}\mathbb{E}\bigl[
 (\ell_{a}(g)+\ell_{a}(g'))\lvert \Gamma_{r}(u,v)\rvert 
\bigr]
&\leq\beta_{r,3},\label{eq-ae-04}
\end{align}
and
\begin{equation}\label{eq-ae-05}
\sum_{r=1}^{\infty}
 \bigl(\beta_{r,1}+\beta_{r,2}+\beta_{r,3}\bigr)
\leq C\bigl(1+a^{1-\delta}+m^{1-\delta}\bigr),
\end{equation}
where $C>0$ is a constant depending only on $\mu$, $\delta$,
$d$, and the fixed norm.
\end{lemma}

\begin{proof}
	
	We first prove  the estimate required for \eqref{eq-ae-02}. By \eqref{eq-ww-06} and the inequality $(x+y+z)^{2}\leq 3(x^{2}+y^{2}+z^{2})$, we have
	\begin{align}\label{eq-ae-06}
		\sup_{u,v}\mathbb{E}\lvert Y_{r}(u)-Y_{r}(v)\rvert ^{2}\leq 3I_{r,1}+3I_{r,2}+3I_{r,3},
	\end{align}
	where 
	\begin{align*}
		I_{r,1}&=\sup_{u,v}\E|X_{r}(u)-X_{r}(v)|^{2},\\
		I_{r,2}&=\sup_{u,v}\E\big|\psi_{m}^{[a]}\bigl(U_{r}(u)\bigr)-\psi_{m}^{[a]}\bigl(U_{r}(v)\bigr)   \big|^{2},\\
		I_{r,3}&=\sup_{u,v}\E\big|\psi_{m}^{[a]}\bigl(U_{r-1}(u)\bigr)-\psi_{m}^{[a]}\bigl(U_{r-1}(v)\bigr)   \big|^{2}.
	\end{align*}
We first estimate $I_{r,1}$. For $1\leq r\leq m-1$,
we have $X_r(u)=x_{a,m}(u_r,\ldots,u_{m-1},g_1,\ldots,g_r)$
by definition. For $r\geq m$, the difference
$X_r(u)-X_r(v)$ vanishes.  By the definition of
	\(\delta_{1,\infty}(r)\) and the $1$-Lipschitz property of the
	censoring map, we have
\begin{align}\label{eq-wa-01}
	\sup_{u,v}
\mathbb{E}\lvert X_{r}(u)-X_{r}(v)\rvert
\leq
\delta_{1,\infty}(r).
\end{align}
	Moreover, we have
$
	\lvert X_{r}(u)-X_{r}(v)\rvert
	\leq
	2\min\{L(g_{r}),a\}.
$
Using \eqref{eq-wa-01}, we obtain
\begin{align}	\label{eq-raw-Ir1}
	\begin{aligned}
		I_{r,1}
	&=\sup_{u,v}\mathbb{E}\Bigl[
	|X_{r}(u)-X_{r}(v)|^{2}
	\bigl(
	\mathbf{1}_{\{L(g_{r})\leq r\}}
	+\mathbf{1}_{\{L(g_{r})>r\}}
	\bigr)\Bigr]\\
	&\leq
	2\min\{r,a\}\sup_{u,v}\mathbb{E}|X_{r}(u)-X_{r}(v)|
	+4\mathbb{E}\bigl[
	\min\{L(g_{r}),a\}^{2}
	\mathbf{1}_{\{L(g_{r})>r\}}
	\bigr]\\
	&\leq
	C\min\{r,a\}\delta_{1,\infty}(r)
	+C\mathbb{E}\bigl[
	\min\{L(g_{r}),a\}^{2}
	\mathbf{1}_{\{L(g_{r})>r\}}
	\bigr].
	\end{aligned}
\end{align}
Since $\min\{r,a\}\leq a^{1-\delta}r^{\delta}$ for $r,a\geq1$,
\Cref{lem-1} yields
	\[
	\sum_{r=1}^{\infty}
	\min\{r,a\}\delta_{1,\infty}(r)
	\leq
	a^{1-\delta}
	\sum_{r=1}^{\infty}
	r^{\delta}\delta_{1,\infty}(r)
	\leq
	C a^{1-\delta}.
	\]
	For the second term in \eqref{eq-raw-Ir1},  by Tonelli's theorem, we have
	\begin{align*}
		\sum_{r=1}^{\infty}
		\mathbb{E}\bigl[
		\min\{L(g_{r}),a\}^{2}
		\mathbf{1}_{\{L(g_{r})>r\}}
		\bigr]&=
		\mathbb{E}\Bigl[
		\min\{L(g_{1}),a\}^{2} 	\sum_{r=1}^{\infty}
		\mathbf{1}_{\{L(g_{1})>r\}}
		\Bigr]\\
		&\leq
		C\mathbb{E}\bigl[
		L(g_{1})\min\{L(g_{1}),a\}^{2}
		\bigr]\\
		&\leq
		C a^{1-\delta}
		\mathbb{E}L(g_{1})^{2+\delta}
		\leq
		C a^{1-\delta}.
	\end{align*}
	Consequently, we have
\begin{align}\label{eq-wa-01a}
		\sum_{r=1}^{\infty}I_{r,1}
	\leq
	C a^{1-\delta}.
\end{align}
We next estimate \(I_{r,2}\). 
By \eqref{eq-ad-02}, for \(1\leq r\leq m-2\), we have
\begin{equation}\label{eq-psi-initial-state-expansion}
	\begin{aligned}
		\psi_{m}^{[a]}\bigl(U_{r}(u)\bigr)
		={}&
		\sum_{j=0}^{m-r-2}
		\int
		(P^{j}\phi^{[a]})
		\bigl(
		g_{r}\cdots g_{1}
		u_{m-1}\cdots u_{r+j+1}x
		\bigr)\,\nu(\mathrm{d}x)\\
		&+
		\sum_{j=m-r-1}^{m-2}
		\int
		(P^{j}\phi^{[a]})
		\bigl(
		g_{r}\cdots g_{r-m+j+2}x
		\bigr)\,\nu(\mathrm{d}x).
	\end{aligned}
\end{equation}
Combining the expansion of $\psi_{m}^{[a]}(U_{r}(u))$
with \eqref{eq-ww-08}, we obtain
\[
\sup_{u,v}
\mathbb{E}\bigl|
\psi_{m}^{[a]}\bigl(U_{r}(u)\bigr)
-
\psi_{m}^{[a]}\bigl(U_{r}(v)\bigr)
\bigr|
\leq
\sum_{j=0}^{m-r-2}
\delta_{1,\infty}(r+j+1).
\]
Since \(\psi_{m}^{[a]}\) is uniformly bounded, we have
$
|
\psi_{m}^{[a]}(U_{r}(u))
-
\psi_{m}^{[a]}(U_{r}(v))
|^{2}
\leq
C
|
\psi_{m}^{[a]}(U_{r}(u))
-
\psi_{m}^{[a]}(U_{r}(v))
|.
$
Therefore,
\[
I_{r,2}
\leq
C\sum_{j=0}^{m-r-2}
\delta_{1,\infty}(r+j+1),
\qquad
1\leq r\leq m-2.
\]
Moreover, \(I_{r,2}=0\) for \(r\geq m-1\) by \eqref{eq-ww-09}. Hence,
\begin{align}\label{eq-ww-10}
	\begin{aligned}
		\sum_{r=1}^{\infty}I_{r,2}
	&\leq
	C\sum_{r=1}^{m-2}
	\sum_{j=0}^{m-r-2}
	\delta_{1,\infty}(r+j+1)\\
	&\leq
	C\sum_{t=1}^{\infty}
	\min\{t,m\}\delta_{1,\infty}(t)\\
	&\leq
	C m^{1-\delta}
	\sum_{t=1}^{\infty}
	t^{\delta}\delta_{1,\infty}(t)
	\leq
	C m^{1-\delta}.
	\end{aligned}
\end{align}
By an argument similar to that leading to \eqref{eq-ww-10},
we have
\[
I_{r,3}
\leq
C\sum_{j=0}^{m-r-1}\delta_{1,\infty}(r+j),
\qquad
1\leq r\leq m-1,
\]
and 
 \(I_{r,3}=0\) for \(r\geq m\). Thus,
\begin{align}\label{eq-ww-10a}
	\sum_{r=1}^{\infty}I_{r,3}
	&\leq
	C\sum_{r=1}^{m-1}
	\sum_{j=0}^{m-r-1}
	\delta_{1,\infty}(r+j)
	\leq
	C\sum_{t=1}^{\infty}
	\min\{t,m\}\delta_{1,\infty}(t)
	\leq
	C m^{1-\delta}.
\end{align}
Combining  \eqref{eq-wa-01a}, \eqref{eq-ww-10}  and \eqref{eq-ww-10a}, we obtain the desired bound in \eqref{eq-ae-02}.

We next prove  the estimate required for \eqref{eq-ae-03}. By the definition of \(Y_{r}\)
and the uniform boundedness of \(\psi_{m}^{[a]}\), we have
$
\lvert Y_{r}(u)\rvert+\lvert Y_{r}(v)\rvert
\leq
C\ell_{a}(g_{r}).
$
Hence,
\begin{align*}
	\bigl\lvert
	(Y_{r}(u))^{2}-(Y_{r}(v))^{2}
	\bigr\rvert
	&=
	\lvert Y_{r}(u)-Y_{r}(v)\rvert
	\lvert Y_{r}(u)+Y_{r}(v)\rvert\leq
	C\ell_{a}(g_{r})
	\lvert Y_{r}(u)-Y_{r}(v)\rvert.
\end{align*}
Using \eqref{eq-ww-06} and the triangle inequality, we have
\[
\sup_{u,v}
\mathbb{E}\bigl\lvert
(Y_{r}(u))^{2}-(Y_{r}(v))^{2}
\bigr\rvert
\leq
C\bigl(I_{r,4}+I_{r,5}+I_{r,6}\bigr),
\]
where
\begin{align*}
	I_{r,4}
	&:=
	\sup_{u,v}
	\mathbb{E}\bigl[
	\ell_{a}(g_{r})
	\lvert X_{r}(u)-X_{r}(v)\rvert
	\bigr],\\
	I_{r,5}
	&:=
	\sup_{u,v}
	\mathbb{E}\bigl[
	\ell_{a}(g_{r})
	\bigl|
	\psi_{m}^{[a]}\bigl(U_{r}(u)\bigr)
	-
	\psi_{m}^{[a]}\bigl(U_{r}(v)\bigr)
	\bigr|
	\bigr],\\
	I_{r,6}
	&:=
	\sup_{u,v}
	\mathbb{E}\bigl[
	\ell_{a}(g_{r})
	\bigl|
	\psi_{m}^{[a]}\bigl(U_{r-1}(u)\bigr)
	-
	\psi_{m}^{[a]}\bigl(U_{r-1}(v)\bigr)
	\bigr|
	\bigr].
\end{align*}
We estimate \(I_{r,4}\), \(I_{r,6}\), and \(I_{r,5}\) one by one. 
We first estimate $I_{r,4}$.
Since $\ell_{a}(g_{r})\leq 2\min\{L(g_{r}),a\}$,
the same argument as for $I_{r,1}$ yields
\begin{align}\label{eq-ww-10b}
	\begin{aligned}
	\sum_{r=1}^{\infty}I_{r,4}
&\leq
C\sum_{r=1}^{\infty}\min\{r,a\}\delta_{1,\infty}(r)
+C\sum_{r=1}^{\infty}\mathbb{E}\bigl[
\min\{L(g_{r}),a\}^{2}
\mathbf{1}_{\{L(g_{r})>r\}}
\bigr]\\
&\leq
Ca^{1-\delta}\Bigl(
\sum_{r=1}^{\infty}r^{\delta}\delta_{1,\infty}(r)
+\mathbb{E}[L(g_{1})^{2+\delta}]
\Bigr)
\leq Ca^{1-\delta}.
	\end{aligned}
\end{align}
We next estimate \(I_{r,6}\). Since
$
\psi_{m}^{[a]}\bigl(U_{r-1}(u)\bigr)
-
\psi_{m}^{[a]}\bigl(U_{r-1}(v)\bigr)
$
depends only on \(g_{1},\ldots,g_{r-1}\), it is independent of
\(g_{r}\). Hence,
\begin{align*}
	I_{r,6}
	&=
	\mathbb{E}\ell_{a}(g_{r})
	\sup_{u,v}
	\mathbb{E}\bigl|
	\psi_{m}^{[a]}\bigl(U_{r-1}(u)\bigr)
	-
	\psi_{m}^{[a]}\bigl(U_{r-1}(v)\bigr)
	\bigr|\\
	&\leq
	C\sum_{j=0}^{m-r-1}
	\delta_{1,\infty}(r+j),
	\qquad
	1\leq r\leq m-1.
\end{align*}
Moreover, \(I_{r,6}=0\) for \(r\geq m\). Therefore,
\begin{align}\label{eq-ww-10c}
	\sum_{r=1}^{\infty}I_{r,6}
	&\leq
	C\sum_{r=1}^{m-1}
	\sum_{j=0}^{m-r-1}
	\delta_{1,\infty}(r+j)
	\leq
	C\sum_{t=1}^{\infty}
	\min\{t,m\}\delta_{1,\infty}(t)
	\leq
	C m^{1-\delta}.
\end{align}
It remains to estimate \(I_{r,5}\). For \(1\leq r\leq m-2\)
and \(0\leq j\leq m-r-2\), set
\begin{align*}
	D_{r,j}(u,v,x)
	:=
	\bigl|
	&(P^{j}\phi^{[a]})
	\bigl(
	g_{r}\cdots g_{1}
	u_{m-1}\cdots u_{r+j+1}x
	\bigr)-
	(P^{j}\phi^{[a]})
	\bigl(
	g_{r}\cdots g_{1}
	v_{m-1}\cdots v_{r+j+1}x
	\bigr)
	\bigr|.
\end{align*}
By \eqref{eq-psi-initial-state-expansion} and the triangle
inequality, we have
\[
I_{r,5}
\leq
\sum_{j=0}^{m-r-2}
\sup_{u,v}
\int
\mathbb{E}\bigl[
\ell_{a}(g_{r})D_{r,j}(u,v,x)
\bigr]
\,\nu(\mathrm{d}{x}).
\]
Let \(c_{0}>0\) be a constant provided by
Lemma~\ref{lem-A1}(i) for which \eqref{eq-ab-05} holds, and put
$
r_{0}
:=
\max\{
2,\lceil 4/c_{0}\rceil
\}.
$
We first consider \(1\leq r<r_{0}\). Since
$
D_{r,j}(u,v,x)
\leq
2\delta_{1,\infty}(j+1)
$
and
$
\mathbb{E}\ell_{a}(g_{r})\leq C,
$
we obtain
$
I_{r,5}
\leq
C\sum_{j=0}^{m-r-2}\delta_{1,\infty}(j+1)
\leq C.
$
Since \(r_{0}\) is independent of \(a\) and \(m\),
$
\sum_{1\leq r<r_{0}}I_{r,5}\leq C.
$
We now assume \(r\geq r_{0}\). For
\(0\leq j\leq m-r-2\), define
\[
\begin{aligned}
	\mathcal A_{r,j}(u,v,x)
	:=\biggl\{
	&\rho\bigl(
	g_{r-1}\cdots g_{1}
	u_{m-1}\cdots u_{r+j+1}x, g_{r-1}\cdots g_{1}
	v_{m-1}\cdots v_{r+j+1}x
	\bigr)>
	e^{-c_{0}(r-1)}
	\biggr\}
\end{aligned}
\]
and
\[
\mathcal B_{r}
:=
\Bigl\{
\log N(g_{r})
>
c_{0}(r-1)/4
\Bigr\}.
\]
Using
$
1
\leq
\mathbf{1}_{\mathcal A_{r,j}}
+
\mathbf{1}_{\mathcal B_{r}}
+
\mathbf{1}_{\mathcal A_{r,j}^{c}\cap\mathcal B_{r}^{c}},
$
we split the estimate of \(I_{r,5}\) into the three contributions
corresponding to
\(\mathcal A_{r,j}\), \(\mathcal B_{r}\), and
\(\mathcal A_{r,j}^{c}\cap\mathcal B_{r}^{c}\), respectively,
and estimate them separately.
On \(\mathcal A_{r,j}(u,v,x)\),  we have
\begin{align*}
	&\mathbb{E}\bigl[
	\ell_{a}(g_{r})
	D_{r,j}(u,v,x)
	\mathbf{1}_{\mathcal A_{r,j}(u,v,x)}
	\bigr]\leq
	C\delta_{1,\infty}(j+1)\P(\mathcal A_{r,j}(u,v,x)).
\end{align*}
Set
$p_t:=\sup_{x,y\in X}\P(\rho(A_t x,A_t y)>e^{-c_0t})$.
The two initial states in $\mathcal A_{r,j}(u,v,x)$ are deterministic,
so $\P(\mathcal A_{r,j}(u,v,x))\leq p_{r-1}$ uniformly in
$u,v,x,j$ whenever the event is defined. Consequently,
\Cref{lem-A1}(i) gives
\begin{align*}
	&\sum_{r=r_{0}}^{m-2}
	\sum_{j=0}^{m-r-2}
	\delta_{1,\infty}(j+1)
	\sup_{u,v,x}\P(\mathcal A_{r,j}(u,v,x))\leq
	\sum_{j=0}^{\infty}\delta_{1,\infty}(j+1)
	\sum_{r=2}^{\infty}p_{r-1}
	\leq C.
\end{align*}
This bound remains valid after taking the suprema and the
$\nu$-integral in the definition of $I_{r,5}$.

On \(\mathcal B_{r}\), we have
$
	\mathbb{E}[
	\ell_{a}(g_{r})
	D_{r,j}(u,v,x)
	\mathbf{1}_{\mathcal B_{r}}
	]\leq
	2\delta_{1,\infty}(j+1)
	\mathbb{E}[
	\ell_{a}(g_{r})
	\mathbf{1}_{\mathcal B_{r}}
	].
$
Combining this with \Cref{lem-1}, we have
\begin{align}
	\sum_{r=r_{0}}^{m-2}
	\sum_{j=0}^{m-r-2} \delta_{1,\infty}(j+1)
	\mathbb{E}\bigl[
	\ell_{a}(g_{r})
	\mathbf{1}_{\mathcal B_{r}}
	\bigr]\leq C	\sum_{r=r_{0}}^{\infty}
	\mathbb{E}\bigl[
	\ell_{a}(g_{r})
	\mathbf{1}_{\mathcal B_{r}}
	\bigr]\leq C\E L(g_{1})^{2}\leq C.
	\end{align}

Finally, consider
\(\mathcal A_{r,j}^{c}(u,v,x)\cap\mathcal B_{r}^{c}\). In this case, 
 by \eqref{eq-ab-07} in Lemma~\ref{lem-A1}(ii), we have
\begin{align*}
	&\rho\bigl(
	g_{r}\cdots g_{1}
	u_{m-1}\cdots u_{r+j+1}x,
	g_{r}\cdots g_{1}
	v_{m-1}\cdots v_{r+j+1}x
	\bigr)\\
&\quad\leq
	\exp\bigl(
	2\log N(g_{r})-c_{0}(r-1)
	\bigr)\leq
	e^{-c_{0}r/4}.
\end{align*}
By the definition of \(\omega_{j,a}\), we have
$
D_{r,j}(u,v,x)
\leq
\omega_{j,a}\bigl(c_{0}r/4\bigr).
$
Therefore,
\begin{align*}
	&\mathbb{E}\bigl[
	\ell_{a}(g_{r})
	D_{r,j}(u,v,x)
	\mathbf{1}_{\mathcal A_{r,j}^{c}(u,v,x)
		\cap\mathcal B_{r}^{c}}
	\bigr]\leq
	C\omega_{j,a}\bigl(c_{0}r/4\bigr),
\end{align*}
where $\omega_{j,a}$ is defined in \Cref{lem-absmod}.
Since \(\omega_{j,a}\) is nonincreasing, a change of scale in
the argument gives
\begin{align*}
	&\sum_{r=r_{0}}^{m-2}
	\sum_{j=0}^{m-r-2}
	\omega_{j,a}\biggl(\frac{c_{0}r}{4}\biggr)\leq
	C\sum_{j=0}^{m-2}
	\sum_{t=1}^{\infty}\omega_{j,a}(t)\leq
	C m^{1-\delta},
\end{align*}
where the last inequality follows from \eqref{eq-ac-03}.
Collecting the estimates on the three events, together with
the finitely many indices \(1\leq r<r_{0}\), we obtain
\begin{align}\label{eq-ww-11a}
\sum_{r=1}^{\infty}I_{r,5}
\leq
C\bigl(1+m^{1-\delta}\bigr)
\leq
C m^{1-\delta}.
\end{align}
Using  \eqref{eq-ww-10b}, \eqref{eq-ww-10c}, and \eqref{eq-ww-11a}, we conclude the estimate required for \eqref{eq-ae-03}.

We finally prove the estimate required for \eqref{eq-ae-04}. By the definition of $\Gamma_r(u,v)$, the terms involving
$\psi_m^{[a]}(U_{r-1}(u))$ and $\psi_m^{[a]}(U_{r-1}(v))$
cancel, and hence
\begin{align*}
	\Gamma_{r}(u,v)
	={}&
	X_{r}(u;g)-X_{r}(v;g)
	-X_{r}(u;g')+X_{r}(v;g')\\
	&+
	\psi_{m}^{[a]}\bigl(T_{g}U_{r-1}(u)\bigr)
	-\psi_{m}^{[a]}\bigl(T_{g}U_{r-1}(v)\bigr)\\
	&-
	\psi_{m}^{[a]}\bigl(T_{g'}U_{r-1}(u)\bigr)
	+\psi_{m}^{[a]}\bigl(T_{g'}U_{r-1}(v)\bigr).
\end{align*}
By the triangle inequality, the symmetry of \(g\) and \(g'\), and
the independence of \(g'\) from \(g,g_{1},\ldots,g_{r-1}\), we have
\begin{align*}
	&\sup_{u,v}
	\mathbb{E}\bigl[
	\bigl(\ell_{a}(g)+\ell_{a}(g')\bigr)
	\lvert\Gamma_{r}(u,v)\rvert
	\bigr]\leq
	C I_{r,4}
	+C I_{r,5}
	+C\delta_{1,\infty}(r)
	+C\sum_{j=0}^{m-r-2}
	\delta_{1,\infty}(r+j+1).
\end{align*}
Therefore, by \eqref{eq-ww-10b}, \eqref{eq-ww-11a},
\eqref{eq-ww-10}, and \Cref{lem-1}, we obtain
\[
\sum_{r=1}^{\infty}
\sup_{u,v}
\mathbb{E}\bigl[
\bigl(\ell_{a}(g)+\ell_{a}(g')\bigr)
\lvert\Gamma_{r}(u,v)\rvert
\bigr]
\leq
C\bigl(1+a^{1-\delta}+m^{1-\delta}\bigr).
\]
Since \(\Gamma_{r}(u,v)=0\) for \(r\geq m\), we obtain
\eqref{eq-ae-04}.
\end{proof}

We next establish a uniform conditional \(L^1\) bound for partial sums of the differences \(Y_{r}(u)-Y_{r}(v)\) arising from different initial states.
\begin{lemma}
	\label{lem-futureL1}
	Let $\mathscr{H}$ be a sub-$\sigma$-algebra of the underlying
	probability space. Let $u,v$ be $\mathscr{H}$-measurable
	random variables with values in $G^{m-1}$.
	Let $(g_{r})_{r\geq1}$ be an i.i.d. sequence with common law
	$\mu$, independent of $\mathscr{H}$. Let $Y_{r}(u)$ be defined as in \eqref{eq-defi-Yr}.
		Then, for every pair of deterministic integers
	$1\leq p\leq q\leq m-1$,
	\[
	\begin{aligned}
		&\mathbb{E}\bigg[
		\Big|
		\sum_{r=p}^{q}\bigl(Y_{r}(u)-Y_{r}(v)\bigr)
		\Big|
		\,\bigg|\,\mathscr{H}
		\bigg]
		\leq C_{0}
		\quad\text{almost surely},
	\end{aligned}
	\]
	where
	\[
	C_{0}:=\sum_{r=1}^{\infty}\delta_{1,\infty}(r)+4\sup_{a\geq1,\;m\geq2}
	\|\psi_{m}^{[a]}\|_{\infty}<\infty.
	\]
\end{lemma}

\begin{proof}
	It follows from the definition of $Y_{r}(u)$ that, for every
	\(r\geq1\),
	\begin{align*}
			Y_{r}(u)-Y_{r}(v)
		&=X_{r}(u)-X_{r}(v)+\big[\psi_{m}^{[a]}\bigl(U_{r}(u)\bigr)
		-\psi_{m}^{[a]}\bigl(U_{r}(v)\bigr)\big]\\
		&\quad-\big[\psi_{m}^{[a]}\bigl(U_{r-1}(u)\bigr)
		-\psi_{m}^{[a]}\bigl(U_{r-1}(v)\bigr)\big].
	\end{align*}
		Conditionally on $\mathscr{H}$, the initial states $u,v$
	are fixed, while $g_{1},\ldots,g_{r}$ remain independent
	with common law $\mu$. By the definition of $x_{a,m}$,
	the $1$-Lipschitz property of censoring, and the definition
	of $\delta_{1,\infty}(r)$, we have
	\[
	\begin{aligned}
		\mathbb{E}\big[
		|X_{r}(u)-X_{r}(v)|
		\mid\mathscr{H}
		\big]\leq
		\sup_{x,y\in X}
		\mathbb{E}|X_{r,x}-X_{r,y}|=\delta_{1,\infty}(r),
		\qquad 1\leq r\leq m-1.
	\end{aligned}
	\]
	Moreover, we have $|\psi_{m}^{[a]}(U_{r-1}(u))|\leq \sup_{a\geq1,\;m\geq2}
	\|\psi_{m}^{[a]}\|_{\infty}.$ Therefore, 
	\[
	\begin{aligned}
		\mathbb{E}\bigg[
		\bigg|
		\sum_{r=p}^{q}\bigl(Y_{r}(u)-Y_{r}(v)\bigr)
		\bigg|
		\,\Big|\,\mathscr{H}
		\bigg]&
	\leq
		\sum_{r=p}^{q}
		\mathbb{E}\left[
		|X_{r}(u)-X_{r}(v)|
		\,\middle|\,\mathscr{H}
		\right]
		+4\sup_{a\geq1,\;m\geq2}
		\|\psi_{m}^{[a]}\|_{\infty}\\
		&\leq
		\sum_{r=p}^{q}\delta_{1,\infty}(r)+4\sup_{a\geq1,\;m\geq2}
		\|\psi_{m}^{[a]}\|_{\infty}
		\leq C_0<\infty.
	\end{aligned}
	\]
\end{proof}

Fix $n\geq1$ and define
\[
S_{n,m}
:=\sum_{k=m}^{m+n-1}Y_{k},
\]
where we recall that $Y_{k}=Y_{k,m}^{[a]}$.
 Let $(g'_{k})_{k\geq 1}$ be an independent copy of $(g_{k})_{k\geq 1}$. Fix $i\in\{m,\ldots,m+n-1\}$.
Define the sequence $(g_{k}^{(i)})_{k\geq1}$ by
\[
g_{k}^{(i)}
:=
\begin{cases}
	g_{i}^{\prime},&k=i,\\
	g_{k},&k\neq i.
\end{cases}
\]
For $k\geq m-1$, define the corresponding states
$
U_{k}^{(i)}
:=
\bigl(g_{k-m+2}^{(i)},\ldots,
g_{k}^{(i)}\bigr),
$
and, for $k\geq m$, define
$
Y_{k}^{(i)}
:=
F_{a,m}\bigl(U_{k-1}^{(i)},g_{k}^{(i)}\bigr).
$
Thus the superscript $(i)$ indicates replacement of the
single coordinate $g_{i}$ by
$g_{i}^{\prime}$ in the defining function.
Set
\[
S_{n,m}^{(i)}
:=\sum_{k=m}^{m+n-1}Y_{k}^{(i)},
\qquad
\Delta_{i}
:=S_{n,m}-S_{n,m}^{(i)}.
\]
Since $Y_{k}$ depends only on
$g_{k-m+1},\ldots,g_{k}$, we have
$
Y_{k}^{(i)}=Y_{k}$
 if $k<i$ or $k>i+m-1.$
Consequently, we have
\[
\Delta_{i}
=
Y_{i}-Y_{i}^{(i)}
+\sum_{j=i+1}^{\min\{i+m-1,m+n-1\}}
\bigl(Y_{j}-Y_{j}^{(i)}\bigr).
\]
An empty sum is understood to be zero. Conditionally on
$\mathscr{F}_{i}\vee\sigma(g_{i}^{\prime})$,
the states $U_{i}$ and $U_{i}^{(i)}$ are fixed, whereas
$g_{i+1},g_{i+2},\ldots$ remain
independent with common law $\mu$.
Applying \Cref{lem-futureL1} with
$
\mathscr{H}
=\mathscr{F}_{i}\vee\sigma(g_{i}^{\prime}),$
$
u=U_{i},$
$v=U_{i}^{(i)},$
and with $(g_{i+r})_{r\geq1}$ as the driving sequence,
we have
\[
\mathbb{E}\biggl[
\biggl|
\sum_{j=i+1}^{\min\{i+m-1,m+n-1\}}
\bigl(Y_{j}-Y_{j}^{(i)}\bigr)
\biggr|
\,\Bigm|\,
\mathscr{F}_{i}\vee\sigma(g_{i}^{\prime})
\biggr]
\leq C.
\]
Since $Y_{i}-Y_{i}^{(i)}$ is measurable with respect to
$\mathscr{F}_{i}\vee\sigma(g_{i}^{\prime})$,
the triangle inequality yields
\begin{equation}\label{eq-ae-08}
	\mathbb{E}\bigl[
	|\Delta_{i}|
	\,\bigm|\,
	\mathscr{F}_{i},g_{i}^{\prime}
	\bigr]
	\leq
	\bigl|Y_{i}-Y_{i}^{(i)}\bigr|+C.
\end{equation}

We next establish the conditional second-moment bound.
Recall that
\[
\int F_{a,m}(u,h)\,\mu(\mathrm{d}h)=0
\qquad
\text{for every }u\in G^{m-1}.
\]
For every $j>i$, both $U_{j-1}$ and $U_{j-1}^{(i)}$
are measurable with respect to
$\mathscr{F}_{j-1}\vee\sigma(g_{i}^{\prime})$,
while  $g_{j}$
is independent of this $\sigma$-algebra and has law $\mu$.
Hence
\[
\mathbb{E}\bigl[
Y_{j}-Y_{j}^{(i)}
\,\bigm|\,
\mathscr{F}_{j-1}\vee\sigma(g_{i}^{\prime})
\bigr]
=0.
\]
In particular, for $i<j<k$, the tower property gives
$
\mathbb{E}\bigl[
\bigl(Y_{j}-Y_{j}^{(i)}\bigr)
\bigl(Y_{k}-Y_{k}^{(i)}\bigr)
\,\bigm|\,
\mathscr{F}_{i},g_{i}^{\prime}
\bigr]
=0.
$
Also, for every $j>i$,
$
\mathbb{E}\bigl[
\bigl(Y_{i}-Y_{i}^{(i)}\bigr)
\bigl(Y_{j}-Y_{j}^{(i)}\bigr)
\,\bigm|\,
\mathscr{F}_{i},g_{i}^{\prime}
\bigr]
=0.
$
Expanding the square of $\Delta_{i}$ therefore yields
\[
\begin{aligned}
	&\mathbb{E}\bigl[
	\Delta_{i}^{2}
	\,\bigm|\,
	\mathscr{F}_{i},g_{i}^{\prime}
	\bigr]=
	\bigl(Y_{i}-Y_{i}^{(i)}\bigr)^{2}
	+\sum_{j=i+1}^{\min\{i+m-1,m+n-1\}}
	\mathbb{E}\bigl[
	\bigl(Y_{j}-Y_{j}^{(i)}\bigr)^{2}
	\,\bigm|\,
	\mathscr{F}_{i},g_{i}^{\prime}
	\bigr].
\end{aligned}
\]
\Cref{prop-budgets}, applied with initial
states $U_{i}$ and $U_{i}^{(i)}$, gives
\[
\mathbb{E}\bigl[
\bigl(Y_{i+r}-Y_{i+r}^{(i)}\bigr)^{2}
\,\bigm|\,
\mathscr{F}_{i},g_{i}^{\prime}
\bigr]
\leq\beta_{r,1},
\qquad 1\leq r\leq m-1.
\]
Using the summability bound for $\beta_{r,1}$, we obtain
\begin{equation}\label{eq-ae-09}
	\begin{aligned}
		\mathbb{E}\bigl[
		\Delta_{i}^{2}
		\,\bigm|\,
		\mathscr{F}_{i},g_{i}^{\prime}
		\bigr]&\leq
		\bigl(Y_{i}-Y_{i}^{(i)}\bigr)^{2}
		+\sum_{r=1}^{m-1}\beta_{r,1}\\
		&\leq
		\bigl(Y_{i}-Y_{i}^{(i)}\bigr)^{2}
		+C\bigl(1+a^{1-\delta}+m^{1-\delta}\bigr).
	\end{aligned}
\end{equation}

More generally, for any deterministic integers
$m\leq s\leq i\leq t$, \eqref{eq-ae-08} and
\eqref{eq-ae-09} hold with $\Delta_{i}$ replaced by
$\sum_{k=s}^{t}\bigl(Y_{k}-Y_{k}^{(i)}\bigr)$.
Indeed, the corresponding  sum runs from $i+1$
to $\min\{i+m-1,t\}$, and the preceding arguments apply
with this endpoint. The constants are independent of
the chosen interval. From these estimates, we obtain the following corollary.

\begin{corollary}
	\label{cor-mixed}
Let $S_{n,m}$ and $S_{n,m}^{(i)}$ be defined as above.  Then there exists a constant $C>0$, depending only on $ \mu, \delta, d$, and the fixed norm, such that the following statements hold.
	\begin{enumerate}
	\item [(i)] 	For every $i\in\{m,\ldots,m+n-1\}$, we have
	\begin{equation}\label{eq-ae-10}
		\begin{aligned}
			&\mathbb{E}\Bigl[
			\bigl(Y_{i}-Y_{i}^{(i)}\bigr)^{2}|\Delta_{i}|
			\Bigr]+
			\mathbb{E}\Bigl[
			\bigl|Y_{i}^{2}-\bigl(Y_{i}^{(i)}\bigr)^{2}\bigr|
			|\Delta_{i}|
			\Bigr]\leq
			C\bigl(1+a^{1-\delta}+m^{1-\delta}\bigr).
		\end{aligned}
	\end{equation}
	\item [(ii)]  For $m\leq i<j$, define the sequence $\big(g_k^{(i,j)}\big)_{k\geq1}$ by
	\[
	g_{k}^{(i,j)}
	:=
	\begin{cases}
		g_{i}^{\prime},&k=i,\\
		g_{j}^{\prime},&k=j,\\
		g_{k},&k\notin\{i,j\}.
	\end{cases}
	\]
	For $k\geq m$, set
	\[
	Y_{k}^{(i,j)}
	:=
	F_{a,m}\bigl(
	\bigl(
	g_{k-m+1}^{(i,j)},\ldots,
	g_{k-1}^{(i,j)}
	\bigr),
	g_{k}^{(i,j)}
	\bigr),
	\]
	and define
	$
	S_{n,m}^{(i,j)}
	:=\sum_{k=m}^{m+n-1}Y_{k}^{(i,j)}.
	$
Define
$
	\Delta_{j}^{(i)}
	:=S_{n,m}^{(i)}-S_{n,m}^{(i,j)},
$
	and 
	$
	\Gamma_{ij}
	:=
	Y_{j}-Y_{j}^{(i)}
	-Y_{j}^{(j)}+Y_{j}^{(i,j)}.
	$
	Then,
	\begin{equation}\label{eq-ae-11}
		\begin{aligned}
			&\sum_{j=i+1}^{\min\{i+m-1,m+n-1\}}
			\mathbb{E}\Bigl[
			\bigl|
			\bigl(Y_{i}-Y_{i}^{(i)}\bigr)\Gamma_{ij}
			\bigr|
			\bigl(
			|\Delta_{j}|+|\Delta_{j}^{(i)}|
			\bigr)
			\Bigr]\leq
			C\bigl(1+a^{1-\delta}+m^{1-\delta}\bigr).
		\end{aligned}
	\end{equation}
	\item [(iii)]  
	Both \eqref{eq-ae-10} and \eqref{eq-ae-11} remain valid when
	$S_{n,m}$ and each of its replaced versions are replaced by the
	corresponding sums over a deterministic interval $[s,t]$ with
	$m\leq s\leq i\leq t$. In this case, the upper summation limit
	in \eqref{eq-ae-11} is $\min\{i+m-1,t\}$. The constants are
	independent of the chosen interval.
	\end{enumerate}

\end{corollary}

\begin{proof}
	Recall that
$
	\ell_{a}(h)=1+\min\{L(h),a\}.
$
	The uniform  bound \eqref{eq-ad-04} gives
	\[
	|F_{a,m}(u,h)|\leq C\ell_{a}(h),
	\qquad
	u\in G^{m-1},\quad h\in G.
	\]
	Moreover,
$
	\mathbb{E}\ell_{a}(g_1)
	+\mathbb{E}\ell_{a}(g_1)^{2}
	\leq C,$
	$\mathbb{E}\ell_{a}(g_1)^{3}
	\leq Ca^{1-\delta}.$
	For the first term of \eqref{eq-ae-10}, the variable
	$(Y_{i}-Y_{i}^{(i)})^{2}$ is measurable with
	respect to
	$\mathscr{F}_{i}\vee\sigma(g_{i}^{\prime})$.
	Therefore, by \eqref{eq-ae-08}, we have
	\[
	\begin{aligned}
		\mathbb{E}\Bigl[
		\bigl(Y_{i}-Y_{i}^{(i)}\bigr)^{2}|\Delta_{i}|
		\Bigr]
		&=
		\mathbb{E}\Bigl[
		\bigl(Y_{i}-Y_{i}^{(i)}\bigr)^{2}
		\mathbb{E}\bigl[
		|\Delta_{i}|
		\,\bigm|\,
		\mathscr{F}_{i},g_{i}^{\prime}
		\bigr]
		\Bigr]\\
		&\quad\leq
		\mathbb{E}\bigl|Y_{i}-Y_{i}^{(i)}\bigr|^{3}
		+C\mathbb{E}\bigl(Y_{i}-Y_{i}^{(i)}\bigr)^{2}\\
		&\quad\leq Ca^{1-\delta},
	\end{aligned}
	\]
	where 
	the last inequality follows from
$
	|Y_{i}-Y_{i}^{(i)}|
	\leq
	C\bigl(
	\ell_{a}(g_{i})
	+\ell_{a}(g_{i}^{\prime})
	\bigr)
$
	and the preceding moment bounds.
	
For the second term of \eqref{eq-ae-10}, note that
$\bigl|Y_{i}^{2}-(Y_{i}^{(i)})^{2}\bigr|$ is measurable with respect to
$\mathscr{F}_{i}\vee\sigma(g_{i}^{\prime})$ and satisfies
$\bigl|Y_{i}^{2}-(Y_{i}^{(i)})^{2}\bigr|
\leq Y_{i}^{2}+(Y_{i}^{(i)})^{2}$.
	Applying \eqref{eq-ae-08} once more gives
	\[
	\begin{aligned}
		\mathbb{E}\Bigl[
		\bigl|Y_{i}^{2}-\bigl(Y_{i}^{(i)}\bigr)^{2}\bigr|
		|\Delta_{i}|
		\Bigr]
		&\leq
		\mathbb{E}\Bigl[
		\bigl|Y_{i}^{2}-\bigl(Y_{i}^{(i)}\bigr)^{2}\bigr|
		\bigl(
		|Y_{i}-Y_{i}^{(i)}|+C
		\bigr)
		\Bigr]\\
		&\quad\leq
		\mathbb{E}\bigl(
		|Y_{i}|+|Y_{i}^{(i)}|
		\bigr)^{3}
		+C\mathbb{E}\bigl(
		|Y_{i}|+|Y_{i}^{(i)}|
		\bigr)^{2}\\
		&\quad\leq Ca^{1-\delta}.
	\end{aligned}
	\]
	Thus \eqref{eq-ae-10} holds.
	
	We now prove \eqref{eq-ae-11}.
	Fix
$
	i<j\leq\min\{i+m-1,m+n-1\}.
$
	Condition first on
$
	\mathscr{F}_{j}
	\vee\sigma\bigl(
	g_{i}^{\prime},
	g_{j}^{\prime}
	\bigr).
$
	Under this conditioning, the original and replaced
	states through time $j$ are fixed, while the matrices
	strictly after $j$ remain independent with law $\mu$.
	Applying \eqref{eq-ae-08} to the original sequence and to
	the sequence with $g_{i}$ replaced by $g_{i}^{\prime}$,
	respectively, yields
	\[
	\begin{aligned}
		\mathbb{E}\bigl[
		|\Delta_{j}|+|\Delta_{j}^{(i)}|
		\,\bigm|\,
		\mathscr{F}_{j},
		g_{i}^{\prime},
		g_{j}^{\prime}
		\bigr]
		\leq
		\bigl|Y_{j}-Y_{j}^{(j)}\bigr|
		+\bigl|Y_{j}^{(i)}-Y_{j}^{(i,j)}\bigr|
		+2C\leq
		C\bigl(
		\ell_{a}(g_{j})
		+\ell_{a}(g_{j}^{\prime})
		\bigr).
	\end{aligned}
	\]
	Both $\Gamma_{ij}$ and $Y_{i}-Y_{i}^{(i)}$ are measurable
	under this conditioning. The tower property therefore gives
	\[
	\begin{aligned}
		&\mathbb{E}\Bigl[
		\bigl|
		\bigl(Y_{i}-Y_{i}^{(i)}\bigr)\Gamma_{ij}
		\bigr|
		\bigl(
		|\Delta_{j}|+|\Delta_{j}^{(i)}|
		\bigr)
		\,\Bigm|\,
		\mathscr{F}_{i},g_{i}^{\prime}
		\Bigr]\\
		&\quad\leq
		C\bigl|Y_{i}-Y_{i}^{(i)}\bigr|
		\mathbb{E}\Bigl[
		|\Gamma_{ij}|
		\bigl(
		\ell_{a}(g_{j})
		+\ell_{a}(g_{j}^{\prime})
		\bigr)
		\,\Bigm|\,
		\mathscr{F}_{i},g_{i}^{\prime}
		\Bigr].
	\end{aligned}
	\]
	
Conditionally on
$\mathscr{F}_{i}\vee\sigma(g_{i}^{\prime})$,
the two initial states $U_{i}$ and $U_{i}^{(i)}$ are fixed.
Both constructions use the same matrices
$g_{i+1},\ldots,g_{j-1}$.
At time $j$, we consider $g_{j}$ and its independent copy
$g_{j}^{\prime}$.
Thus \eqref{eq-ae-04}, with $r=j-i$, gives
	\[
	\mathbb{E}\Bigl[
	|\Gamma_{ij}|
	\bigl(
	\ell_{a}(g_{j})
	+\ell_{a}(g_{j}^{\prime})
	\bigr)
	\,\Bigm|\,
	\mathscr{F}_{i},g_{i}^{\prime}
	\Bigr]
	\leq\beta_{j-i,3}.
	\]
	Combining these inequalities, summing over $j$, and
	taking expectations, we obtain
	\[
	\begin{aligned}
		&\sum_{j=i+1}^{\min\{i+m-1,m+n-1\}}
		\mathbb{E}\Bigl[
		\bigl|
		\bigl(Y_{i}-Y_{i}^{(i)}\bigr)\Gamma_{ij}
		\bigr|
		\bigl(
		|\Delta_{j}|+|\Delta_{j}^{(i)}|
		\bigr)
		\Bigr]\\
		&\quad\quad\quad\leq
		C\mathbb{E}\bigl|Y_{i}-Y_{i}^{(i)}\bigr|
		\sum_{r=1}^{m-1}\beta_{r,3}\leq
		C\bigl(1+a^{1-\delta}+m^{1-\delta}\bigr),
	\end{aligned}
	\]
	where the last inequality follows from
$
	\mathbb{E}|Y_{i}-Y_{i}^{(i)}|
	\leq
	C\mathbb{E}[
	\ell_{a}(g_{i})
	+\ell_{a}(g_{i}^{\prime})]
	\leq C
$
	and \eqref{eq-ae-05}. This proves \eqref{eq-ae-11}.
	
Finally, the constant in \eqref{eq-ae-08} is independent
of the endpoints of the summation interval.
	Replacing the full sum by a sum over $[s,t]$ therefore
	leaves the proof unchanged, except that the upper
	summation limit becomes $\min\{i+m-1,t\}$.
	This proves the final assertion.
\end{proof}

\section{A concentration inequality}\label{sec-concentration}
To prove Berry--Esseen bounds using Stein's method, we need to
handle the discontinuity of the indicator function
$\mathbf{1}_{\{w\leq z\}}$.
A perturbation of the argument changes the value of the indicator
only when the argument crosses the threshold $z$.
We therefore need upper bounds on probabilities of the form
\[
\mathbb{P}(a\leq W\leq b),
\]
where $W$ is the statistic of interest and $a$ and $b$ are
possibly random interval endpoints satisfying $a\leq b$.
Such upper bounds are called concentration inequalities.

Concentration inequalities and their extensions have been
extensively studied.
For sums of independent random variables, concentration
inequalities and their randomized versions can be found in
\cite{ChenShao2001,ChenShao2007,ShaoZhou2016,LeiShao2022}.
Exponential randomized concentration inequalities were developed
in \cite{Shao2010,LeungShaoZhang2024}.
For locally dependent random variables, related results can be
found in \cite{ChenShao2004,Zhang2024,CaiShaoZhang2026}.

In this section, we establish a sharper conditional concentration
inequality for the $m$-dependent martingale difference sequence
$(Y_{k,m}^{[a]})$ obtained from the approximation
of random matrix products. The inequality improves the dependence of the upper bound on $m$,
yielding the probability bounds required for the proof of
\Cref{thm-main} when $m$ grows with the sample size.

Throughout this section, the censoring level $a\geq1$ and the integer 
$m\geq2$ are fixed. We use the $m$-dependent martingale difference sequence
$Y_{i}=Y^{[a]}_{i,m}$ constructed in \Cref{lem-finiteG}. We further assume that there exist two fixed positive constants $c^{*}_{1}$ and $c^{*}_{2}$ such that $c^{*}_{1}s\leq s_{a,m}\leq c^{*}_{2}s$.
We consider a block of consecutive matrices.
A matrix near its right end can also affect terms after the block.
We include all such terms in the sum before conditioning
on the matrices on either side of the block.
Specifically, we take $g_{m},\ldots,g_{N}$ as the central block,
where $N=m+H-1$ and $H\geq m$.
The matrices $g_{1},\ldots,g_{m-1}$ and
$g_{N+1},\ldots,g_{N+m-1}$ form the left and right
boundary blocks, respectively. Let
\begin{align}\label{eq-ua-01}
	\mathscr{G}
	&=\sigma\bigl(g_{1},\ldots,g_{m-1},
	g_{N+1},\ldots,g_{N+m-1}\bigr),
	\qquad
	S_{H,m}=\sum_{i=m}^{m+H-1}Y_{i}.
\end{align}
For $1\leq r<m$, define
\[
h_{r,m}\bigl(g_{N+1},\ldots,g_{N+r}\bigr)
=\mathbb{E}\bigl(Y_{N+r}\mid
             g_{N+1},\ldots,g_{N+r}\bigr).
\]
This definition integrates the other $m-r$ coordinates of the
kernel with respect to their independent product law. Those
coordinates have indices $N+r-m+1,\ldots,N$. They are central
coordinates because $H\geq m$. Moreover, the other boundary coordinates are independent
of the coordinates defining $Y_{N+r}$.  Consequently,
$
\mathbb{E}(Y_{N+r}\mid\mathscr{G})=h_{r,m}(g_{N+1},\ldots,g_{N+r}).
$
Define
\begin{equation}\label{eq-af-01}
R=\sum_{r=1}^{m-1}\bigl(Y_{N+r}-h_{r,m}\bigr),
\qquad V=S_{H,m}+R.
\end{equation}

For a central index $i$, integrating $Y_{i}$ with respect to
$g_{i}$ gives zero for every fixed state $U_{i-1}$, by
\Cref{lem-finiteG}.
 Hence,  $\mathbb{E}(Y_{i}\mid\mathscr{G})=0$. Summing this
identity and using the definition of $h_{r,m}$ yields
\begin{equation}\label{eq-af-02}
V=\sum_{i=m}^{N+m-1}Y_{i}
 -\mathbb{E}\Bigl[\sum_{i=m}^{N+m-1}Y_{i}\Bigm|\mathscr{G}\Bigr].
\end{equation}
In particular, $\mathbb{E}(V\mid\mathscr{G})=0$.

We can also express \(V\) in terms of \(X^{[a]}_{i,m}\). Applying \eqref{eq-ad-06} with \(j=m\) and \(n=N+m-1\) gives
\[
\sum_{i=m}^{N+m-1}Y_{i}
 =\sum_{i=m}^{N+m-1}X^{[a]}_{i,m}
  +\psi_{m}^{[a]}\bigl(U_{N+m-1}\bigr)
  -\psi_{m}^{[a]}\bigl(U_{m-1}\bigr).
\]
The states $U_{m-1}$ and $U_{N+m-1}$ consist of the
left and right boundary coordinates, respectively. The two terms involving $\psi_{m}^{[a]}$
are thus $\mathscr{G}$-measurable and disappear when we
subtract the conditional expectation given $\mathscr{G}$.
It follows that
\begin{equation}\label{eq-af-03}
V=\sum_{i=m}^{N+m-1}X^{[a]}_{i,m}
 -\mathbb{E}\biggl(\sum_{i=m}^{N+m-1}X^{[a]}_{i,m}\biggm|\mathscr{G}\biggr).
\end{equation}

We now state the conditional concentration inequality for $V$.
\begin{proposition}
	\label{prop-profile}
Fix $\eta>1$ and set $a=m$.
There exist constants $L_{*}>0$ and $C>0$, depending only on
$\eta$, $c^{*}_{1}$, $c^{*}_{2}$, $s$, $\mu$, $\delta$, $d$,
and the fixed norm, such that, for
$\alpha=L_{*}m^{1-\delta}$ and every $H\geq64m$, we have
	\begin{equation}\label{eq-ag-17}
		\mathbb{E}\biggl[
		\sup_{z\in\mathbb{R}}
		\mathbb{P}\bigl(z<V\leq z+2\alpha\mid\mathscr{G}\bigr)
		\biggr]
		\leq
		C\biggl\{
		\frac{\alpha}{\sqrt{H}}
		+
		\biggl(\frac{m}{H}\biggr)^{\eta}
		\biggr\}.
	\end{equation}
\end{proposition}

To prove \Cref{prop-profile}, we first estimate $R$ and
its changes under single-coordinate replacements.
\begin{lemma}
	\label{lem-outgoing}
	For each $m\leq i\leq N$, let
	\[
	R^{(i)}
	:=
	\sum_{r=1}^{m-1}
	\Bigl[
	Y_{N+r}^{(i)}
	-h_{r,m}\bigl(
	g_{N+1},\ldots,g_{N+r}
	\bigr)
	\Bigr],
	\qquad
	V^{(i)}
	:=
	\sum_{k=m}^{N}Y_{k}^{(i)}+R^{(i)}.
	\]
	Then  there exists a constant $C>0$, depending only on
	 $c^{*}_{1}$, $c^{*}_{2}$, $s$, $\mu$, $\delta$, $d$, and the fixed norm,  such that the following statements hold.
	
	\begin{enumerate}
		\item[(i)]
	We have
		\begin{equation}\label{eq-af-04}
			\mathbb{E}\bigl[
			R^{2}\bigm|\mathscr{F}_{N}
			\bigr]
			\leq
			C\bigl(1+a^{1-\delta}+m^{1-\delta}\bigr).
		\end{equation}
		Moreover, for every $m\leq i\leq N$,
		\[
		\mathbb{E}\Bigl[
		\bigl(R^{(i)}\bigr)^{2}
		\,\Bigm|\,
		\mathscr{F}_{N}\vee
		\sigma(g_{i}^{\prime})
		\Bigr]
		\leq
		C\bigl(1+a^{1-\delta}+m^{1-\delta}\bigr).
		\]
		
		\item[(ii)]
		Consider the  sum
	$
		S_{N,m}
		=
		\sum_{k=m}^{N+m-1}Y_{k}.
$ 
		For every $m\leq k\leq N+m-1$, define
	$
		\Delta_{k}
		:=
		S_{N,m}-S_{N,m}^{(k)}.
$
Then, 
		for each central index $m\leq i\leq N$, we have
		\begin{equation}\label{eq-af-05}
			\Delta_{i}
			=
			V-V^{(i)}
			=
			\sum_{j=i}^{i+m-1}
			\bigl(Y_{j}-Y_{j}^{(i)}\bigr).
		\end{equation}
		\item[(iii)]
		For every $m\leq i\leq N$, we have
		\begin{equation}\label{eq-af-06}
			\begin{aligned}
				&\mathbb{E}\Bigl[
				\bigl|Y_{i}-Y_{i}^{(i)}\bigr|
				\bigl(
				|R^{(i)}|+|R-R^{(i)}|
				\bigr)
				|\Delta_{i}|
				\Bigr]\leq
				C\bigl(1+a^{1-\delta}+m^{1-\delta}\bigr).
			\end{aligned}
		\end{equation}
	\end{enumerate}
\end{lemma}
\begin{proof}
	We first prove \eqref{eq-af-04}.  As shown in 	\Cref{lem-finiteG}, $(Y_{k})_{k\geq m}$ is a martingale difference sequence with respect to the filtration $(\mathscr{F}_{k})_{k\geq m}$. By the definition of $h_{r,m}$, we have $h_{r,m}\bigl(g_{N+1},\ldots,g_{N+r}\bigr)
	=\mathbb{E}\bigl(Y_{N+r}\mid
	g_{N+1},\ldots,g_{N+r}\bigr).$ Hence, for every fixed $g_{1},\ldots,g_{r-1}$, we have
\[
\int h_{r,m}(g_1,\ldots,g_{r-1},g)\,\mu(\mathrm dg)=0.
\]
Since $g_{N+r}$ is independent of $\mathscr F_{N+r-1}$, both $Y_{N+r}$ and $h_{r,m}$ have conditional mean zero given $\mathscr F_{N+r-1}$. Hence
\begin{align*}
	\mathbb{E}\bigl[
Y_{N+r}
-h_{r,m}\bigl(
g_{N+1},\ldots,g_{N+r}
\bigr)
\,\bigm|\,
\mathscr{F}_{N+r-1}
\bigr]
=0.
\end{align*}
Consequently, we have
\begin{align}\label{eq-kk-02}
	\mathbb{E}\bigl[
	\bigl(Y_{N+r}-h_{r,m}\bigr)
	\bigl(Y_{N+s}-h_{s,m}\bigr)
	\,\bigm|\,
	\mathscr{F}_{N}
	\bigr]
	=0, \quad 1\leq r<s<m
\end{align}
We next estimate the conditional second moment of each
summand.  Fix $1\leq r<m$. Conditionally on $\mathscr{F}_{N}$,
the matrices $g_{N+1},\ldots,g_{N+r}$
are independent with common law $\mu$.
By the definition of $h_{r,m}$ and Jensen's inequality,
we have
\begin{equation}\label{eq-kk-04}
	\begin{aligned}
		&\mathbb{E}\Bigl[
		\bigl|
		Y_{N+r}
		-h_{r,m}(g_{N+1},\ldots,g_{N+r})
		\bigr|^{2}
		\,\Bigm|\,
		\mathscr{F}_{N}
		\Bigr]\\
		&\quad\leq
		\int_{G^{m-r}}
		\mathbb{E}\Bigl[
		\Bigl|
		Y_{N+r}
		-F_{a,m}\Bigl(
		\bigl(
		v_{1},\ldots,v_{m-r},
		g_{N+1},\ldots,g_{N+r-1}
		\bigr),
		g_{N+r}
		\Bigr)
		\Bigr|^{2}
		\,\Bigm|\,
		\mathscr{F}_{N}
		\Bigr]
		\prod_{\ell=1}^{m-r}\mu(\mathrm{d}v_{\ell})\\
		&\quad\leq \beta_{r,1},
	\end{aligned}
\end{equation}
where the last inequality follows from \eqref{eq-ae-02},
which holds uniformly over the initial states.
Combining \eqref{eq-kk-02} and \eqref{eq-kk-04}, and then using
\eqref{eq-ae-05}, we have
	\[
	\begin{aligned}
		\mathbb{E}\bigl[
		R^{2}\bigm|\mathscr{F}_{N}
		\bigr]
		&=
		\sum_{r=1}^{m-1}
		\mathbb{E}\Bigl[
		\bigl|Y_{N+r}-h_{r,m}\bigr|^{2}
		\,\Bigm|\,
		\mathscr{F}_{N}
		\Bigr]\leq
		\sum_{r=1}^{m-1}\beta_{r,1}\leq
		C\bigl(1+a^{1-\delta}+m^{1-\delta}\bigr).
	\end{aligned}
	\]
	This proves \eqref{eq-af-04}. Conditionally on
	$\mathscr{F}_{N}\vee\sigma(g_{i}^{\prime})$,
	the state $U_{N}^{(i)}$ is fixed, while
	$g_{N+1},\ldots,g_{N+m-1}$ remain independent with
	common law $\mu$.
	Since the preceding estimates are uniform over the initial states,
	the same argument gives
\begin{equation}\label{eq-af-04-resampled}
	\mathbb{E}\bigl[
	\bigl(R^{(i)}\bigr)^{2}
	\,\bigm|\,
	\mathscr{F}_{N}\vee\sigma(g_{i}^{\prime})
	\bigr]
	\leq
	C\bigl(1+a^{1-\delta}+m^{1-\delta}\bigr).
\end{equation}

	We next prove \eqref{eq-af-05}.
The replacement at $i$ leaves every $h_{r,m}$ unchanged,
since $i\leq N$ and $h_{r,m}$ depends only on
$g_{N+1},\ldots,g_{N+r}$.
	Moreover, since $Y_j$ depends only on $g_{j-m+1},\ldots,g_j$, we have
$
	Y_{j}-Y_{j}^{(i)}=0$,
	if $j<i$ or $j>i+m-1.$
	Since $m\leq i\leq N$, the  interval
	$\{i,\ldots,i+m-1\}$ is contained in
	$\{m,\ldots,N+m-1\}$. It follows that
	\[
	\begin{aligned}
		V-V^{(i)}
		&=
		\sum_{j=m}^{N+m-1}
		\bigl(Y_{j}-Y_{j}^{(i)}\bigr)
		=
		S_{N,m}-S_{N,m}^{(i)}
		=
		\sum_{j=i}^{i+m-1}
		\bigl(Y_{j}-Y_{j}^{(i)}\bigr)
		=\Delta_{i}.
	\end{aligned}
	\]
	
	It remains to prove \eqref{eq-af-06}.
	Since
$
	\mathscr{F}_{i}\vee\sigma(g_{i}^{\prime})
	\subseteq
	\mathscr{F}_{N}\vee\sigma(g_{i}^{\prime}),
$
it follows from \eqref{eq-af-04},
\eqref{eq-af-04-resampled} and the tower property that
	\[
	\mathbb{E}\bigl[
	\bigl(
	|R^{(i)}|+|R-R^{(i)}|
	\bigr)^{2}
	\,\bigm|\,
	\mathscr{F}_{i},g_{i}^{\prime}
	\bigr]\leq  	C\mathbb{E}\bigl[
	\bigl(|R|^{2}+
	|R^{(i)}|^{2}\bigr)
	\,\bigm|\,
	\mathscr{F}_{i},g_{i}^{\prime}
	\bigr]
	\leq
	C\bigl(1+a^{1-\delta}+m^{1-\delta}\bigr).
	\]
	Combining this with the
	conditional Cauchy--Schwarz inequality  and \eqref{eq-ae-09}, we have
	\[
	\begin{aligned}
		&\mathbb{E}\bigl[
		\bigl(
		|R^{(i)}|+|R-R^{(i)}|
		\bigr)|\Delta_{i}|
		\,\Bigm|\,
		\mathscr{F}_{i},g_{i}^{\prime}
		\bigr]\\
		&\quad\leq
		C\bigl(1+a^{1-\delta}+m^{1-\delta}\bigr)^{1/2}
		\Bigl[
		\bigl(Y_{i}-Y_{i}^{(i)}\bigr)^{2}
		+C\bigl(1+a^{1-\delta}+m^{1-\delta}\bigr)
		\Bigr]^{1/2}.
	\end{aligned}
	\]
	Finally, $Y_{i}-Y_{i}^{(i)}$ is measurable with respect
	to $\mathscr{F}_{i}\vee\sigma(g_{i}^{\prime})$.
	Multiplying both sides of the preceding inequality by
	$|Y_{i}-Y_{i}^{(i)}|$, taking expectations,
	and using $\sqrt{x^{2}+y^{2}}\leq |x|+|y|$, we obtain
	\[
	\begin{aligned}
		\mathbb{E}\Bigl[
		\bigl|Y_{i}-Y_{i}^{(i)}\bigr|
		\bigl(
		|R^{(i)}|+|R-R^{(i)}|
		\bigr)|\Delta_{i}|
		\Bigr]
		&\leq
		C\bigl(1+a^{1-\delta}+m^{1-\delta}\bigr)^{1/2}
		\mathbb{E}\bigl[
		\bigl(Y_{i}-Y_{i}^{(i)}\bigr)^{2}
		\bigr]\\
		&\quad
		+C\bigl(1+a^{1-\delta}+m^{1-\delta}\bigr)
		\mathbb{E}\bigl|Y_{i}-Y_{i}^{(i)}\bigr|\\
		&\leq
		C\bigl(1+a^{1-\delta}+m^{1-\delta}\bigr).
	\end{aligned}
	\]
	This proves \eqref{eq-af-06}.
\end{proof}
The following lemma gives an $L^1$ bound for the difference
between $\sum_{i=m}^{N}\mathbb{E}(Y_{i}^{2}\mid\mathscr{F}_{i-1})$
and $Hs_{a,m}^{2}$.
\begin{lemma}
\label{lem-predictable}
With $N=m+H-1$, we have
\begin{equation}\label{eq-af-07}
\mathbb{E}\biggl|
 \sum_{i=m}^{N}\bigl\{
   \mathbb{E}(Y_{i}^{2}\mid\mathscr{F}_{i-1})-s_{a,m}^{2}\bigr\}
 \biggr|
 \leq C\sqrt{H\bigl(1+a^{1-\delta}+m^{1-\delta}\bigr)},
\end{equation}
	where $C$ is a positive constant depending only on $c^{*}_{1},c^{*}_{2}, s, \mu, \delta, d$ and the fixed norm.  
\end{lemma}

\begin{proof}
For notational convenience, in the following proof, let
$\widehat{Y}_{i}=\mathbb{E}(Y_{i}^{2}\mid\mathscr{F}_{i-1})$.
Then, we have $0\leq \widehat{Y}_{i}\leq C$ and $\mathbb{E} \widehat{Y}_{i}=s_{a,m}^{2}$. Recalling that $U_{i-1}
=
(g_{i-m+1},\ldots,g_{i-1})$,  we see that
the random variable $\widehat{Y}_{i}$ depends only on
$
U_{i-1}
$
and 
\[
\mathbb{E}\bigl[
\widehat{Y}_{i+r}\bigm|U_{i-1}
\bigr]
=
\mathbb{E}\bigl[
Y_{i+r}^{2}\bigm|U_{i-1}
\bigr],\quad  \text{for }1\leq r<m-1
\]
By  \eqref{eq-ae-03}
and the inequality
$|\mathbb{E}A-\mathbb{E}B|\leq\mathbb{E}|A-B|$, we have
\begin{align}\label{eq-kk-08}
	\sup_{u,v}\Bigl|
	\mathbb{E}\bigl[
	Y_{i+r}^{2}\bigm|U_{i-1}=u
	\bigr]
	-
	\mathbb{E}\bigl[
	Y_{i+r}^{2}\bigm|U_{i-1}=v
	\bigr]
	\Bigr|
	\leq \beta_{r+1,2}.
\end{align}
By the tower property and stationarity, we have
$
\mathbb{E}[
\mathbb{E}[
Y_{i+r}^{2}\bigm|U_{i-1}
]
]
=
\mathbb{E}Y_{i+r}^{2}
=
s_{a,m}^{2}.
$
Combining this  with \eqref{eq-kk-08}, we have
\[
\begin{aligned}
	&\Bigl|
	\mathbb{E}\bigl[
	Y_{i+r}^{2}\bigm|U_{i-1}=u
	\bigr]
	-s_{a,m}^{2}
	\Bigr|=
	\Bigl|
	\mathbb{E}\Bigl[
	\mathbb{E}\bigl[
	Y_{i+r}^{2}\bigm|U_{i-1}=u
	\bigr]
	-
	\mathbb{E}\bigl[
	Y_{i+r}^{2}\bigm|U_{i-1}
	\bigr]
	\Bigr]
	\Bigr|\leq \beta_{r+1,2}.
\end{aligned}
\]
Therefore,
\[
\Bigl|
\mathbb{E}\bigl[
\widehat{Y}_{i+r}\bigm|U_{i-1}
\bigr]
-s_{a,m}^{2}
\Bigr|
\leq \beta_{r+1,2},
\qquad 1\leq r<m-1.
\]
For $r\geq m-1$, $\widehat{Y}_{i+r}$ depends only on
$g_{i},g_{i+1},\ldots$ and is therefore independent of
$U_{i-1}$. Hence
$\mathbb{E}(\widehat{Y}_{i+r}\mid U_{i-1})
=\mathbb{E}\widehat{Y}_{i+r}=s_{a,m}^{2}$.
Conditioning the product on $U_{i-1}$ now gives, for $r\geq1$,
\begin{align}\label{eq-kk-09}
\bigl|\operatorname{Cov}(\widehat{Y}_{i},\widehat{Y}_{i+r})\bigr|
 =\bigl|\mathbb{E}\bigl[
   (\widehat{Y}_{i}-s_{a,m}^{2})
   \{\mathbb{E}(\widehat{Y}_{i+r}\mid U_{i-1})-s_{a,m}^{2}\}\bigr]\bigr|\leq C\beta_{r+1,2}.
\end{align}
Combining \eqref{eq-kk-09} and \eqref{eq-ae-05}, we obtain
\begin{align*}
\mathbb{E}\bigg|\sum_{i=m}^{N}
                   (\widehat{Y}_{i}-s_{a,m}^{2})\bigg|^{2}
 \leq CH+CH\sum_{r=1}^{m-2}\beta_{r+1,2}\leq CH\bigl(1+a^{1-\delta}+m^{1-\delta}\bigr),
\end{align*}
as desired.
\end{proof}

We next introduce two functions used in the recursion.
For $z\in\mathbb{R}$ and $\alpha>0$, let
\[
\chi_{z,\alpha}(t)=
\begin{cases}
0,&t\leq z-\alpha,\\
\dfrac{t-z+\alpha}{\alpha},&z-\alpha<t<z,\\
1,&z\leq t\leq z+\alpha,\\
\dfrac{z+2\alpha-t}{\alpha},&z+\alpha<t<z+2\alpha,\\
0,&t\geq z+2\alpha.
\end{cases}
\qquad
F_{z,\alpha}(t)=\int_{-\infty}^{t}\chi_{z,\alpha}(u)\,\,\mathrm{d} u.
\]

The following lemma establishes some useful properties of the function $\chi_{z,\alpha}(t).$

\begin{lemma}
	\label{lem-chi-properties}
	For every $z\in\mathbb{R}$ and $\alpha>0$, the following
	properties hold.
	\begin{enumerate}
		\item[(i)]
		The function $\chi_{z,\alpha}$ takes values in $[0,1]$ and is
		$1/\alpha$-Lipschitz. In particular,
		\[
		\bigl|
		\chi_{z,\alpha}(x)-\chi_{z,\alpha}(y)
		\bigr|
		\leq
		|x-y|/\alpha,
		\qquad x,y\in\mathbb{R}.
		\]
		
		\item[(ii)]
		For every $t\in\mathbb{R}$,
		\[
		0\leq F_{z,\alpha}(t)\leq 2\alpha.
		\]
		
		\item[(iii)]
For every $x,y\in\mathbb{R}$,
		\[
		\begin{aligned}
			&\bigl|
			\chi_{z,\alpha}(x+y)
			-
			\chi_{z,\alpha}(x)
			\bigr| \leq
			\frac{|y|}{\alpha}
			\int_{0}^{1}
			\bigl[
			\chi_{z-\alpha-ty,\alpha}(x)
			+
			\chi_{z+\alpha-ty,\alpha}(x)
			\bigr]
			\,\mathrm{d}t.
		\end{aligned}
		\]
	\end{enumerate}
\end{lemma}
\begin{proof}
	By definition, $\chi_{z,\alpha}$ takes values in $[0,1]$ and is
	piecewise linear, with slopes belonging to
	$\{-1/\alpha,0,1/\alpha\}$. Hence it is $1/\alpha$-Lipschitz,
	which proves {\rm (i)}.
	Since $\chi_{z,\alpha}\geq0$, the function $F_{z,\alpha}$ is
	nonnegative. Moreover,
$
	\int_{\mathbb{R}}\chi_{z,\alpha}(u)\,\mathrm{d}u
	=
	\alpha/2+\alpha+\alpha/2
	=
	2\alpha.
$
	Therefore, for each $t\in \R$, 
$
	0\leq F_{z,\alpha}(t)\leq 2\alpha
$
	and {\rm (ii)} follows.
	Finally, $\chi_{z,\alpha}$ is absolutely continuous and, for
	almost every $x\in\mathbb{R}$,
$
	|\chi_{z,\alpha}'(x)|
	=
[
	\mathbf{1}_{(z-\alpha,z)}(x)
	+
	\mathbf{1}_{(z+\alpha,z+2\alpha)}(x)
	]/\alpha.
$
Combining this with the inequalities  $	\mathbf{1}_{(z-\alpha,z)}(x)
\leq \chi_{z-\alpha,\alpha}(x)$ and  $
	\mathbf{1}_{(z+\alpha,z+2\alpha)}(x)
	\leq \chi_{z+\alpha,\alpha}(x),$ we have
\begin{align}\label{eq-ua-02}
		|\chi_{z,\alpha}'(x)|
	\leq
	\frac{1}{\alpha}
	\bigl[
	\chi_{z-\alpha,\alpha}(x)
	+
	\chi_{z+\alpha,\alpha}(x)
	\bigr].
\end{align}
Now, for any \(x,y\in\mathbb{R}\), by the absolute continuity of \(\chi_{z,\alpha}\) and \eqref{eq-ua-02}, we have
	\begin{align*}
			\bigl|	\chi_{z,\alpha}(x+y)-\chi_{z,\alpha}(x)	\bigl|
		&=
		\Big|y\int_{0}^{1}\chi_{z,\alpha}'(x+ty)\,\mathrm{d}t\Big|\\
		&\leq 	\frac{|y|}{\alpha}
		\int_{0}^{1}
		\bigl[
		\chi_{z-\alpha,\alpha}(x+ty)
		+
		\chi_{z+\alpha,\alpha}(x+ty)
		\bigr]
		\,\mathrm{d}t\\
		&\leq \frac{|y|}{\alpha}
		\int_{0}^{1}
		\bigl[
		\chi_{z-\alpha-ty,\alpha}(x)
		+
		\chi_{z+\alpha-ty,\alpha}(x)
		\bigr]
		\,\mathrm{d}t,
	\end{align*}
where the last inequality follows from  the translation identity
$
	\chi_{w,\alpha}(x+ty)
	=
	\chi_{w-ty,\alpha}(x).
$
	This completes the proof.
\end{proof}

For $\alpha>0$, we consider the quantity
\begin{equation}\label{eq-af-08}
\mathfrak{C}^{[a]}_{H,m}(\alpha)
 =\mathbb{E}\Bigl[\sup_{z\in\mathbb{R}}
       \mathbb{E}\bigl(\chi_{z,\alpha}(V)\mid\mathscr{G}\bigr)\Bigr].
\end{equation}
Since the conditional expectation is Lipschitz in $z$,
with constant $1/\alpha$, the supremum in \eqref{eq-af-08}
can be taken over rational $z$ and is therefore measurable. 

To derive a recursive bound for
$\mathfrak{C}^{[a]}_{H,m}(\alpha)$,
we first establish an estimate for handling
coordinate replacements.
\begin{lemma}
	\label{lem-protected}
	 Let $H\geq 64m$ and
$
	N=m+H-1.
$
	Let $\mathcal{I}\subseteq\{1,2,\ldots\}$
	be the union of at most two deterministic finite intervals of integers
	(the empty set is allowed), such that
$
	|\mathcal{I}|\leq 6m.
$
	Let $\mathscr{R}$ be a $\sigma$-field independent of
	$\sigma(g_{k}:k\geq1)$, and let $Z,D$ be real-valued
	random variables measurable with respect to
$
	\sigma(
	g_{k}:k\in\mathcal{I}
	)\vee\mathscr{R},
$
	with $Z\geq0$.
	Then the following statements hold.
	\begin{enumerate}
		\item[(i)]
		There exist deterministic integers $p\leq q$, depending only on
		$H,m,\mathcal{I}$, such that
	$
		q-p+1=\lfloor H/16\rfloor
$
		and
		\[
		\{p-m+1,\ldots,q+m-1\}
		\subseteq
		\{m,\ldots,N\}
		\setminus\mathcal{I}.
		\]
		\item[(ii)] Let $	\mathcal{J}:=\{p,\ldots,q\}$.  Define
	$
		\mathscr{G}_{\mathcal{J}}
		:=
		\sigma\bigl(
		g_{p-m+1},\ldots,g_{p-1},
		g_{q+1},\ldots,g_{q+m-1}
		\bigr)
	$
		and
		\[
		V'
		:=
		\sum_{k=p}^{q+m-1}Y_{k}
		-
		\mathbb{E}\Bigl[
		\sum_{k=p}^{q+m-1}Y_{k}
		\,\Bigm|\,
		\mathscr{G}_{\mathcal{J}}
		\Bigr].
		\]
	Let
	$
		\mathscr{F}_{\mathcal{J}}^{\mathrm{out}}
		:=
		\sigma\bigl(
		g_{k}:k\geq1,\ k\notin\mathcal{J}
		\bigr)\vee\mathscr{R}.
$
		Then, for every $\alpha>0$, every real-valued
		$\mathscr{F}_{\mathcal{J}}^{\mathrm{out}}$-measurable random
		variable $T$, and every real-valued $\mathscr{G}$-measurable
		random variable $z$, we have
		\begin{equation}\label{eq-af-09}
			\mathbb{E}\Bigl[
			Z\bigl|
			\chi_{z,\alpha}(V'+T+D)
			-
			\chi_{z,\alpha}(V'+T)
			\bigr|
			\Bigr]
			\leq
			\frac{2}{\alpha}
			\mathfrak{C}^{[a]}_{\lfloor H/16\rfloor,m}(\alpha)
			\mathbb{E}[Z|D|],
		\end{equation}
		where $\mathscr{G}$ is defined   by \eqref{eq-ua-01}.
	\end{enumerate}
\end{lemma}
\begin{proof}
	We first prove  (i).
	Removing $\mathcal{I}$ from $\{m,\ldots,N\}$
	leaves at most three intervals of integers, whose
	total length is at least $H-6m$.
	At least one of these intervals therefore has length
	at least $(H-6m)/3$.
	Since $H\geq 64m$, it follows that
	\[
	\frac{H-6m}{3}
	\geq
	\lfloor H/16\rfloor+2(m-1),
	\qquad
	\lfloor H/16\rfloor\geq m.
	\]
	Choose the leftmost remaining interval whose length
	is at least $\lfloor H/16\rfloor+2(m-1)$, and let
	$b$ be its left endpoint. Set
$
	p:=b+m-1,$
	$q:=p+\lfloor H/16\rfloor-1.
$
	Then
$
	q-p+1=\lfloor H/16\rfloor
$
	and
	\[
	\{p-m+1,\ldots,q+m-1\}
	\subseteq
	\{m,\ldots,N\}\setminus\mathcal{I}.
	\]
	This choice depends only on $H,m,\mathcal{I}$.
	
	We next prove (ii).  By assumption, \(Z,z,T,D\) are \(\mathscr{F}_{\mathcal{J}}^{\mathrm{out}}\)-measurable and hence may be regarded as fixed conditional on \(\mathscr{F}_{\mathcal{J}}^{\mathrm{out}}\). Applying \Cref{lem-chi-properties} (iii)  with $x=V'+T$ and $y=D$,  we have
	\begin{align}\label{eq-ua-03}
		\begin{aligned}
				&\mathbb{E}\Bigl[ Z
		\bigl|
		\chi_{z,\alpha}(V'+T+D)-\chi_{z,\alpha}(V'+T)
		\bigr|
		\,\Bigm|\,
		\mathscr{F}_{\mathcal{J}}^{\mathrm{out}}
		\Bigr]\\
		&\quad\leq \alpha^{-1}\E \Big\{ Z|D| \int_{0}^{1}\big(\chi_{z-\alpha-tD,\alpha}(V'+T)+\chi_{z+\alpha-tD,\alpha}(V'+T) \big) \mathrm{d}t 		\,\Bigm|\,	\mathscr{F}_{\mathcal{J}}^{\mathrm{out}}\Big\}\\
		&\quad\leq 2\alpha^{-1}Z|D|\sup_{\theta\in \R}\E\big[\chi_{\theta,\alpha}(V')\mid \mathscr{F}_{\mathcal{J}}^{\mathrm{out}}\big],
		\end{aligned}
			\end{align}
			where the last inequality follows from the identity $\chi_{z,\alpha}(x+h)=\chi_{z-h,\alpha}(x).$ By the definitions of $V'$,  $\mathscr{F}_{\mathcal{J}}^{\mathrm{out}}$ and $\mathscr{G}_{\mathcal{J}}$, we have, for every
	$\theta\in\mathbb{R}$,
	\[
	\mathbb{E}\bigl[
	\chi_{\theta,\alpha}(V')
	\bigm|
	\mathscr{F}_{\mathcal{J}}^{\mathrm{out}}
	\bigr]
	=
	\mathbb{E}\bigl[
	\chi_{\theta,\alpha}(V')
	\bigm|
	\mathscr{G}_{\mathcal{J}}
	\bigr].
	\]
	Combining this with \eqref{eq-ua-03}, we have
	\begin{align}
		\mathbb{E}\Bigl[ Z
		\bigl|
		\chi_{z,\alpha}(V'+T+D)-\chi_{z,\alpha}(V'+T)
		\bigr|
		\Bigr]\leq \frac{2}{\alpha} \E\Big\{Z|D|\sup_{\theta\in \R}\E\big[\chi_{\theta,\alpha}(V')\mid  	\mathscr{G}_{\mathcal{J}} \big]\Big\}.
	\end{align}
Since the index sets defining $\mathscr{G}_{\mathcal{J}}$
	are disjoint from $\mathcal{I}$, and
	$\mathscr{R}$ is independent of $\sigma(g_{k}, k\geq1)$, it follows that $(Z,D)$ is independent of $\sup_{\theta\in \mathbb{Q}}\E\big[\chi_{\theta,\alpha}(V')\mid  	\mathscr{G}_{\mathcal{J}} \big]$,  where \(\mathbb{Q}\) denotes the set of rational numbers. 
	Since
$
	\theta\mapsto
	\mathbb{E}[
	\chi_{\theta,\alpha}(V')
	\mid
	\mathscr{G}_{\mathcal{J}}
	]
$
	is $1/\alpha$-Lipschitz and $\mathbb{Q}$ is dense in $\mathbb{R}$, we have
$
	\sup_{\theta\in\mathbb{Q}}
	\mathbb{E}[
	\chi_{\theta,\alpha}(V')
	\mid
	\mathscr{G}_{\mathcal{J}}
	]
	=
	\sup_{\theta\in\mathbb{R}}
	\mathbb{E}[
	\chi_{\theta,\alpha}(V')
	\mid
	\mathscr{G}_{\mathcal{J}}
	].
$
Thus, $(Z,D)$ is independent of $\sup_{\theta\in \mathbb{R}}\E[\chi_{\theta,\alpha}(V')\mid  	\mathscr{G}_{\mathcal{J}} ]$.
	Moreover,	by stationarity and \eqref{eq-af-08}, we have
$
	\mathbb{E}\bigl\{
	\sup_{\theta\in\mathbb{R}}
	\mathbb{E}\bigl[
	\chi_{\theta,\alpha}(V')
	\bigm|
	\mathscr{G}_{\mathcal{J}}
	\bigr]
	\bigr\}
	=
	\mathfrak{C}^{[a]}_{\lfloor H/16\rfloor,m}(\alpha).
$
Therefore, 
	\[
	\begin{aligned}
		\mathbb{E}\Bigl[
		Z\bigl|
		\chi_{z,\alpha}(V'+T+D)-\chi_{z,\alpha}(V'+T)
		\bigr|
		\Bigr]
		&\leq 
		\frac{2}{\alpha}
		\mathfrak{C}^{[a]}_{\lfloor H/16\rfloor,m}(\alpha)
		\mathbb{E}[Z|D|].
	\end{aligned}
	\]
	This proves \eqref{eq-af-09}.
\end{proof}

The following proposition gives a recursive estimate relating \(\mathfrak{C}^{[a]}_{H,m}(\alpha)\) to \(\mathfrak{C}^{[a]}_{\lfloor H/16\rfloor,m}(\alpha)\).
\begin{proposition}\label{prop-ramp}
		For  every $\alpha>0$ and
	$H\geq 64m$, we have
	\begin{equation}\label{eq-ag-04}
		\mathfrak{C}^{[a]}_{H,m}(\alpha)
		\leq C\frac{\alpha+\sqrt{1+a^{1-\delta}+m^{1-\delta}}}{\sqrt{H}}
		+\frac{C\bigl(1+a^{1-\delta}+m^{1-\delta}\bigr)}{\alpha}
		\mathfrak{C}^{[a]}_{\lfloor H/16\rfloor,m}(\alpha),
	\end{equation}
	where $C$ is a positive constant depending only on $c^{*}_{1},c^{*}_{2}, s, \mu, \delta, d$ and the fixed norm.  
\end{proposition}

\begin{proof}
	Recall that
	\[
	N=m+H-1,
	\qquad
S_{H,m}=\sum_{i=m}^{N}Y_{i},
	\qquad
	V=S_{H,m}+R.
	\]
	For any $m\leq i\leq N$, let
$
	\widehat{Y}_{i}
	:=\mathbb{E}[Y_{i}^{2}\bigm|\mathscr{F}_{i-1}].
$
Fix $\alpha>0$ and an arbitrary real-valued
$\mathscr{G}$-measurable random variable $z$.
	For the remainder of the argument, fix $m\leq i\leq N$. The replacement at $i$ leaves
	$\mathscr{G}$, and hence $z$, unchanged. Integrating $g_i$ in the first identity below and $g_i^{\prime}$ in the second, and using $\E[Y_{i}\mid \mathscr{F}_{i-1}]=0=\E[Y_{i}^{(i)}\mid \mathscr{F}_{i-1}]$, we have
	\[
	\mathbb{E}\bigl[
	Y_{i}F_{z,\alpha}(V^{(i)})\bigm|\mathscr{G}
	\bigr]=0,
	\qquad
	\mathbb{E}\bigl[
	Y_{i}^{(i)}F_{z,\alpha}(V)\bigm|\mathscr{G}
	\bigr]=0.
	\]
On the other hand, by the i.i.d. property of \(g_i\) and \(g_i'\), we have
	\[
	\mathbb{E}\bigl[
	Y_{i}^{(i)}F_{z,\alpha}(V^{(i)})\bigm|\mathscr{G}
	\bigr]
	=
	\mathbb{E}\bigl[
	Y_{i}F_{z,\alpha}(V)\bigm|\mathscr{G}
	\bigr].
	\]
Combining the above equalities, we have
	\begin{align}	\label{eq-ag-01}
		\begin{aligned}
				\mathbb{E}\bigl[
		Y_{i}F_{z,\alpha}(V)\bigm|\mathscr{G}
		\bigr]
		&=
		\frac{1}{2}\mathbb{E}\Bigl[
		\bigl(Y_{i}-Y_{i}^{(i)}\bigr)
		\bigl\{F_{z,\alpha}(V)-F_{z,\alpha}(V^{(i)})\bigr\}
		\,\Bigm|\,\mathscr{G}
		\Bigr]\\
		&=
		\frac{1}{2}\mathbb{E}\biggl[
		\bigl(Y_{i}-Y_{i}^{(i)}\bigr)\Delta_{i}
		\int_{0}^{1}
		\chi_{z,\alpha}\bigl((1-t)V^{(i)}+tV\bigr)
		\,\mathrm{d}t
		\,\Bigm|\,\mathscr{G}
		\biggr],
		\end{aligned}
	\end{align}
where $V-V^{(i)}=\Delta_i$ by \eqref{eq-af-05}.
	The difference $Y_{j}-Y_{j}^{(i)}$ vanishes unless
	$i\leq j\leq i+m-1$. Hence
	\begin{equation}\label{eq-ag-02}
		\begin{aligned}
			\Delta_{i}
			&=
			\bigl(Y_{i}-Y_{i}^{(i)}\bigr)
			+\sum_{j=i+1}^{\min\{N,i+m-1\}}
			\bigl(Y_{j}-Y_{j}^{(i)}\bigr)
			+R-R^{(i)},\qquad m\leq i\leq N.
		\end{aligned}
	\end{equation}
Combining \eqref{eq-ag-01} and \eqref{eq-ag-02}, we have
	\begin{align}\label{eq-ag-03}
		\mathbb{E}\bigl[\chi_{z,\alpha}(V)\bigm|\mathscr{G}\bigr]=\widehat{I}_{z,1}+\widehat{I}_{z,2} +\widehat{I}_{z,3}+\widehat{I}_{z,4}-\widehat{I}_{z,5}-\widehat{I}_{z,6},
	\end{align}
	where 
	\begin{align*}
		\widehat{I}_{z,1}&=	\frac{1}{Hs_{a,m}^{2}}
		\mathbb{E}\bigl[S_{H,m} F_{z,\alpha}(V)\bigm|\mathscr{G}\bigr],\\
		\widehat{I}_{z,2}&=\mathbb{E}\biggl\{
		\Bigl[1-\frac{1}{Hs_{a,m}^{2}}
		\sum_{i=m}^{N}\widehat{Y}_{i}\Bigr]
		\chi_{z,\alpha}(V)
		\,\Bigm|\,\mathscr{G}
		\biggr\}\\
		\widehat{I}_{z,3}&=	\frac{1}{2Hs_{a,m}^{2}}\sum_{i=m}^{N}
		\mathbb{E}\Bigl\{
		\bigl[2\widehat{Y}_{i}
		-\bigl(Y_{i}-Y_{i}^{(i)}\bigr)^{2}\bigr]
		\chi_{z,\alpha}(V)
		\,\Bigm|\,\mathscr{G}
		\Bigr\},\\
	\widehat{I}_{z,4}&=\frac{1}{2Hs_{a,m}^{2}}\sum_{i=m}^{N}
		\mathbb{E}\biggl\{
		\bigl(Y_{i}-Y_{i}^{(i)}\bigr)^{2}
		\int_{0}^{1}
		\bigl\{\chi_{z,\alpha}(V)
		-\chi_{z,\alpha}\bigl((1-t)V^{(i)}+tV\bigr)\bigr\}
		\,\mathrm{d}t
		\,\Bigm|\,\mathscr{G}
		\biggr\},\\
	\widehat{I}_{z,5}&=\frac{1}{2Hs_{a,m}^{2}}\sum_{i=m}^{N}
		\sum_{j=i+1}^{\min\{N,i+m-1\}}
		\mathbb{E}\biggl\{
		\bigl(Y_{i}-Y_{i}^{(i)}\bigr)
		\bigl(Y_{j}-Y_{j}^{(i)}\bigr)\\
		&\hspace{24mm}\times\int_{0}^{1}
		\chi_{z,\alpha}\bigl((1-t)V^{(i)}+tV\bigr)
		\,\mathrm{d}t
		\,\biggm|\,\mathscr{G}
		\biggr\},\\
		\widehat{I}_{z,6}&=\frac{1}{2Hs_{a,m}^{2}}\sum_{i=m}^{N}
		\mathbb{E}\biggl\{
		\bigl(Y_{i}-Y_{i}^{(i)}\bigr)
		\bigl(R-R^{(i)}\bigr)
		\int_{0}^{1}
		\chi_{z,\alpha}\bigl((1-t)V^{(i)}+tV\bigr)
		\,\mathrm{d}t
		\,\biggm|\,\mathscr{G}
		\biggr\}.
	\end{align*}
For each $k=1,\ldots,6$, 	let $I_{k}=\E |\widehat{I}_{z,k}|$.
In what follows, we derive an upper bound for each of these six terms in turn.

\smallskip
\noindent
{\it (i) Upper bound for $I_{1}$.}	Recalling that  $(Y_{k})_{k\geq m}$ is a martingale difference sequence, we have	$\mathbb{E}(S_{H,m}^{2})=Hs_{a,m}^{2}$.
	Using $|F_{z,\alpha}|\leq2\alpha$ and the  Cauchy--Schwarz inequality, we have
	\begin{align}
		I_{1}&=\frac{1}{Hs_{a,m}^{2}}
		\mathbb{E}\Bigl| 
		\mathbb{E}\bigl[S_{H,m} F_{z,\alpha}(V)\bigm|\mathscr{G}\bigr]
		\Bigr|\leq
		\frac{2\alpha}{Hs_{a,m}^{2}}\mathbb{E}|S_{H,m}|
		\leq
		\frac{2\alpha\sqrt{Hs_{a,m}^{2}}}{Hs_{a,m}^{2}}
		\leq\frac{C\alpha}{\sqrt{H}},
		\label{eq-ag-05}
	\end{align}
	where 
	the last inequality uses the lower bound on $s_{a,m}$.
	
	\smallskip
	\noindent
{\it (ii) Upper bound for $I_{2}$.}
	Using  $0\leq\chi_{z,\alpha}\leq1$
	and \Cref{lem-predictable},  we have
	\begin{align}
		I_{2}&\leq \frac{1}{Hs_{a,m}^{2}}
		\mathbb{E}\biggl|
		\sum_{i=m}^{N}
		\bigl(\widehat{Y}_{i}-s_{a,m}^{2}\bigr)
		\biggr|\leq
		C\sqrt{\frac{1+a^{1-\delta}+m^{1-\delta}}{H}}.
		\label{eq-ag-06}
	\end{align}
	
	\smallskip
\noindent
{\it (iii) Upper bound for $I_{3}$.}
	It follows from the definitions of $Y_{i}$, $Y_{i}^{(i)}$ and $\widehat{Y}_{i}$ that 
	\[
	\begin{aligned}
		\mathbb{E}\bigl[
		Y_{i}^{(i)}
		\,\bigm|\,
		\sigma(g_{k}:k\geq1)
		\bigr]
		=0\quad\text{and}\quad
		\mathbb{E}\bigl[
		\bigl(Y_{i}^{(i)}\bigr)^{2}
		\,\bigm|\,
		\sigma(g_{k}:k\geq1)
		\bigr]
		=\widehat{Y}_{i}.
	\end{aligned}
	\]
Expanding \((Y_i-Y_i^{(i)})^2\) and using the preceding identities, we have
	\[
	\begin{aligned}
		&\mathbb{E}\Bigl[
		\Bigl\{2\widehat{Y}_{i}
		-\bigl(Y_{i}-Y_{i}^{(i)}\bigr)^{2}\Bigr\}
		\chi_{z,\alpha}(V)
		\,\Bigm|\,\mathscr{G}
		\Bigr]=
		\mathbb{E}\Bigl[
		\Bigl\{\bigl(Y_{i}^{(i)}\bigr)^{2}-Y_{i}^{2}\Bigr\}
		\chi_{z,\alpha}(V)
		\,\Bigm|\,\mathscr{G}
		\Bigr].
	\end{aligned}
	\]
By the conditional exchangeability of $g_{i}$ and
$g'_{i}$ given $\mathscr{G}$, exchanging these two variables
changes the sign of $(Y_i^{(i)})^2-Y_i^2$ and replaces $V$ by
$V^{(i)}$. Averaging the two resulting representations gives
	\begin{align}	\label{eq-ag-07}
		\begin{aligned}
					&\mathbb{E}\Bigl[
			\Bigl\{2\widehat{Y}_{i}
			-\bigl(Y_{i}-Y_{i}^{(i)}\bigr)^{2}\Bigr\}
			\chi_{z,\alpha}(V)
			\,\Bigm|\,\mathscr{G}
			\Bigr]
		\\
			&\quad=\frac{1}{2}\mathbb{E}\Bigl[
			\Bigl\{\bigl(Y_{i}^{(i)}\bigr)^{2}-Y_{i}^{2}\Bigr\}
			\bigl\{\chi_{z,\alpha}(V)-\chi_{z,\alpha}(V^{(i)})\bigr\}
			\,\Bigm|\,\mathscr{G}
			\Bigr].
		\end{aligned}
	\end{align}
We apply \Cref{lem-protected} to bound the above term.  To this end, let 
$
\mathcal{I}_{1}
:=
\{i-m+1,\ldots,i+m-1\},
$
so that $|\mathcal{I}_{1}|=2m-1$.
Choose deterministic integers $p_{1}\leq q_{1}$ such that
$
q_{1}-p_{1}+1=\lfloor H/16\rfloor
$
and
\[
\{p_{1}-m+1,\ldots,q_{1}+m-1\}
\subseteq
\{m,\ldots,N\}\setminus\mathcal{I}_{1}.
\]
Set
$
\mathcal{J}_{1}:=\{p_{1},\ldots,q_{1}\}.
$
Define
$
\mathscr{G}_{\mathcal{J}_{1}}
:=
\sigma\bigl(
g_{p_{1}-m+1},\ldots,g_{p_{1}-1},
g_{q_{1}+1},\ldots,g_{q_{1}+m-1}
\bigr)
$
and
\begin{align}\label{eq-ua-04}
V_{1}'
:=
\sum_{k=p_{1}}^{q_{1}+m-1}Y_{k}
-
\mathbb{E}\biggl[
\sum_{k=p_{1}}^{q_{1}+m-1}Y_{k}
\,\biggm|\,
\mathscr{G}_{\mathcal{J}_{1}}
\biggr].
\end{align}
Let
$
\mathscr{F}_{\mathcal{J}_{1}}^{\mathrm{out}}
:=
\sigma\bigl(
g_{k}:k\geq1,\ k\notin\mathcal{J}_{1}
\bigr)
\vee\sigma(g_{i}^{\prime}).
$
It follows from the definitions of  $V^{(i)}$ and $V_{1}'$ that 
\begin{align*}
V^{(i)}-V_{1}'&=\sum_{k=m}^{N+m-1}Y_{k}^{(i)}-\sum_{k=p_{1}}^{q_{1}+m-1}Y_{k}
-\mathbb{E}\Bigl[\sum_{k=m}^{N+m-1}Y_{k}\Bigm|\mathscr{G}\Bigr]+
\mathbb{E}\biggl[
\sum_{k=p_{1}}^{q_{1}+m-1}Y_{k}
\,\biggm|\,
\mathscr{G}_{\mathcal{J}_{1}}
\biggr]\\
&=\sum_{k=m}^{p_{1}-1}Y_{k}^{(i)}+\sum_{k=q_{1}+m}^{N+m-1}Y_{k}^{(i)}
-\mathbb{E}\Bigl[\sum_{k=m}^{N+m-1}Y_{k}\Bigm|\mathscr{G}\Bigr]+
\mathbb{E}\biggl[
\sum_{k=p_{1}}^{q_{1}+m-1}Y_{k}
\,\biggm|\,
\mathscr{G}_{\mathcal{J}_{1}}
\biggr].
\end{align*}
Hence $V^{(i)}-V_{1}'$ 
is $\mathscr{F}_{\mathcal{J}_{1}}^{\mathrm{out}}$-measurable. Moreover, by \eqref{eq-af-05}, $V-V^{(i)}$ is measurable with respect to $\sigma(g_{k}, k\in \mathcal{I}_{1})\vee \sigma(g'_{i})$.
Applying \eqref{eq-af-09} with
$
Z=|Y_{i}^{2}-(Y_{i}^{(i)})^{2}|,$
$
T=V^{(i)}-V_{1}',$
$
D=V-V^{(i)}$,
we have
\[
\begin{aligned}
	&\mathbb{E}\Bigl[
	\bigl|Y_{i}^{2}-\bigl(Y_{i}^{(i)}\bigr)^{2}\bigr|
	\bigl|\chi_{z,\alpha}(V)-\chi_{z,\alpha}\bigl(V^{(i)}\bigr)\bigr|
	\Bigr]\leq
	\frac{2}{\alpha}
	\mathfrak{C}^{[a]}_{\lfloor H/16\rfloor,m}(\alpha)
	\mathbb{E}\Bigl[
	\bigl|Y_{i}^{2}-\bigl(Y_{i}^{(i)}\bigr)^{2}\bigr|
	|\Delta_{i}|
	\Bigr].
\end{aligned}
\]
Combining this bound with \eqref{eq-ag-07}, the triangle
inequality, and conditional Jensen's inequality, we obtain
\begin{align}	\label{eq-ag-08}
	\begin{aligned}
		&
	\frac{1}{2Hs_{a,m}^{2}}\mathbb{E}
	\biggl|\sum_{i=m}^{N}
	\mathbb{E}\Bigl[
	\Bigl\{2\widehat{Y}_{i}
	-\bigl(Y_{i}-Y_{i}^{(i)}\bigr)^{2}\Bigr\}
	\chi_{z,\alpha}(V)
	\,\Bigm|\,\mathscr{G}
	\Bigr]
	\biggr|
\\
	&\quad\leq
	\frac{C}{H\alpha}
	\mathfrak{C}^{[a]}_{\lfloor H/16\rfloor,m}(\alpha)
	\sum_{i=m}^{N}
	\mathbb{E}\Bigl[
	\bigl|Y_{i}^{2}-\bigl(Y_{i}^{(i)}\bigr)^{2}\bigr|
	|\Delta_{i}|
	\Bigr]
\\
	&\quad\leq
	\frac{C\bigl(1+a^{1-\delta}+m^{1-\delta}\bigr)}{\alpha}
	\mathfrak{C}^{[a]}_{\lfloor H/16\rfloor,m}(\alpha),
	\end{aligned}
\end{align}
where 
the last inequality follows from \eqref{eq-ae-10}
and the identity $N-m+1=H$.
	
	\smallskip
	\noindent
{\it (iv) Upper bound for $I_{4}$.}
For each fixed central index $m\leq i\leq N$, use the same
deterministic intervals $\mathcal{I}_1$ and $\mathcal{J}_1$
as chosen for this index in the derivation of the upper
bound for $I_{3}$. Let $V_{1}'$ be defined by \eqref{eq-ua-04}.
For each $t\in[0,1]$, we shall  apply \eqref{eq-af-09} with
\[
\begin{aligned}
	Z=\bigl(Y_{i}-Y_{i}^{(i)}\bigr)^{2},\quad
	T=V^{(i)}-V_{1}'+t\Delta_{i},\quad
	D=(1-t)\Delta_{i}.
\end{aligned}
\]
As established above, $V^{(i)}-V_{1}'$ is
$\mathscr{F}_{\mathcal{J}_{1}}^{\mathrm{out}}$-measurable.
The variable $\Delta_{i}$ depends only on the matrices $(g_{k})_{k\in\mathcal{I}_{1}}$ and the copy
$g_{i}^{\prime}$. Since
$\mathcal{I}_{1}\cap\mathcal{J}_{1}=\varnothing$, it is also
$\mathscr{F}_{\mathcal{J}_{1}}^{\mathrm{out}}$-measurable.
Thus $T$ has the required measurability.
Using $\Delta_{i}=V-V^{(i)}$, we obtain
$
V_{1}'+T=(1-t)V^{(i)}+tV,$ and 
$
V_{1}'+T+D=V.
$
Consequently, \eqref{eq-af-09} gives
\begin{align}\label{eq-ua-05}
	\begin{aligned}
		&\mathbb{E}\Bigl[
		\bigl(Y_{i}-Y_{i}^{(i)}\bigr)^{2}
		\bigl|
		\chi_{z,\alpha}(V)
		-\chi_{z,\alpha}\bigl((1-t)V^{(i)}+tV\bigr)
		\bigr|
		\Bigr]\\
		&\quad\leq
		\frac{2}{\alpha}
		\mathfrak{C}^{[a]}_{\lfloor H/16\rfloor,m}(\alpha)
		\mathbb{E}\Bigl[
		\bigl(Y_{i}-Y_{i}^{(i)}\bigr)^{2}
		(1-t)|\Delta_{i}|
		\Bigr].
	\end{aligned}
\end{align}
By \eqref{eq-ua-05} and the triangle inequality, we have
\begin{align}		\label{eq-ag-09}
	\begin{aligned}
			I_{4}&\leq \frac{C}{H\alpha}
		\mathfrak{C}^{[a]}_{\lfloor H/16\rfloor,m}(\alpha)
		\sum_{i=m}^{N}\int_{0}^{1}
		\mathbb{E}\Bigl[
		\bigl(Y_{i}-Y_{i}^{(i)}\bigr)^{2}
		(1-t)|\Delta_{i}|
		\Bigr]\,\mathrm{d}t
		\\
		&\leq
		\frac{C\bigl(1+a^{1-\delta}+m^{1-\delta}\bigr)}{\alpha}
		\mathfrak{C}^{[a]}_{\lfloor H/16\rfloor,m}(\alpha),
	\end{aligned}
\end{align}
where 
the last inequality follows from \eqref{eq-ae-10} and the identity $N-m+1=H$.
	
	\smallskip
\noindent
{\it (v) Upper bound for $I_{5}$.}
	Fix $i<j\leq\min\{N,i+m-1\}$. Let \(V^{(i,j)}\) denote the version of \(V\) obtained by replacing
	\((g_{i},g_{j})\) with
	\((g_{i}^{\prime},g_{j}^{\prime})\), while leaving all other variables unchanged.  Integrating with respect to $g_j$ and
	$g_j^{\prime}$, respectively, and applying
	\eqref{eq-hh-01}, we obtain
	\[
	\begin{aligned}
		&\mathbb{E}\biggl[
		\bigl(Y_{i}-Y_{i}^{(i)}\bigr)
		\bigl(Y_{j}-Y_{j}^{(i)}\bigr)
		\int_{0}^{1}
		\chi_{z,\alpha}\bigl((1-t)V^{(i,j)}+tV^{(j)}\bigr)
		\,\mathrm{d}t
		\,\biggm|\,\mathscr{G}
		\biggr]=0,\\
		&\mathbb{E}\biggl[
		\bigl(Y_{i}-Y_{i}^{(i)}\bigr)
		\bigl(Y_{j}^{(j)}-Y_{j}^{(i,j)}\bigr)
		\int_{0}^{1}
		\chi_{z,\alpha}\bigl((1-t)V^{(i)}+tV\bigr)
		\,\mathrm{d}t
		\,\biggm|\,\mathscr{G}
		\biggr]=0.
	\end{aligned}
	\]
By the exchangeability of \(g_{j}\) and \(g_{j}^{\prime}\), together with the two identities above, we obtain
	\begin{align}	\label{eq-ag-10}
		\begin{aligned}
			&\mathbb{E}\biggl[
		\bigl(Y_{i}-Y_{i}^{(i)}\bigr)
		\bigl(Y_{j}-Y_{j}^{(i)}\bigr)
		\int_{0}^{1}
		\chi_{z,\alpha}\bigl((1-t)V^{(i)}+tV\bigr)
		\,\mathrm{d}t
		\,\biggm|\,\mathscr{G}
		\biggr]
	\\
		&\quad=\frac{1}{2}\mathbb{E}\biggl[
		\bigl(Y_{i}-Y_{i}^{(i)}\bigr)\Gamma_{ij}
		\int_{0}^{1}
		\Bigl\{
		\chi_{z,\alpha}\bigl((1-t)V^{(i)}+tV\bigr)
		-\chi_{z,\alpha}\bigl((1-t)V^{(i,j)}+tV^{(j)}\bigr)
		\Bigr\}\,\mathrm{d}t
		\,\biggm|\,\mathscr{G}
		\biggr],
		\end{aligned}
	\end{align}
	where we recall that $	\Gamma_{ij}
	:=Y_{j}-Y_{j}^{(i)}-Y_{j}^{(j)}+Y_{j}^{(i,j)}.$ 
	Since \(i\) and \(j\) are central indices, the four versions \(V\), \(V^{(i)}\), \(V^{(j)}\), and \(V^{(i,j)}\) have the same \(\mathscr G\)-measurable centering term. Thus
	\begin{align}
		&\bigl\{(1-t)V^{(i)}+tV\bigr\}
		-\bigl\{(1-t)V^{(i,j)}+tV^{(j)}\bigr\}=(1-t)\Delta_{j}^{(i)}+t\Delta_{j},
		\label{eq-ag-11}
	\end{align}
	where \(\Delta_{j}^{(i)}\) is defined in \Cref{cor-mixed}. Let $\mathcal{I}_{2}=	\{i-m+1,\ldots,j+m-1\}$ and $\mathscr{R}_{1}=\sigma(g'_{i}, g'_{j})$. The  interval has length
	$j-i+2m-1\leq3m-2$.
Moreover, 
	$\bigl|\bigl(Y_{i}-Y_{i}^{(i)}\bigr)\Gamma_{ij}\bigr|$
	and $\Delta_{j}^{(i)}, \Delta_{j}$  are measurable with respect to $\sigma(g_{k}, k\in \mathcal{I}_{2})\vee\mathscr{R}_{1}$.
Choose deterministic integers $p_{2}\leq q_{2}$ such that
$
q_{2}-p_{2}+1=\lfloor H/16\rfloor$ and
$\{p_{2}-m+1,\ldots,q_{2}+m-1\}
\subseteq
\{m,\ldots,N\}\setminus\mathcal{I}_{2}.$
Set $\mathcal{J}_{2}:=\{p_{2},\ldots,q_{2}\}$ and
$
\mathscr{G}_{\mathcal{J}_{2}}
:=
\sigma\bigl(
g_{p_{2}-m+1},\ldots,g_{p_{2}-1},
g_{q_{2}+1},\ldots,g_{q_{2}+m-1}
\bigr).
$
Define
\begin{align}\label{eq-ua-06}
	V_{2}^{\prime}
	:=
	\sum_{k=p_{2}}^{q_{2}+m-1}Y_{k}
	-
	\mathbb{E}\biggl[
	\sum_{k=p_{2}}^{q_{2}+m-1}Y_{k}
	\,\biggm|\,
	\mathscr{G}_{\mathcal{J}_{2}}
	\biggr],
\end{align}
and let
$
\mathscr{F}_{\mathcal{J}_{2}}^{\mathrm{out}}
:=
\sigma(
g_{k}:k\geq1,\ k\notin\mathcal{J}_{2}
)
\vee
\sigma(g_{i}',g_{j}^{\prime}).
$
As in the argument used to estimate \(I_{3}\), the definitions of \(V^{(i,j)}\), \(V^{(j)}\), and \(V\) imply that
$(1-t)V^{(i,j)}+tV^{(j)}-V_{2}'$ is $\mathscr{F}_{\mathcal{J}_{2}}^{\mathrm{out}}$-measurable.
Applying conditional Jensen's inequality to \eqref{eq-ag-10},
taking expectations, and then applying \eqref{eq-af-09}
for each $t\in[0,1]$ with
$Z=|(Y_i-Y_i^{(i)})\Gamma_{ij}|$,
$T=(1-t)V^{(i,j)}+tV^{(j)}-V_2'$, and
$D=(1-t)\Delta_j^{(i)}+t\Delta_j$,
we obtain
\begin{align*}
\mathbb{E}\bigl|\text{RHS of \eqref{eq-ag-10}}\bigr|
&\leq \frac{C}{\alpha}
\mathfrak{C}^{[a]}_{\lfloor H/16\rfloor,m}(\alpha)\mathbb{E}\Bigl[
\bigl|(Y_i-Y_i^{(i)})\Gamma_{ij}\bigr|
\bigl(|\Delta_j^{(i)}|+|\Delta_j|\bigr)\Bigr],
\end{align*}
where we have integrated over $t\in[0,1]$ and used
$|(1-t)\Delta_j^{(i)}+t\Delta_j|
\leq |\Delta_j^{(i)}|+|\Delta_j|$.
Using this inequality and \eqref{eq-ae-11}, we have
\begin{align}	\label{eq-ag-12}
	\begin{aligned}
		I_{5}&\leq  \frac{C}{H\alpha}
	\mathfrak{C}^{[a]}_{\lfloor H/16\rfloor,m}(\alpha)
	\sum_{i=m}^{N}
	\sum_{j=i+1}^{\min\{N,i+m-1\}}
	\mathbb{E}\Bigl[
	\bigl|\bigl(Y_{i}-Y_{i}^{(i)}\bigr)\Gamma_{ij}\bigr|
	\bigl(|\Delta_{j}^{(i)}|+|\Delta_{j}|\bigr)
	\Bigr]
\\
	&\leq
	\frac{C\bigl(1+a^{1-\delta}+m^{1-\delta}\bigr)}{\alpha}
	\mathfrak{C}^{[a]}_{\lfloor H/16\rfloor,m}(\alpha).
	\end{aligned}
\end{align}
	
	\smallskip
\noindent
{\it (vi) Upper bound for $I_{6}$.} We first decompose $\widehat{I}_{z,6}$ as follows: 
\[
\widehat{I}_{z,6}:=\widehat{I}_{z,61}+\widehat{I}_{z,62},
\]
where 
\begin{align*}
	\widehat{I}_{z,61}&:=\frac{1}{2Hs_{a,m}^{2}}\sum_{i=m}^{N}
\mathbb{E}\biggl\{
\bigl(Y_i-Y_i^{(i)}\bigr)\bigl(R-R^{(i)}\bigr)\int_0^1
\bigl[\chi_{z,\alpha}((1-t)V^{(i)}+tV)-\chi_{z,\alpha}(V)\bigr]
\,\mathrm{d}t\,\biggm|\,\mathscr G\biggr\},\\
\widehat{I}_{z,62}&:=\frac{1}{2Hs_{a,m}^{2}}\sum_{i=m}^{N}
	\mathbb{E}\bigl\{
	\bigl(Y_{i}-Y_{i}^{(i)}\bigr)
	\bigl(R-R^{(i)}\bigr)
	\chi_{z,\alpha}\bigl(V\bigr)
	\,\bigm|\,\mathscr{G}
	\bigr\}. 
\end{align*}
We first bound $\E\bigl|\widehat{I}_{z,61}\bigr|$
using \Cref{lem-protected}. Note that $R-R^{(i)}=\sum_{j=N+1}^{i+m-1}(Y_{j}-Y_{j}^{(i)})$, where an empty sum is zero. 
Recalling that
$\mathcal{I}_{1}=\{i-m+1,\ldots,i+m-1\}$,
we see that both
	$|(Y_{i}-Y_{i}^{(i)}\bigr)(R-R^{(i)})|$
	and $\Delta_{i}$ are measurable with respect to $\sigma(g_{k}, k\in \mathcal{I}_{1})\vee \sigma(g'_{i})$.
Let $V_{1}^{\prime}$ be as defined in \eqref{eq-ua-04}, and set  $T_{1}:=(1-t)V^{(i)}+tV-V_{1}'$.  The random variable $T_{1}$ is
$\mathscr{F}_{\mathcal{J}_{1}}^{\mathrm{out}}$-measurable. 
	Applying \eqref{eq-af-09} with
$
		Z=\bigl|\bigl(Y_{i}-Y_{i}^{(i)}\bigr)(R-R^{(i)})\bigr|,$
		$T=(1-t)V^{(i)}+tV-V_{1}'$ and
		$D=(1-t)\Delta_{i}$, we have
	\begin{align}	\label{eq-ag-13}
		\begin{aligned}
				\E\bigl|\widehat{I}_{z,61}\bigr|&\leq \frac{C}{H\alpha}
		\mathfrak{C}^{[a]}_{\lfloor H/16\rfloor,m}(\alpha)
		\sum_{i=m}^{N}
		\mathbb{E}\Bigl[
		\bigl|Y_{i}-Y_{i}^{(i)}\bigr|
		|R-R^{(i)}|\,|\Delta_{i}|
		\Bigr]
	\\
		&\leq
		\frac{C\bigl(1+a^{1-\delta}+m^{1-\delta}\bigr)}{\alpha}
		\mathfrak{C}^{[a]}_{\lfloor H/16\rfloor,m}(\alpha),
		\end{aligned}
	\end{align}
where 	the second inequality follows from \eqref{eq-af-06}
	and the identity $N-m+1=H$. We next deal with $\widehat{I}_{z,62}$.
	Integrating with respect to $g_{i}^{\prime}$ gives
	\[
	\mathbb{E}\Bigl[
	\bigl(Y_{i}-Y_{i}^{(i)}\bigr)R\chi_{z,\alpha}(V)
	\,\Bigm|\,\mathscr{G}
	\Bigr]
	=
	\mathbb{E}\bigl[
	Y_{i}R\chi_{z,\alpha}(V)\bigm|\mathscr{G}
	\bigr].
	\]
Moreover,
	exchanging $g_{i}$ and $g_{i}^{\prime}$
	also gives
	\[
	\mathbb{E}\Bigl[
	\bigl(Y_{i}-Y_{i}^{(i)}\bigr)
	R^{(i)}\chi_{z,\alpha}(V^{(i)})
	\,\Bigm|\,\mathscr{G}
	\Bigr]
	=-
	\mathbb{E}\bigl[
	Y_{i}R\chi_{z,\alpha}(V)\bigm|\mathscr{G}
	\bigr].
	\]
	Subtracting and rearranging these equalities yields
	\begin{align}
		&\frac{1}{2}\mathbb{E}\Bigl[
		\bigl(Y_{i}-Y_{i}^{(i)}\bigr)
		\bigl(R-R^{(i)}\bigr)\chi_{z,\alpha}(V)
		\,\Bigm|\,\mathscr{G}
		\Bigr]
		\notag\\
		&\quad=
		\mathbb{E}\bigl[
		Y_{i}R\chi_{z,\alpha}(V)\bigm|\mathscr{G}
		\bigr]-
		\frac{1}{2}\mathbb{E}\Bigl[
		\bigl(Y_{i}-Y_{i}^{(i)}\bigr)R^{(i)}
		\bigl\{\chi_{z,\alpha}(V)-\chi_{z,\alpha}(V^{(i)})\bigr\}
		\,\Bigm|\,\mathscr{G}
		\Bigr].
		\label{eq-ag-14}
	\end{align}
	For the first term on the right of \eqref{eq-ag-14},  	by conditional Jensen's inequality, $|\chi_{z,\alpha}|\leq1$,
	and  the Cauchy--Schwarz inequality, we have
	\begin{align}
		\frac{1}{Hs_{a,m}^{2}}
		\mathbb{E}\Bigl|
		\mathbb{E}\Bigl[\sum_{i=m}^{N} Y_{i}R\chi_{z,\alpha}(V)\Bigm|\mathscr{G}\Bigr]
		\Bigr|&\leq
		\frac{C}{H}\mathbb{E}|S_{H,m}R|
		\leq
		\frac{C}{H}
		\sqrt{\mathbb{E}(S_{H,m}^{2})\mathbb{E}(R^{2})}
		\notag\\
		&\leq
		C\sqrt{\frac{1+a^{1-\delta}+m^{1-\delta}}{H}},
		\label{eq-ag-15}
	\end{align}
	where the  last inequality follows from $\mathbb{E}(S_{H,m}^{2})=Hs_{a,m}^{2}$
	and the unconditional consequence of \eqref{eq-af-04}. For the second term on the right of \eqref{eq-ag-14}, set
\[
\mathcal I_R:=\{N-m+2,\ldots,N+m-1\},
\qquad
\mathcal I_3:=\mathcal I_1\cup\mathcal I_R.
\]
The variables $Y_{N+1},\ldots,Y_{N+m-1}$ and the functions $h_{r,m}$ all depend only on the matrices $(g_{k})_{k\in\mathcal{I}_{R}}$. Thus
$Z=|(Y_i-Y_i^{(i)})R^{(i)}|$ and $D=\Delta_i$ are measurable with respect to
$\sigma(g_k:k\in\mathcal I_3)\vee\sigma(g_i^{\prime})$.
The set $\mathcal I_3$ is the union of at most two intervals and
$|\mathcal I_3|\leq(2m-1)+(2m-2)=4m-3\leq6m$.
By \Cref{lem-protected}, choose deterministic $p_3\leq q_3$ with
$q_3-p_3+1=\lfloor H/16\rfloor$ and
\[
\{p_3-m+1,\ldots,q_3+m-1\}
\subseteq\{m,\ldots,N\}\setminus\mathcal I_3.
\]
Let $\mathcal J_3=\{p_3,\ldots,q_3\}$, and define
$\mathscr G_{\mathcal J_3}$, $V_3^{\prime}$ and
$\mathscr F_{\mathcal J_3}^{\mathrm{out}}$ as in \Cref{lem-protected}, using $p_3,q_3$ and $\mathscr R=\sigma(g_i^{\prime})$.  Since $Y_{k}^{(i)}=Y_{k}$ for
$p_{3}\leq k\leq q_{3}+m-1$, we have
\begin{align*}
	V^{(i)}-V_{3}'
	&=\sum_{k=m}^{N+m-1}Y_{k}^{(i)}
	-\sum_{k=p_{3}}^{q_{3}+m-1}Y_{k}
	-\mathbb{E}\Bigl[
	\sum_{k=m}^{N+m-1}Y_{k}\Bigm|\mathscr{G}
	\Bigr]
	+\mathbb{E}\Bigl[
	\sum_{k=p_{3}}^{q_{3}+m-1}Y_{k}
	\Bigm|\mathscr{G}_{\mathcal{J}_{3}}
	\Bigr]\\
	&=\sum_{k=m}^{p_{3}-1}Y_{k}^{(i)}
	+\sum_{k=q_{3}+m}^{N+m-1}Y_{k}^{(i)}
	-\mathbb{E}\Bigl[
	\sum_{k=m}^{N+m-1}Y_{k}\Bigm|\mathscr{G}
	\Bigr]
	+\mathbb{E}\Bigl[
	\sum_{k=p_{3}}^{q_{3}+m-1}Y_{k}
	\Bigm|\mathscr{G}_{\mathcal{J}_{3}}
	\Bigr].
\end{align*}
Hence, $T=V^{(i)}-V_{3}'$ is
$\mathscr{F}_{\mathcal{J}_{3}}^{\mathrm{out}}$-measurable. 
Applying \eqref{eq-af-09} with this $T$ and the preceding $Z,D$, then using \eqref{eq-af-06}, gives
\begin{align}\label{eq-ag-15.1}\begin{aligned}
	&	\frac{1}{2Hs_{a,m}^{2}}\E\Big\{ \sum_{i=m}^{N} \mathbb{E}\Bigl[
	\big|	\bigl(Y_{i}-Y_{i}^{(i)}\bigr)R^{(i)}
	\bigl\{\chi_{z,\alpha}(V)-\chi_{z,\alpha}(V^{(i)})\bigr\}\big|
	\,\Bigm|\,\mathscr{G}
	\Bigr]\Big\}\\
	&\quad\leq \frac{C}{H\alpha}\mathfrak{C}^{[a]}_{\lfloor H/16\rfloor,m}(\alpha)\sum_{i=m}^{N}\E \big\{ \bigl|\bigl(Y_{i}-Y_{i}^{(i)}\bigr)R^{(i)}\bigr||\Delta_{i}|   \big\}\\
	&\quad\leq 	\frac{C\bigl(1+a^{1-\delta}+m^{1-\delta}\bigr)}{\alpha}
	\mathfrak{C}^{[a]}_{\lfloor H/16\rfloor,m}(\alpha),
	\end{aligned}
\end{align}
where 	the last inequality follows from \eqref{eq-af-06}
and  $N-m+1=H$.
	Combining \eqref{eq-ag-13}, 
	\eqref{eq-ag-15} and \eqref{eq-ag-15.1}, we have 
	\begin{equation}\label{eq-ag-16}
	I_{6}\leq 	C\sqrt{\frac{1+a^{1-\delta}+m^{1-\delta}}{H}}
		+\frac{C\bigl(1+a^{1-\delta}+m^{1-\delta}\bigr)}{\alpha}
		\mathfrak{C}^{[a]}_{\lfloor H/16\rfloor,m}(\alpha).
	\end{equation}
	
Taking expectations in \eqref{eq-ag-03}, applying the
triangle inequality, and using \eqref{eq-ag-05}, \eqref{eq-ag-06}, \eqref{eq-ag-08}, \eqref{eq-ag-09}, \eqref{eq-ag-12}, \eqref{eq-ag-16},   we obtain
\begin{align}\label{eq-ag-16a}
\begin{aligned}
	\mathbb{E}\bigl[\chi_{z,\alpha}(V)\bigr]
	&\leq
	\sum_{k=1}^{6}
	\mathbb{E}\bigl|\widehat{I}_{z,k}\bigr|\\
	&\leq
	C\frac{\alpha+\sqrt{1+a^{1-\delta}+m^{1-\delta}}}{\sqrt{H}}+
	\frac{C\bigl(1+a^{1-\delta}+m^{1-\delta}\bigr)}{\alpha}
	\mathfrak{C}^{[a]}_{\lfloor H/16\rfloor,m}(\alpha).
\end{aligned}
\end{align}
We remark that this bound holds for every real-valued
$\mathscr{G}$-measurable random variable $z$, with the
same constant $C$.

To recover the conditional supremum, fix an enumeration
$q_{1},q_{2},\ldots$ of $\mathbb{Q}$.
For each $n\geq1$, let $z_{n}$ be the first element of
$q_{1},\ldots,q_{n}$ at which
$
\max_{1\leq j\leq n}
\mathbb{E}[
\chi_{q_{j},\alpha}(V)\bigm|\mathscr{G}
]
$
is attained. This defines a $\mathscr{G}$-measurable
random variable, and
\[
\mathbb{E}\bigl[
\chi_{z_{n},\alpha}(V)\bigm|\mathscr{G}
\bigr]
=
\max_{1\leq j\leq n}
\mathbb{E}\bigl[
\chi_{q_{j},\alpha}(V)\bigm|\mathscr{G}
\bigr].
\]
The conditional expectation is Lipschitz in the
threshold. Consequently,
\[
\max_{1\leq j\leq n}
\mathbb{E}\bigl[
\chi_{q_{j},\alpha}(V)\bigm|\mathscr{G}
\bigr]
\uparrow
\sup_{q\in\mathbb{Q}}
\mathbb{E}\bigl[
\chi_{q,\alpha}(V)\bigm|\mathscr{G}
\bigr]=\sup_{q\in\mathbb{R}}
\mathbb{E}\bigl[
\chi_{q,\alpha}(V)\bigm|\mathscr{G}
\bigr]
\qquad\text{almost surely}.
\]
The tower property and the monotone convergence theorem
therefore give
\[
\lim_{n\to\infty}
\mathbb{E}\bigl[\chi_{z_{n},\alpha}(V)\bigr]
=
\mathbb{E}\Bigl[
\sup_{q\in\mathbb{R}}
\mathbb{E}\bigl[
\chi_{q,\alpha}(V)\bigm|\mathscr{G}
\bigr]
\Bigr]
=
\mathfrak{C}^{[a]}_{H,m}(\alpha).
\]
Applying \eqref{eq-ag-16a} with $z=z_{n}$
and letting $n\to\infty$, we obtain \eqref{eq-ag-04}.
\end{proof}

Now, we are ready to prove \Cref{prop-profile}.

\begin{proof}[Proof of \Cref{prop-profile}]
Since $0<\delta\leq1$ and $m\geq1$, we have
$m^{1-\delta}\geq1$. Therefore
\[
1+a^{1-\delta}+m^{1-\delta}
 =1+2m^{1-\delta}\leq3m^{1-\delta}
 \quad\hbox{when }a=m.
\]
With $\alpha=L_{*}m^{1-\delta}$ and $L_{*}\geq1$,
the first term on the right of \eqref{eq-ag-04}
is bounded by $C\alpha/\sqrt{H}$, and the coefficient of
$\mathfrak{C}^{[a]}_{\lfloor H/16\rfloor,m}(\alpha)$
is bounded by $C/L_{*}$.
Enlarge $L_{*}$ until this coefficient is at most
$16^{-\eta}$. The resulting recursion is
\begin{equation}\label{eq-ag-18}
	\mathfrak{C}^{[m]}_{H,m}(\alpha)
	\leq \frac{C\alpha}{\sqrt{H}}
	+16^{-\eta}
	\mathfrak{C}^{[m]}_{\lfloor H/16\rfloor,m}(\alpha).
\end{equation}
The parameters $a,m,\alpha$ are kept fixed throughout its
iteration.

To account for integer parts, put $H_{0}=H$ and
$H_{j+1}=\lfloor H_{j}/16\rfloor$. Let $j_{0}$ be the first
index such that $H_{j_{0}}<64m$. Because
$H_{j_{0}-1}\geq 64m$, 
$
H_{j_{0}}\geq 4m-1\geq m.
$
Thus $\mathfrak{C}^{[m]}_{H_j,m}(\alpha)$ is well defined
for every $0\leq j\leq j_0$. Repeated integer
division gives $H_{j}=\lfloor H/16^{j}\rfloor$ exactly.
For $0\leq j\leq j_{0}$, the elementary inequality
$\lfloor x\rfloor\geq x/2$ for $x\geq1$ gives
$
H_{j}\geq H/(2\cdot16^{j}).
$
Iterating \eqref{eq-ag-18} $j_{0}$ times and using
$\mathfrak{C}^{[m]}_{H_{j_{0}},m}(\alpha)\leq1$ yields
\begin{align*}
\mathfrak{C}^{[m]}_{H,m}(\alpha)
 &\leq C\alpha\sum_{j=0}^{j_{0}-1}
            \frac{16^{-j\eta}}{\sqrt{H_{j}}}
          +16^{-j_{0}\eta}
 \leq\frac{C\alpha}{\sqrt{H}}
        \sum_{j=0}^{j_{0}-1}\bigl(4\cdot16^{-\eta}\bigr)^{j}
          +16^{-j_{0}\eta}.
\end{align*}
The geometric series is bounded uniformly in $j_{0}$ because
$\eta>1$ implies $4\cdot16^{-\eta}<1$.
We next bound $16^{-j_{0}\eta}$. Since $H_{j_{0}}<64m$ and
$H_{j_{0}}=\lfloor H/16^{j_{0}}\rfloor$, we have
\[
\frac{H}{16^{j_{0}}}<64m+1\leq 65m,
\quad \text{and}\quad
16^{-j_{0}\eta}\leq 65^{\eta}
                    \biggl(\frac{m}{H}\biggr)^{\eta}.
\]
Consequently, we have
\begin{equation}\label{eq-ag-19}
\mathfrak{C}^{[m]}_{H,m}(\alpha)
 \leq C\biggl\{\frac{\alpha}{\sqrt{H}}
                +\biggl(\frac{m}{H}\biggr)^{\eta}\biggr\}.
\end{equation}

Finally, split $(z,z+2\alpha]$ into $(z,z+\alpha]$ and
$(z+\alpha,z+2\alpha]$. The first interval is contained in the
region where $\chi_{z,\alpha}=1$, while the second is contained in
the region where $\chi_{z+\alpha,\alpha}=1$. Hence, 
\[
\sup_{z\in\mathbb{R}}
 \mathbb{P}\bigl(z<V\leq z+2\alpha\mid\mathscr{G}\bigr)
 \leq2\sup_{z\in\mathbb{R}}
       \mathbb{E}\bigl(\chi_{z,\alpha}(V)\mid\mathscr{G}\bigr).
\]
Taking expectations and applying \eqref{eq-ag-19}, we obtain
\eqref{eq-ag-17}.
\end{proof}

\section{Proof of \texorpdfstring{\Cref{thm-main}}{Theorem 2.1}}\label{sec-proof}

In this section, we prove \Cref{thm-main} using Stein's method,
establishing the bounds \eqref{eq-tt-01}--\eqref{eq-tt-04} in turn.
Throughout, we set $a=m=\lfloor\sqrt{n}\rfloor$ and write
$s_{m}=s_{m,m}$. Without loss of generality, we assume that $n$ is sufficiently
large so that $m\geq 2$ and $n\geq 512m$.  By \eqref{eq-ad-09}, we may also assume
that $c^{*}_{1}\leq s_{m}/s\leq c^{*}_{2}$, where
$c^{*}_{1}=1/2$ and $c^{*}_{2}=2$.
The remaining finitely many values of $n$ are covered by
enlarging the constant in the four bounds, since the
Kolmogorov distance is at most one.

\subsection{Proof of  \protect\eqref{eq-tt-01}}

Write $Y_i=Y^{[m]}_{i,m}$ and set $H=\lfloor n/4\rfloor$.
Since $n\geq512m$, we have
$
H
\geq n/4-1\geq128m-1\geq64m.
$
Recall that
$
S_{n,m}=\sum_{k=m}^{m+n-1}Y_{k}.
$
Since $(Y_k)_{k\geq m}$ is a stationary martingale difference
sequence, we have
\[
\mathbb{E} S_{n,m}=0,\qquad
\operatorname{Var}(S_{n,m})=\sum_{k=m}^{n+m-1}\mathbb{E} Y_{k}^{2}=n s_{m}^{2}.
\]
Choose the constant $L_{*}$ as in \Cref{prop-profile} with $\eta=2$. Let
$\alpha=L_{*}m^{1-\delta}$, and let $U$ be uniform on $[-\alpha,\alpha]$ and
independent of the initial direction, the sequence $(g_{k})_{k\geq1}$, and all
replacement copies. The same $U$ is used for all replacements below.
We first approximate
\[
W=\frac{S_{n,m}+U}{s_{m}\sqrt{n}}.
\]
To this end, we first establish the following lemma using the
concentration inequality \eqref{eq-ag-17}. It controls the change in
a test function of bounded variation when the additive term $x$ is
replaced by $y$, in terms of $\operatorname{TV}(q)$ and
$\mathbb{E}[|Z|\,|x-y|]$.
\begin{lemma}
	\label{lem-smootheddeletion}
	Let $m\leq \ell_{0}\leq r_{0}$ be deterministic integers such that $H^{\dagger}:=r_{0}-\ell_{0}+1\geq64m.$ Set
	$
	\mathcal{J}^{\dagger}:=\{\ell_{0},\ldots,r_{0}\}$, 	$\mathscr{G}_{\mathcal{J}^{\dagger}}
	:=
	\sigma(
	g_{\ell_{0}-m+1},\ldots,g_{\ell_{0}-1},
	g_{r_{0}+1},\ldots,g_{r_{0}+m-1}
).$ Define 
			\[
	V_{\mathcal{J}^{\dagger}}
	:=
	\sum_{k=\ell_{0}}^{r_{0}+m-1}Y_{k}
	-
	\mathbb{E}\biggl[
	\sum_{k=\ell_{0}}^{r_{0}+m-1}Y_{k}
	\,\biggm|\,
	\mathscr{G}_{\mathcal{J}^{\dagger}}
	\biggr]. 
	\]
	Let $\mathscr{R}^{\dagger}$ be a sigma-field independent of
$
	\sigma(W_{0},U,g_{k}:k\geq1),
$
	and set
	\[
	\mathscr{H}
	:=
	\sigma(W_{0})
	\vee
	\sigma\bigl(
	g_{k}:k\geq1,\ k\notin\mathcal{J}^{\dagger}
	\bigr)
	\vee
	\mathscr{R}^{\dagger}.
	\]
	Suppose that $Z$, $x$, $y$, and $T$ are real-valued
	$\mathscr{H}$-measurable random variables such that
	$\sigma(Z,x,y)$ is independent of
	$\mathscr{G}_{\mathcal{J}^{\dagger}}$ and
$
	\mathbb{E}[|Z|\,|x-y|]<\infty.
$
	Then, for every deterministic function
	$q:\mathbb{R}\to\mathbb{R}$ of bounded variation, we have
	\begin{align}	\label{eq-ah-01}
		\begin{aligned}
				&\mathbb{E}\biggl[
		|Z|
		\biggl|
		q\biggl(
		\frac{V_{\mathcal{J}^{\dagger}}+U+T+x}{s_{m}\sqrt{n}}
		\biggr)
		-
		q\biggl(
		\frac{V_{\mathcal{J}^{\dagger}}+U+T+y}{s_{m}\sqrt{n}}
		\biggr)
		\biggr|
		\biggr]
	\\
		&\quad\leq
		C\biggl\{
		\frac{1}{\sqrt{H^{\dagger}}}
		+
		\frac{1}{\alpha}
		\biggl(\frac{m}{H^{\dagger}}\biggr)^{2}
		\biggr\}
		\operatorname{TV}(q)\,
		\mathbb{E}\bigl[|Z|\,|x-y|\bigr],
		\end{aligned}
	\end{align}
	where
 $\operatorname{TV}(q)$ denotes the total variation
	of $q$ on $\mathbb{R}$, and  $C>0$ depends only on
	$c^{*}_{1}$, $c^{*}_{2}$, $s$, $\mu$, $\delta$, $d$, and the fixed norm.
\end{lemma}

\begin{proof}
	Since $U$ is uniform on $[-\alpha,\alpha]$ and independent of
	the sequence $(g_{k})_{k\geq1}$, a conditional density of
	$V_{\mathcal{J}^{\dagger}}+U$ given
	$\mathscr{G}_{\mathcal{J}^{\dagger}}$ is
	\[
	p_{\mathcal{J}^{\dagger}}(u)
	:=
	\frac{1}{2\alpha}
	\mathbb{P}\bigl(
	u-\alpha<V_{\mathcal{J}^{\dagger}}\leq u+\alpha
	\bigm|
	\mathscr{G}_{\mathcal{J}^{\dagger}}
	\bigr).
	\]
Here the conditional probability is defined by fixing
$(g_{\ell_{0}-m+1},\ldots,g_{\ell_{0}-1},
g_{r_{0}+1},\ldots,g_{r_{0}+m-1})$
and integrating over
$(g_{\ell_{0}},\ldots,g_{r_{0}})$
with respect to $\mu^{\otimes H^{\dagger}}$.
This construction makes $p_{\mathcal{J}^{\dagger}}$
jointly measurable in $(g_{\ell_0-m+1},\ldots,g_{\ell_0-1},
g_{r_0+1},\ldots,g_{r_0+m-1})$ and $u$. 	Write $\lVert p_{\mathcal{J}^{\dagger}}\rVert_{\infty}$ for the
essential supremum with respect to Lebesgue measure. 
Since $p_{\mathcal{J}^{\dagger}}$ is right-continuous
in $u$, we have
\[
\lVert p_{\mathcal{J}^{\dagger}}\rVert_{\infty}
=
\sup_{u\in\mathbb{Q}}p_{\mathcal{J}^{\dagger}}(u),
\]
which is $\mathscr{G}_{\mathcal{J}^{\dagger}}$-measurable. Hence,
	\[
	\lVert p_{\mathcal{J}^{\dagger}}\rVert_{\infty}
	\leq
	\frac{1}{2\alpha}
	\sup_{z\in\mathbb{R}}
	\mathbb{P}\bigl(
	z<V_{\mathcal{J}^{\dagger}}\leq z+2\alpha
	\bigm|
	\mathscr{G}_{\mathcal{J}^{\dagger}}
	\bigr).
	\]
	Taking expectations and using stationarity, together with \eqref{eq-ag-17}, we obtain
	\begin{equation}\label{eq-ah-02}
		\mathbb{E}\lVert p_{\mathcal{J}^{\dagger}}\rVert_{\infty}
		\leq
		C\biggl\{
		\frac{1}{\sqrt{H^{\dagger}}}
		+
		\frac{1}{\alpha}
		\biggl(\frac{m}{H^{\dagger}}\biggr)^{2}
		\biggr\}.
	\end{equation}
By construction, $V_{\mathcal{J}^{\dagger}}$ depends only on $(g_{\ell_{0}-m+1},g_{\ell_{0}-m+2},\ldots,g_{r_{0}+m-2},g_{r_{0}+m-1})$.
	Moreover, $\mathscr{G}_{\mathcal{J}^{\dagger}}\subseteq\mathscr{H}$,
	and the vector
	$(g_{\ell_{0}},\ldots,g_{r_{0}},U)$
	is independent of $\mathscr{H}$.
	Thus $p_{\mathcal{J}^{\dagger}}$ is also a conditional
	density of $V_{\mathcal{J}^{\dagger}}+U$ given $\mathscr{H}$.
Since $T,x,y$ are $\mathscr{H}$-measurable, the conditional
density formula yields, for every bounded Borel function
$f:\mathbb{R}^{3}\to\mathbb{R}$, 
\begin{align}\label{eq-ua-08}
\mathbb{E}\Bigl[
f(V_{\mathcal{J}^{\dagger}}+U,T+x,T+y)
\Bigm|\mathscr{H}
\Bigr]
=
\int_{\mathbb{R}}
f(u,T+x,T+y)
p_{\mathcal{J}^{\dagger}}(u)\,\mathrm{d}u.
\end{align}

	We next use the bounded-variation inequality
	\begin{equation}\label{eq-ah-03}
		\int_{\mathbb{R}}
		|q(v+h)-q(v)|\,\mathrm{d}v
		\leq
		|h|\operatorname{TV}(q),
		\qquad h\in\mathbb{R}.
	\end{equation}
	To verify it for $h>0$, partition $\mathbb{R}$ into
	the intervals $[kh,(k+1)h)$, $k\in\mathbb{Z}$.
	For each $r\in[0,h)$, the definition of total
	variation gives
	\[
	\sum_{k\in\mathbb{Z}}
	\bigl|
	q\bigl(r+(k+1)h\bigr)-q(r+kh)
	\bigr|
	\leq
	\operatorname{TV}(q).
	\]
	Tonelli's theorem now yields
	\[
	\begin{aligned}
		\int_{\mathbb{R}}
		|q(v+h)-q(v)|\,\mathrm{d}v
		&=
		\int_{0}^{h}
		\sum_{k\in\mathbb{Z}}
		\bigl|
		q\bigl(r+(k+1)h\bigr)-q(r+kh)
		\bigr|\,\mathrm{d}r\leq
		h\operatorname{TV}(q).
	\end{aligned}
	\]
	The case $h<0$ follows by a change of variable, and
	the case $h=0$ is immediate.
	
	Applying \eqref{eq-ua-08} and using the
	$\mathscr{H}$-measurability of $Z$, we have
\begin{align}\label{eq-ah-04}
	\begin{aligned}
	&\mathbb{E}\biggl[
	|Z|
	\biggl|
	q\biggl(
	\frac{V_{\mathcal{J}^{\dagger}}+U+T+x}{s_{m}\sqrt{n}}
	\biggr)
	-
	q\biggl(
	\frac{V_{\mathcal{J}^{\dagger}}+U+T+y}{s_{m}\sqrt{n}}
	\biggr)
	\biggr|
	\,\biggm|\,
	\mathscr{H}
	\biggr]\\
	&\quad=
	|Z|
	\int_{\mathbb{R}}
	\biggl|
	q\biggl(\frac{u+T+x}{s_{m}\sqrt{n}}\biggr)
	-
	q\biggl(\frac{u+T+y}{s_{m}\sqrt{n}}\biggr)
	\biggr|
	p_{\mathcal{J}^{\dagger}}(u)\,\mathrm{d}u\\
	&\quad\leq
	|Z|\lVert p_{\mathcal{J}^{\dagger}}\rVert_{\infty}
	\int_{\mathbb{R}}
	\biggl|
	q\biggl(\frac{u+T+x}{s_{m}\sqrt{n}}\biggr)
	-
	q\biggl(\frac{u+T+y}{s_{m}\sqrt{n}}\biggr)
	\biggr|\,\mathrm{d}u\\
	&\quad=
	|Z|\lVert p_{\mathcal{J}^{\dagger}}\rVert_{\infty}s_{m}\sqrt{n}
	\int_{\mathbb{R}}
	\biggl|
	q\biggl(v+\frac{x-y}{s_{m}\sqrt{n}}\biggr)-q(v)
	\biggr|\,\mathrm{d}v\\
	&\quad\leq
	|Z|\lVert p_{\mathcal{J}^{\dagger}}\rVert_{\infty}
	|x-y|\operatorname{TV}(q),
\end{aligned}
\end{align}
	where 
	the second equality uses the change of variable
	$v=(u+T+y)/(s_{m}\sqrt{n})$, and the last inequality follows from
	\eqref{eq-ah-03}.
	Finally, $\lVert p_{\mathcal{J}^{\dagger}}\rVert_{\infty}$
	is $\mathscr{G}_{\mathcal{J}^{\dagger}}$-measurable, whereas
	$\sigma(Z,x,y)$ is independent of
	$\mathscr{G}_{\mathcal{J}^{\dagger}}$.
	Therefore,
	\[
	\mathbb{E}\bigl[
	\lVert p_{\mathcal{J}^{\dagger}}\rVert_{\infty}
	|Z|\,|x-y|
	\bigr]
	=
	\mathbb{E}\lVert p_{\mathcal{J}^{\dagger}}\rVert_{\infty}
	\,
	\mathbb{E}\bigl[|Z|\,|x-y|\bigr].
	\]
	Taking expectations in \eqref{eq-ah-04} and using \eqref{eq-ah-02},  we obtain	\eqref{eq-ah-01}.
\end{proof}

Applying \Cref{lem-smootheddeletion} to indicator
functions yields the following bound on the difference
between distribution functions.
\begin{corollary}\label{cor-cdfcomparison}
	In the setting of \Cref{lem-smootheddeletion},
	let $T$ and $D$ be real-valued
	$\mathscr{H}$-measurable random variables such that
	$\mathbb{E}|D|<\infty$ and $\sigma(D)$
	is independent of $\mathscr{G}_{\mathcal{J}^{\dagger}}$.
	Then,
	\begin{equation}\label{eq-ai-02}
		\begin{aligned}
			&\sup_{z\in\mathbb{R}}
			\Bigl|
			\mathbb{P}\bigl(V_{\mathcal{J}^{\dagger}}+T+U\leq z\bigr)
			-\mathbb{P}\bigl(V_{\mathcal{J}^{\dagger}}+T-D+U\leq z\bigr)
			\Bigr|\\
			&\qquad\leq
			C\biggl\{
			\frac{1}{\sqrt{H^{\dagger}}}
			+\frac{1}{\alpha}
			\biggl(\frac{m}{H^{\dagger}}\biggr)^{2}
			\biggr\}
			\mathbb{E}|D|,
		\end{aligned}
	\end{equation}
	where  $C>0$ depends only on
	$c^{*}_{1}$, $c^{*}_{2}$, $s$, $\mu$, $\delta$, $d$, and the fixed norm.
\end{corollary}

\begin{proof}
	For each $z\in\mathbb{R}$, apply
	\Cref{lem-smootheddeletion} with
	$Z=1$, $x=0$, $y=-D$, and
	\[
	q_z(w)=\mathbf{1}_{\{s_m\sqrt{n}\,w\leq z\}}.
	\]
	Since $\operatorname{TV}(q_z)=1$, the result follows
	from $|\mathbb{E}A|\leq\mathbb{E}|A|$ and
	taking the supremum over $z$.
\end{proof}

The following proposition establishes a normal approximation for
$S_{n,m}+U$ and then removes the smoothing term $U$ without changing
the convergence rate.
\begin{proposition}\label{prop-Stein}
Under the same assumptions as in \Cref{thm-main}, we have
	\begin{equation}\label{eq-ah-10}
		{d_{\mathrm{K}}}\biggl(\mathcal{L}\biggl(\frac{S_{n,m}+U}{s_{m}\sqrt{n}}\biggr),
		\mathcal{N}(0,1)\biggr)\le Cn^{-\delta/2},
	\end{equation}
		where  $C>0$ depends only on
	 $\mu$, $\delta$, $d$, and the fixed norm.
	The same estimate holds with $S_{n,m}$ in place of $S_{n,m}+U$.
\end{proposition}
\begin{proof}
	We first keep $n$ and $m$ separate, subject to
$
	n\geq 512m
$
 and $m\geq 2.$ Throughout the remainder of the proof, we assume that $n$ and $m$
 are sufficiently large, so that $s_{m}>0$.
	Recall that
	\[
	Y_{i}=Y_{i,m}^{[m]},\qquad
	S_{n,m}=\sum_{i=m}^{m+n-1}Y_{i},\qquad
	W=\frac{S_{n,m}+U}{s_{m}\sqrt{n}},\qquad  	\widehat{Y}_{i}
	:=\mathbb{E}\bigl[Y_{i}^{2}\bigm|\mathscr{F}_{i-1}\bigr].
	\]
	Fix $z\in\mathbb{R}$, and let $f_{z}$ be the bounded
	solution to the Stein equation
	\[
	f_{z}'(w)-wf_{z}(w)
	=\mathbf{1}_{\{w\leq z\}}-\Phi(z).
	\]
It can be shown (see, e.g., \cite{ChenGoldsteinShao2011}) that
$
\lVert f_z\rVert_\infty\le C,$
$\lVert f_z'\rVert_\infty\le1.
$
Moreover, applying \cite[Theorem~3.4.4(1)]{Raic2020} with
$h(w)=\mathbf{1}_{\{w\leq z\}}$, we have
$
\operatorname{TV}(f_{z}^{\prime})\leq2,
$
where $f_{z}^{\prime}(z)$ is defined as  the left derivative of $f_{z}$ at $z$.

In what follows, let
$(g_{i}^{\prime})_{i\geq1}$
be an independent copy of
$(g_{i})_{i\geq1}$.
The superscript $(i)$ indicates that
$g_{i}$ is replaced by $g_{i}^{\prime}$,
with all other coordinates unchanged, while $(i,j)$ indicates that
$g_{i}$ and $g_{j}$ are replaced by
$g_{i}^{\prime}$ and $g_{j}^{\prime}$, respectively.
The variable $U$ is left unchanged under all such replacements.

	By the exchangeability of $g_{i}$ and
	$g_{i}^{\prime}$, we have, for each $i\in [m,m+n-1]$,
	\begin{align}\label{eq-ah-06}
		\begin{aligned}
			&\mathbb{E}\Big\{Y_{i}f_{z}\Big(\frac{S_{n,m}+U}{s_{m}\sqrt{n}}\Big)\Big\}\\
			&\,\,\,\,\,=\frac{1}{2}\E\Big\{\big(Y_{i}-Y_{i}^{(i)}\big)\Big[f_{z}\Big(\frac{S_{n,m}+U}{s_{m}\sqrt{n}}\Big)-f_{z}\Big(\frac{S_{n,m}^{(i)}+U}{s_{m}\sqrt{n}}\Big)\Big]\Big\}\\
			&\,\,\,\,\,=\frac{1}{2s_{m}\sqrt{n}}\mathbb{E}\biggl\{
			\bigl(Y_{i}-Y_{i}^{(i)}\bigr)\big(	S_{n,m}-S_{n,m}^{(i)}\big)
			\int_{0}^{1}
			f_{z}'\Bigl(
			\frac{(1-t)S_{n,m}^{(i)}+tS_{n,m}+U}{s_{m}\sqrt{n}}
			\Bigr)\,\mathrm{d}t
			\biggr\}.
		\end{aligned}
	\end{align}
	By the definitions of $S_{n,m}$ and $S_{n,m}^{(i)}$, we have,	for each $m\leq i\leq m+n-1$
	\begin{equation}\label{eq-ah-05}
		S_{n,m}-S_{n,m}^{(i)}
		=Y_{i}-Y_{i}^{(i)}
		+\sum_{j=i+1}^{\min\{i+m-1,m+n-1\}}
		\bigl(Y_{j}-Y_{j}^{(i)}\bigr).
	\end{equation}	
Taking expectations in the Stein equation, summing \eqref{eq-ah-06},
and using $s_{m}\sqrt{n}W=S_{n,m}+U$ and \eqref{eq-ah-05}, we have
	\begin{equation}\label{eq-ah-11}
		\mathbb{P}(W\leq z)-\Phi(z)
		=I_{1}+I_{2}+I_{3}-I_{4}
		-\frac{1}{s_{m}\sqrt{n}}\mathbb{E}\bigl[Uf_{z}(W)\bigr],
	\end{equation}
	where
	\begin{align*}
		I_{1}
		&:=\mathbb{E}\biggl\{
		\Bigl(1-\frac{1}{s_{m}^{2}n}
		\sum_{i=m}^{m+n-1}\widehat{Y}_{i}\Bigr)
		f_{z}'(W)\biggr\},\\	
		I_{2}
		&:=\frac{1}{2s_{m}^{2}n}\sum_{i=m}^{m+n-1}
		\mathbb{E}\Bigl\{\Bigl[
		2\widehat{Y}_{i}
		-\bigl(Y_{i}-Y_{i}^{(i)}\bigr)^{2}\Bigr]
		f_{z}'(W)\Bigr\},\\
		I_{3}
		&:=\frac{1}{2s_{m}^{2}n}\sum_{i=m}^{m+n-1}
		\mathbb{E}\bigg\{
		\bigl(Y_{i}-Y_{i}^{(i)}\bigr)^{2}
		\int_{0}^{1}\Big[
		f_{z}'(W)
		-f_{z}'\Bigl(
		\frac{(1-t)S_{n,m}^{(i)}+tS_{n,m}+U}{s_{m}\sqrt{n}}
		\Bigr)\Big]\,\mathrm{d}t
		\biggr\},\\
		I_{4}
		&:=\frac{1}{2s_{m}^{2}n}\sum_{i=m}^{m+n-1}
		\sum_{j=i+1}^{\min\{i+m-1,m+n-1\}}
		\mathbb{E}\bigg\{
		\bigl(Y_{i}-Y_{i}^{(i)}\bigr)
		\bigl(Y_{j}-Y_{j}^{(i)}\bigr)\\
		&\hspace{60mm}\times
		\int_{0}^{1}
		f_{z}'\Big(
		\frac{(1-t)S_{n,m}^{(i)}+tS_{n,m}+U}{s_{m}\sqrt{n}}
		\Big)\,\mathrm{d}t
		\bigg\}.
	\end{align*}
	
	We next estimate $|I_{1}|$,$|I_{2}|,|I_{3}|, |I_{4}|$ and $\E[Uf_{z}(W)]$ in turn. Since we will repeatedly apply \Cref{lem-smootheddeletion} below,
	we first specify a common construction for these applications.
	Set $H^{\dagger}=\lfloor n/4\rfloor$, so that $H^{\dagger}\geq64m$
	and $H^{\dagger}\geq n/8$.
	For fixed indices $	m\leq i\leq j\leq\min\{i+m-1,m+n-1\},$ define
	\[
	\mathcal{I}_{i,j}:=\{i-m+1,\ldots,j+m-1\}.
	\]
We have $|\mathcal{I}_{i,j}|<3m$.
	Removing this interval from $\{m,\ldots,m+n-1\}$
	leaves at most two consecutive intervals, one of
	length at least $(n-3m)/2$.
	Since $n\geq 512m$, it follows that
$
	(n-3m)/2>\lfloor n/4\rfloor+2(m-1).
$
	We can therefore choose an integer $p$, depending
	only on $i,j,n,m$, such that
	\[
	\mathcal{J}^{\dagger}:=\{p,\ldots,p+H^{\dagger}-1\},
	\quad
	\{p-m+1,\ldots,p+H^{\dagger}+m-2\}
	\subseteq
	\{m,\ldots,m+n-1\}\setminus\mathcal{I}_{i,j}.
	\]
Take $\mathscr{R}^{\dagger}
=\sigma(g_{i}',g_{j}')$,
and let $\mathscr{G}_{\mathcal{J}^{\dagger}}$,
$V_{\mathcal{J}^{\dagger}}$ and $\mathscr{H}$
be as in \Cref{lem-smootheddeletion}. The following decomposition verifies the required
	measurability:
	\[
	\begin{aligned}
		S_{n,m}-V_{\mathcal{J}^{\dagger}}
		&=\sum_{k=m}^{p-1}Y_{k}
		+\sum_{k=p+H^{\dagger}+m-1}^{m+n-1}Y_{k}+
		\mathbb{E}\biggl[
		\sum_{k=p}^{p+H^{\dagger}+m-2}Y_{k}
		\,\biggm|\,\mathscr{G}_{\mathcal{J}^{\dagger}}
		\biggr].
	\end{aligned}
	\]
	The first two terms depend only on the matrices $g_{k}$ with $k\geq1$ and $k\notin\mathcal{J}^{\dagger}$, and the last term is
	$\mathscr{G}_{\mathcal{J}^{\dagger}}$-measurable.
	Thus $S_{n,m}-V_{\mathcal{J}^{\dagger}}$ is
	$\mathscr{H}$-measurable.
	The coordinates $i,j$ lie outside
	$\{p-m+1,\ldots,p+H^{\dagger}+m-2\}$, so the same
	decomposition, with the corresponding replacements
	in the two outside sums, proves that
	\[
	S_{n,m}^{(i)}-V_{\mathcal{J}^{\dagger}},\qquad
	S_{n,m}^{(j)}-V_{\mathcal{J}^{\dagger}},\qquad
	S_{n,m}^{(i,j)}-V_{\mathcal{J}^{\dagger}}
	\]
	are also $\mathscr{H}$-measurable.
	The chosen interval $\mathcal{J}^{\dagger}$ may depend on
	$i,j$, but its length is always $H^{\dagger}$. With this common construction in place, we now turn to the estimates.
	
	\smallskip
	\noindent
	\emph{(i) Upper bound for $|I_{1}|$.}
Using
	$\lVert f_{z}'\rVert_{\infty}\leq1$,
	\Cref{lem-predictable} and $m\geq 2$, we have
\begin{align}\label{eq-mm-02}
	\begin{aligned}
		|I_{1}|
		\leq\frac{1}{s_{m}^{2}n}
		\mathbb{E}\biggl|
		\sum_{i=m}^{m+n-1}
		\bigl(\widehat{Y}_{i}-s_{m}^{2}\bigr)
		\biggr|\leq\frac{C}{n}\sqrt{nm^{1-\delta}}
		=C\frac{m^{(1-\delta)/2}}{\sqrt{n}}.
	\end{aligned}
\end{align}

	\smallskip
	\noindent
	\emph{(ii) Upper bound for $|I_{2}|.$}
	Integrating with respect to $g_{i}^{\prime}$ gives
	\[
	\begin{aligned}
		&\mathbb{E}
		\Bigl\{\Bigl[\bigl(Y_{i}-Y_{i}^{(i)}\bigr)^{2}
		-2\widehat{Y}_{i}\Bigr]f_{z}'(W)
		\Bigr\}=
		\mathbb{E}\Bigl\{
		\bigl(Y_{i}^{2}-\widehat{Y}_{i}\bigr)f_{z}'(W)
		\Bigr\}.
	\end{aligned}
	\]
	Exchanging $g_{i}$ and
	$g_{i}^{\prime}$ then yields
	\begin{equation}\label{eq-ah-07}
		\begin{aligned}
			&\mathbb{E}
			\Bigl\{\Bigl[\bigl(Y_{i}-Y_{i}^{(i)}\bigr)^{2}
			-2\widehat{Y}_{i}\Bigr]f_{z}'(W)
			\Bigr\}\\
			&\quad=\frac{1}{2}\mathbb{E}\bigg\{
			\Big[Y_{i}^{2}-\bigl(Y_{i}^{(i)}\bigr)^{2}\Bigr]
			\Big[
		f_{z}'\Bigl(\frac{S_{n,m}+U}{s_{m}\sqrt{n}}\Bigr)
			-f_{z}'\Bigl(\frac{S_{n,m}^{(i)}+U}{s_{m}\sqrt{n}}\Bigr)
			\Big]\bigg\}.
		\end{aligned}
	\end{equation}
	For each $i$, choose $\mathcal{J}^{\dagger}$ using
	$\mathcal{I}_{i,i}$.  We apply
	\Cref{lem-smootheddeletion} with
	\[
	\begin{aligned}
		q=f_{z}',\qquad
		Z=Y_{i}^{2}-\bigl(Y_{i}^{(i)}\bigr)^{2},\qquad
		T=S_{n,m}^{(i)}-V_{\mathcal{J}^{\dagger}},\qquad
		x=S_{n,m}-S_{n,m}^{(i)},\qquad y=0.
	\end{aligned}
	\]
	Since $Z$ and $x$ depend only on
	$(g_{k})_{k\in\mathcal{I}_{i,i}}$
	and $g_{i}'$, they are jointly
	independent of $\mathscr{G}_{\mathcal{J}^{\dagger}}$.
	Using \eqref{eq-ah-07}, $H^{\dagger}\geq n/8$,
	$\operatorname{TV}(f_{z}')\leq2$, and
	\eqref{eq-ae-10}, we obtain
\begin{align}\label{eq-mm-03}
	\begin{aligned}
		|I_{2}|
		&\leq\frac{C}{s_{m}^{2}n}
		\biggl\{n^{-1/2}
		+\frac{1}{\alpha}\biggl(\frac{m}{H^{\dagger}}\biggr)^{2}\biggr\}
		\sum_{i=m}^{m+n-1}
		\mathbb{E}\Bigl\{
		\bigl|Y_{i}^{2}-\bigl(Y_{i}^{(i)}\bigr)^{2}\bigr|
		|\Delta_{i}|\Bigr\}\\
		&\leq C m^{1-\delta}
		\biggl\{n^{-1/2}
		+\frac{1}{\alpha}\biggl(\frac{m}{n}\biggr)^{2}\biggr\}.
	\end{aligned}
\end{align}
	
	\smallskip
	\noindent
	\emph{(iii) Upper bound for $|I_{3}|.$}
	For each $i$ and $t\in[0,1]$, use the interval
	chosen from $\mathcal{I}_{i,i}$ and apply
	\Cref{lem-smootheddeletion} with
	\[
	\begin{aligned}
		q&=f_{z}',\quad
		Z=\bigl(Y_{i}-Y_{i}^{(i)}\bigr)^{2},\quad
		T=(1-t)S_{n,m}^{(i)}+tS_{n,m}-V_{\mathcal{J}^{\dagger}},\\
		x&=(1-t)(S_{n,m}-S_{n,m}^{(i)}),\quad y=0.
	\end{aligned}
	\]
	The two arguments are $(S_{n,m}+U)/(s_{m}\sqrt{n})$ and
	$((1-t)S_{n,m}^{(i)}+tS_{n,m}+U)/(s_{m}\sqrt{n})$. By 
	the triangle inequality, Tonelli's theorem, $H^{\dagger}\geq n/8$ and
	\eqref{eq-ae-10}, we have
\begin{align}\label{eq-mm-04}
	\begin{aligned}
		|I_{3}|
		&\leq\frac{C}{s_{m}^{2}n}
		\biggl\{n^{-1/2}
		+\frac{1}{\alpha}\biggl(\frac{m}{H^{\dagger}}\biggr)^{2}\biggr\}
		\sum_{i=m}^{m+n-1}\int_{0}^{1}
		(1-t)\mathbb{E}\Bigl\{
		\bigl(Y_{i}-Y_{i}^{(i)}\bigr)^{2}|\Delta_{i}|
		\Bigr\}\,\mathrm{d}t\\
		&\leq C m^{1-\delta}
		\biggl\{n^{-1/2}
		+\frac{1}{\alpha}\biggl(\frac{m}{n}\biggr)^{2}\biggr\}.
	\end{aligned}
\end{align}
	
	\smallskip
	\noindent
	\emph{(iv) Upper bound for $I_{4}.$}
	For $i<j\leq\min\{i+m-1,m+n-1\}$, recall that
	\[
	\Gamma_{ij}
	=Y_{j}-Y_{j}^{(i)}-Y_{j}^{(j)}+Y_{j}^{(i,j)},
	\qquad
	\Delta_{j}^{(i)}
	=S_{n,m}^{(i)}-S_{n,m}^{(i,j)}.
	\]
	Fix $t\in[0,1]$.  By exchangeability, we have
	\begin{equation}\label{eq-ah-08}
		\begin{aligned}
			&\mathbb{E}\bigg\{
			\bigl(Y_{i}-Y_{i}^{(i)}\bigr)
			\bigl(Y_{j}-Y_{j}^{(i)}\bigr)
			f_{z}'\Bigl(
			\frac{(1-t)S_{n,m}^{(i)}+tS_{n,m}+U}{s_{m}\sqrt{n}}
			\Bigr)\biggr\}\\
			&\quad=\frac{1}{2}\mathbb{E}\biggl\{
			\bigl(Y_{i}-Y_{i}^{(i)}\bigr)\Gamma_{ij}
			\biggl[
			f_{z}'\Bigl(
			\frac{(1-t)S_{n,m}^{(i)}+tS_{n,m}+U}{s_{m}\sqrt{n}}
			\Bigr)
			-f_{z}'\Bigl(
			\frac{(1-t)S_{n,m}^{(i,j)}+tS_{n,m}^{(j)}+U}{s_{m}\sqrt{n}}
			\Bigr)
			\biggr]\biggr\}.
		\end{aligned}
	\end{equation}
	The difference of the two unnormalized sums is
	\begin{equation*}
		\begin{aligned}
			&\bigl((1-t)S_{n,m}^{(i)}+tS_{n,m}\bigr)
			-\bigl((1-t)S_{n,m}^{(i,j)}+tS_{n,m}^{(j)}\bigr)=(1-t)\Delta_{j}^{(i)}+t\Delta_{j}.
		\end{aligned}
	\end{equation*}
	Choose $\mathcal{J}^{\dagger}$ from $\mathcal{I}_{i,j}$.
	For each $t\in[0,1]$, apply
	\Cref{lem-smootheddeletion} with
	\[
	\begin{aligned}
		q&=f_{z}',\qquad
		Z=\bigl(Y_{i}-Y_{i}^{(i)}\bigr)\Gamma_{ij},\quad
		T=(1-t)S_{n,m}^{(i,j)}+tS_{n,m}^{(j)}
		-V_{\mathcal{J}^{\dagger}},\\
		x&=(1-t)\Delta_{j}^{(i)}+t\Delta_{j},
		\qquad y=0.
	\end{aligned}
	\]
	The variable $T$ is $\mathscr{H}$-measurable by the
	decompositions above. The pair $(Z,x)$ depends only
	on the matrices $(g_{k})_{k\in\mathcal{I}_{i,j}}$ and the copies
	at $i,j$, so it is jointly independent of
	$\mathscr{G}_{\mathcal{J}^{\dagger}}$.
	Integrating \eqref{eq-ah-08} in $t$, applying
	\Cref{lem-smootheddeletion} with the above choices,  and then using \eqref{eq-ae-11}, we obtain
\begin{align}\label{eq-mm-05}
	\begin{aligned}
		|I_{4}|
		&\leq\frac{C}{n}
		\biggl\{\frac{1}{\sqrt{n}}
		+\frac{1}{\alpha}\biggl(\frac{m}{n}\biggr)^{2}\biggr\}
		\sum_{i=m}^{m+n-1}
		\sum_{j=i+1}^{\min\{i+m-1,m+n-1\}}
		\mathbb{E}\Bigl[
		\bigl|\bigl(Y_{i}-Y_{i}^{(i)}\bigr)\Gamma_{ij}\bigr|
		\bigl(|\Delta_{j}|+|\Delta_{j}^{(i)}|\bigr)
		\Bigr]\\
		&\leq C m^{1-\delta}
		\biggl\{n^{-1/2}
		+\frac{1}{\alpha}\biggl(\frac{m}{n}\biggr)^{2}\biggr\}.
	\end{aligned}
\end{align}
	
	\smallskip
	\noindent
	\emph{(v) Upper bound for $\E[Uf_{z}(W)].$}
	Since $|U|\leq \alpha$ and $\lVert f_{z}\rVert_{\infty}\leq C$, it follows that
\begin{align}\label{eq-mm-06}
	\frac{1}{s_{m}\sqrt{n}}
\bigl|\mathbb{E}[Uf_{z}(W)]\bigr|
\leq C\frac{\alpha}{\sqrt{n}}.
\end{align}
	
	Substituting the bounds in \eqref{eq-mm-02},
	\eqref{eq-mm-03}, \eqref{eq-mm-04}, \eqref{eq-mm-05},
	and \eqref{eq-mm-06} into \eqref{eq-ah-11},
	and taking the supremum over $z$, we obtain  
	\begin{align}	\label{eq-ah-12}
		\begin{aligned}
			d_{\mathrm{K}}\bigl(
	\mathcal{L}(W),\mathcal{N}(0,1)\bigr)
	&\leq C\biggl\{
	\frac{m^{(1-\delta)/2}+m^{1-\delta}+\alpha}{\sqrt{n}}
	+\frac{m^{1-\delta}}{\alpha}
	\biggl(\frac{m}{n}\biggr)^{2}
	\biggr\}
	\\
	&\leq C\biggl\{
	\frac{m^{1-\delta}}{\sqrt{n}}
	+\biggl(\frac{m}{n}\biggr)^{2}
	\biggr\},
		\end{aligned}
	\end{align}
	where 
	the last inequality uses $m^{1-\delta}\geq1$ and
	$\alpha=L_{*}m^{1-\delta}$, with $L_{*}$ fixed.
	For $m=\lfloor\sqrt{n}\rfloor$,
	\[
	\frac{m^{1-\delta}}{\sqrt{n}}
	\leq n^{-\delta/2},
	\qquad
	\biggl(\frac{m}{n}\biggr)^{2}
	\leq Cn^{-1}
	\leq Cn^{-\delta/2},
	\]
	because  $\delta\leq1$.
	This proves \eqref{eq-ah-10} for all sufficiently
	large $n$; any remaining admissible values of $n$
	are covered by enlarging $C$.
	
	Finally, $-\alpha\leq U\leq \alpha$ implies
	\[
	\mathbb{P}\big(W\leq z-\alpha/(s_{m}\sqrt{n})\big)
	\leq
	\mathbb{P}\bigl(S_{n,m}/(s_{m}\sqrt{n})\leq z\bigr)
	\leq
	\mathbb{P}\big(W\leq z+\alpha/(s_{m}\sqrt{n})\big).
	\]
	Using the bound $1/\sqrt{2\pi}$ on the standard
	normal density, we obtain
	\[
	\begin{aligned}
		d_{\mathrm{K}}\bigl(
		\mathcal{L}(S_{n,m}/(s_{m}\sqrt{n})),\mathcal{N}(0,1)
		\bigr)
		&\leq
		d_{\mathrm{K}}\bigl(
		\mathcal{L}(W),\mathcal{N}(0,1)
		\bigr)
		+\frac{\alpha}{s_{m}\sqrt{n}\sqrt{2\pi}}\leq Cn^{-\delta/2}.
	\end{aligned}
	\]
	This proves the assertion without $U$ as well.
\end{proof}

We next compare the stationary censored sum with the sum of the conditional expectations $X_{k,m}^{[a]}$. 
Define
\[
T_{n,m}^{[a]}
:=\sum_{k=m}^{m+n-1}
H^{[a]}(g_{k},W_{k-1}).
\]
By stationarity, this has the same law as
$\sum_{k=1}^{n}H^{[a]}(g_{k},W_{k-1})$.
The shift of the summation interval retains
exactly $n$ terms.

The following proposition shows that replacing the
stationary censored sum by the sum of the conditional expectations $X_{k,m}^{[a]}$
produces an error of $o(m^{-\delta})$ in Kolmogorov
distance after smoothing by $U$.

\begin{proposition}
	\label{prop-localization}
	For $a=m=\lfloor\sqrt{n}\rfloor$, we have
	\begin{equation}\label{eq-ai-04}
		\begin{aligned}
			&\sup_{z\in\mathbb{R}}
			\Bigl|
			\mathbb{P}\bigl(T_{n,m}^{[a]}+U\leq z\bigr)
			-
			\mathbb{P}\Bigl(
			\sum_{k=m}^{m+n-1}X_{k,m}^{[a]}+U\leq z
			\Bigr)
			\Bigr|\leq
			Cm\delta_{1,\infty}(m)
			=o(m^{-\delta}),
		\end{aligned}
	\end{equation}
		where  $C>0$ depends only on
	$\mu$, $\delta$, $d$, and the fixed norm.
\end{proposition}
\begin{proof}
	Fix $K_{1}=2048$, and put
\begin{align}\label{eq-ub-03}
		q:=\biggl\lceil\frac{m^{\delta}}{K_{1}}\biggr\rceil,
	\qquad
	h:=\biggl\lceil\frac{n}{q}\biggr\rceil,
	\qquad
	b_{r}:=m-1+\biggl\lfloor\frac{rn}{q}\biggr\rfloor
	\quad(0\leq r\leq q).
\end{align}
	The intervals
$
	\mathcal{I}_{r}:=\{b_{r-1}+1,\ldots,b_{r}\},$
 $1\leq r\leq q,
$
	partition $\{m,\ldots,m+n-1\}$, and each has length
	$h$ or $h-1$.
	We  assume that $m$ is sufficiently large and
	$m^{\delta}\geq K_{1}$. Then
	$m^{\delta}/K_{1}\leq q\leq2m^{\delta}/K_{1}$.
	Since $m^{2}\leq n<(m+1)^{2}\leq4m^{2}$, we obtain
	\begin{equation}\label{eq-ai-05}
		\frac{K_{1}}{2}m^{2-\delta}
		\leq h\leq5K_{1}m^{2-\delta},
		\qquad n\leq qh\leq2n.
	\end{equation}
	In particular, $h\geq1024m$.	For $0\leq r\leq q$, define
	\[
	T_{r}
	:=
	\sum_{k=m}^{b_{r}}
	H^{[a]}(g_{k},W_{k-1})
	+\sum_{k=b_{r}+1}^{b_{q}}X_{k,m}^{[a]},
	\]
	where empty sums are zero. Thus
$
	T_{q}=T_{n,m}^{[a]},$ and
$
	T_{0}=\sum_{k=m}^{m+n-1}X_{k,m}^{[a]}.$
	For $1\leq r\leq q$, put
	\[
	D_{r}:=T_{r}-T_{r-1}
	=\sum_{k\in\mathcal{I}_{r}}
	\bigl\{
	H^{[a]}(g_{k},W_{k-1})-X_{k,m}^{[a]}
	\bigr\}.
	\]
	The random variable $D_{r}$ is measurable with respect to
	$\mathscr{F}_{b_{r}}
	=\sigma(W_{0},g_1,\ldots,g_{b_{r}})$. Moreover, recalling that 
$
		\mathbb{E}\lvert X^{[a]}_{k,m}-H^{[a]}(g_{k},W_{k-1})|\leq  \delta_{1,\infty}(m)
$
by \eqref{eq-ai-03},
we have
	\[
	\mathbb{E}|D_{r}|
	\leq |\mathcal{I}_{r}|\delta_{1,\infty}(m)
	\leq h\delta_{1,\infty}(m).
	\]
In what follows,	we compare $T_{q},T_{q-1},\ldots,T_{0}$ in this order,
	using the same independent $U$ throughout.
	
	For the first comparison, conditioning on $W_{0}$
	and the sequence $(g_{k})_{k\geq1}$ leaves only $U$ random.
	Since its density is bounded by $1/(2\alpha)$,
	\begin{align}\label{eq-ai-last-step}
		&\sup_{z\in\mathbb{R}}
		\bigl|
		\mathbb{P}(T_{q}+U\leq z)
		-\mathbb{P}(T_{q-1}+U\leq z)
		\bigr|\leq
		\frac{\mathbb{E}|D_{q}|}{2\alpha}
		\leq\frac{h\delta_{1,\infty}(m)}{2\alpha}.
	\end{align}
	
	Now fix $1\leq r<q$, and set
$
	\ell_{r}:=b_{q}-b_{r},$
$H_{r}:=\lfloor\ell_{r}/2\rfloor.$
	Since
	$(q-r)(h-1)\leq\ell_{r}\leq(q-r)h$
	and $h\geq1024m$, we have
	\[
	H_{r}\geq64m,
	\qquad H_{r}+2(m-1)\leq\ell_{r},
	\qquad H_{r}\geq\frac{(q-r)h}{4}.
	\]
	To apply \Cref{cor-cdfcomparison}, take
$
	H^{\dagger}=H_{r},$
$	\mathcal{J}^{\dagger}
	=\{b_{r}+m,\ldots,b_{r}+m+H_{r}-1\},$
	and let $\mathscr{R}^{\dagger}$ be trivial.
	Use $\mathscr{G}_{\mathcal{J}^{\dagger}}$,
	$V_{\mathcal{J}^{\dagger}}$, and $\mathscr{H}$
	as in \Cref{lem-smootheddeletion} for this choice.
	The two boundary intervals are
$
	\{b_{r}+1,\ldots,b_{r}+m-1\},$
and
$	\{b_{r}+m+H_{r},\ldots,b_{r}+H_{r}+2m-2\}.
$
	Thus the center and both boundary intervals lie in
	$\{b_{r}+1,\ldots,b_{q}\}$, and the matrices $g_{b_{r}+1},\ldots,g_{b_{r}+H_{r}+2m-2}$
	are independent of $\mathscr{F}_{b_{r}}$.
	
	With a similar argument as that leading to \eqref{eq-af-03}, we have
	\[
	V_{\mathcal{J}^{\dagger}}
	=\sum_{k=b_{r}+m}^{b_{r}+H_{r}+2m-2}X_{k,m}^{[a]}
	-\mathbb{E}\biggl[
	\sum_{k=b_{r}+m}^{b_{r}+H_{r}+2m-2}X_{k,m}^{[a]}
	\,\biggm|\,\mathscr{G}_{\mathcal{J}^{\dagger}}
	\biggr].
	\]
 Consequently, we have
	\[
	\begin{aligned}
		T_{r}-V_{\mathcal{J}^{\dagger}}
		&=\sum_{k=m}^{b_{r}}
		H^{[a]}(g_{k},W_{k-1})
		+\sum_{k=b_{r}+1}^{b_{r}+m-1}X_{k,m}^{[a]}\\
		&\quad+\sum_{k=b_{r}+H_{r}+2m-1}^{b_{q}}X_{k,m}^{[a]}+\mathbb{E}\biggl[
		\sum_{k=b_{r}+m}^{b_{r}+H_{r}+2m-2}X_{k,m}^{[a]}
		\,\biggm|\,\mathscr{G}_{\mathcal{J}^{\dagger}}
		\biggr].
	\end{aligned}
	\]
Every term on the right is $\mathscr{H}$-measurable:
the first two sums depend only on $W_0$ and
$g_1,\ldots,g_{b_r+m-1}$;
the third sum depends only on
$g_{b_r+H_r+m},\ldots,g_{b_q}$;
and the last term is
$\mathscr{G}_{\mathcal{J}^{\dagger}}$-measurable.
	Moreover, $D_{r}$ is $\mathscr{F}_{b_{r}}$-measurable,
	and hence is $\mathscr{H}$-measurable and independent
	of $\mathscr{G}_{\mathcal{J}^{\dagger}}$.
	Therefore,  applying \Cref{cor-cdfcomparison} with
$
	T=T_{r}-V_{\mathcal{J}^{\dagger}},$
$D=D_{r},$
we have
	\begin{align}\label{eq-ai-intermediate-step}
		&\sup_{z\in\mathbb{R}}
		\bigl|
		\mathbb{P}(T_{r}+U\leq z)
		-\mathbb{P}(T_{r-1}+U\leq z)
		\bigr|\leq
		Ch\delta_{1,\infty}(m)
		\biggl\{
		H_{r}^{-1/2}
		+\frac{1}{\alpha}\biggl(\frac{m}{H_{r}}\biggr)^{2}
		\biggr\}.
	\end{align}
Applying the triangle inequality and using
\eqref{eq-ai-last-step}, \eqref{eq-ai-intermediate-step},
and $H_{r}\geq(q-r)h/4$ for $1\leq r<q$, we obtain
	\begin{equation}\label{eq-ai-06}
		\begin{aligned}
			&\sup_{z\in\mathbb{R}}
			\bigl|
			\mathbb{P}(T_{q}+U\leq z)
			-\mathbb{P}(T_{0}+U\leq z)
			\bigr|\\
			&\quad\leq
			Ch\delta_{1,\infty}(m)
			\biggl\{
			\frac{1}{\alpha}
			+\sum_{j=1}^{q-1}(jh)^{-1/2}
			+\frac{1}{\alpha}\sum_{j=1}^{q-1}
			\biggl(\frac{m}{jh}\biggr)^{2}
			\biggr\}.
		\end{aligned}
	\end{equation}
	Recall that $\alpha=L_{*}m^{1-\delta}$.
	By \eqref{eq-ai-05}, $h\geq m$, and
	$\alpha\geq1$,  we have  $h/\alpha\leq Cm,$
	\[
	\begin{aligned}
		h\sum_{j=1}^{q-1}(jh)^{-1/2}
		&\leq2\sqrt{qh}\leq C\sqrt{n}\leq Cm,\quad
		\frac{h}{\alpha}\sum_{j=1}^{q-1}
		\biggl(\frac{m}{jh}\biggr)^{2}
		\leq\frac{C}{\alpha}m^{2}h^{-1}
		\leq Cm.
	\end{aligned}
	\]
	This proves the first inequality in \eqref{eq-ai-04}
	for sufficiently large $m$.
	For the remaining finitely many values, smoothing
	by $U$ alone bounds the left-hand side by
	$n\delta_{1,\infty}(m)/(2\alpha)$, so the same
	inequality holds after enlarging $C$.
	
	Finally, 
	since $\sum_{k=1}^{\infty}k^{\delta}\delta_{1,\infty}(k)<\infty$ and $\delta_{1,\infty}(k)$ is nonincreasing,  it follows that
	\[
	m^{1+\delta}\delta_{1,\infty}(m)
	\leq C\sum_{k=\lceil m/2\rceil}^{m}
	k^{\delta}\delta_{1,\infty}(k)
	\longrightarrow0\quad \text{as } m\to\infty.  
	\]
	Thus $m\delta_{1,\infty}(m)=o(m^{-\delta})$,
	which completes the proof.
\end{proof}

The following proposition gives a normal approximation for the censored
stationary sum  with rate $n^{-\delta/2}$. 
\begin{proposition}
	\label{prop-stationaryclipped}
	For $a=m=\lfloor\sqrt{n}\rfloor$ and all sufficiently large $n$, we have
	\begin{equation}\label{eq-ai-07}
		d_{\mathrm{K}}\biggl(
		\mathcal{L}\biggl(
		\frac{\sum_{k=1}^{n}
			H^{[a]}(g_{k},W_{k-1})}{s\sqrt{n}}
		\biggr),
		\mathcal{N}(0,1)
		\biggr)
		\leq Cn^{-\delta/2}.
	\end{equation}
	The same bound holds if $U$ is added to the numerator.
\end{proposition}

\begin{proof}
	The Gordin decomposition of the approximating sequence in \eqref{eq-ad-01} gives
	\[
	S_{n,m}-\sum_{k=m}^{m+n-1}X_{k,m}^{[a]}
	=
	\psi_{m}^{[a]}(U_{m+n-1})
	-\psi_{m}^{[a]}(U_{m-1}).
	\]
By \eqref{eq-ad-04}, there exists a constant $C>0$,
independent of $n$, $a$, and $m$, such that, for every
$z\in\mathbb{R}$,
	\[
	\begin{aligned}
		\mathbb{P}(S_{n,m}+U\leq z-C)
		&\leq
		\mathbb{P}\biggl(
		\sum_{k=m}^{m+n-1}X_{k,m}^{[a]}+U\leq z
		\biggr)
		\leq
		\mathbb{P}(S_{n,m}+U\leq z+C).
	\end{aligned}
	\]
	Combining this inequality and
 \Cref{prop-Stein}, we have
	\[
	\begin{aligned}
		&d_{\mathrm{K}}\biggl(
		\mathcal{L}\biggl(
		\frac{\sum_{k=m}^{m+n-1}X_{k,m}^{[a]}+U}
		{s_{m}\sqrt{n}}
		\biggr),\mathcal{N}(0,1)
		\biggr)\leq
		Cn^{-\delta/2}
		+\frac{C}{s_{m}\sqrt{2\pi n}}
		\leq Cn^{-\delta/2}.
	\end{aligned}
	\]
By \Cref{prop-localization}, replacing
$\sum_{k=m}^{m+n-1}X_{k,m}^{[a]}$ by $T_{n,m}^{[a]}$
in the preceding inequality introduces an error of at most
$
Cm\delta_{1,\infty}(m)
=o(m^{-\delta})
=o(n^{-\delta/2}).
$
	Stationarity then gives
\begin{align}\label{eq-ub-01}
	d_{\mathrm{K}}\biggl(
\mathcal{L}\biggl(
\frac{\sum_{k=1}^{n}
	H^{[a]}(g_{k},W_{k-1})+U}
{s_{m}\sqrt{n}}
\biggr),\mathcal{N}(0,1)
\biggr)
\leq Cn^{-\delta/2}.
\end{align}
	It remains to replace $s_{m}$ by $s$.
	By \eqref{eq-ad-09}, we have
	\[
	|s_{m}^{2}-s^{2}|\leq Cm^{-\delta},
	\qquad
	\Bigl|\frac{s_{m}}{s}-1\Bigr|
	\leq Cm^{-\delta},
	\]
	where we use the positive lower bounds on
	$s_{m}$ and $s$.
	Consequently, we have
\begin{align}\label{eq-ub-02}
	\sup_{z\in\mathbb{R}}
	|\Phi(zs/s_{m})-\Phi(z)|
	\leq Cm^{-\delta}\leq Cn^{-\delta/2}.
\end{align}
Combining \eqref{eq-ub-01} and \eqref{eq-ub-02}, we have
	\[
	d_{\mathrm{K}}\biggl(
	\mathcal{L}\biggl(
	\frac{\sum_{k=1}^{n}
		H^{[a]}(g_{k},W_{k-1})+U}
	{s\sqrt{n}}
	\biggr),\mathcal{N}(0,1)
	\biggr)
	\leq Cn^{-\delta/2}.
	\]
Finally, using $|U|\leq\alpha$ and arguing as in
\Cref{prop-Stein}, we remove $U$ with an additional
error of at most
$\alpha/(s\sqrt{2\pi n})=O(n^{-\delta/2})$.
	This proves \eqref{eq-ai-07}.
\end{proof}

Now, we are ready to prove \eqref{eq-tt-01}.
\begin{proof}[Proof of \eqref{eq-tt-01}]
Let 	$S_{n}^{[a]}
=\sum_{k=1}^{n}
H^{[a]}(g_{k},W_{k-1})$.  By \Cref{prop-stationaryclipped}, it suffices to prove the following bound:
\begin{align}\label{eq-uu-01}
				&d_{\mathrm{K}}\biggl(
	\mathcal{L}\biggl(\frac{S_{n}}{s\sqrt{n}}\biggr),
	\mathcal{N}(0,1)\biggr)\leq
	d_{\mathrm{K}}\biggl(
	\mathcal{L}\biggl(\frac{S_{n}^{[a]}}{s\sqrt{n}}\biggr),
	\mathcal{N}(0,1)\biggr)
	+Cn^{-\delta/2}.
\end{align}
	Define
$
\mathcal{E}_{n,a}
:=\bigcap_{k=1}^{n}\{L(g_{k})\leq a\}.
$ Then on $\mathcal{E}_{n,a}$, we have $	S_{n}=S_{n}^{[a]}+nc^{[a]},$ where we recall that $c^{[a]}$ is defined by \eqref{eq-ll-01}.  Therefore, for any fixed $z\in \R$, we have 
\begin{align}\label{eq-uu-02}
	\begin{aligned}
		&\big|
		\mathbb{P}\big(S_{n}/(s\sqrt{n})\leq z\big)
		-\Phi(z)
		\big|\\
		&\quad\leq
		\mathbb{P}(\mathcal{E}_{n,a}^{c})+
		\Big|
		\mathbb{P}\big(
		S_{n}^{[a]}/(s\sqrt{n})
		\leq z-\sqrt{n}\,c^{[a]}/s
		\big)
		-
		\Phi\big(z-\sqrt{n}\,c^{[a]}/s\big)
		\Big|\\
		&\quad\quad+
		\big|
		\Phi\big(z-\sqrt{n}\,c^{[a]}/s\big)
		-\Phi(z)
		\big|.
	\end{aligned}
\end{align}
By the union bound and Markov's inequality, we have $\mathbb{P}(\mathcal{E}_{n,a}^{c})\leq Cn^{-\delta/2}.$ For the third term on the right-hand side of \eqref{eq-uu-02}, using $|\Phi(x)-\Phi(y)|\leq |x-y|$,  \eqref{eq-ac-07} and $a\asymp\sqrt{n}$ , we have 	$|
\Phi\big(z-\sqrt{n}\,c^{[a]}/s\big)
-\Phi(z)
|\leq Cn^{-\delta/2}. $ This proves \eqref{eq-uu-01}.
\end{proof}

\subsection{Proof of  \protect\eqref{eq-tt-02}}\label{sec-5.2}
For any fixed $x\in X$, define
\[
S_{n,x}^{[a]}=\sum_{k=1}^{n}H^{[a]}(g_k,A_{k-1}x).
\]
On $\mathcal E_{n,a}$, we have $S_{n,x}=S_{n,x}^{[a]}+nc^{[a]}$.
Thus, by the same argument as in \eqref{eq-uu-02}, it suffices to prove   
\begin{align}\label{eq-uu-03}
	\sup_{x\in X}
	d_{\mathrm{K}}\Big(
	\mathcal{L}\Big(
	S_{n,x}^{[a]}/(s\sqrt{n})
	\Big),\mathcal{N}(0,1)
	\Big)
	&\leq Cn^{-\delta/2}.
\end{align}
For $x\in X$, define
\[
\widetilde{S}_{n,m}^{[a]}(x)
:=
\sum_{k=1}^{m}H^{[a]}(g_{k},A_{k-1}x)
+\sum_{k=m+1}^{n}X_{k,m}^{[a]}.
\]
Using $|U|\leq\alpha$, $\alpha/\sqrt{n}\leq Cn^{-\delta/2}$,
the triangle inequality, and \Cref{prop-stationaryclipped}, we have 
\begin{align}\label{eq-uc-01}
		\sup_{x\in X}
	d_{\mathrm{K}}\Big(
	\mathcal{L}\Big(
	S_{n,x}^{[a]}/(s\sqrt{n})
	\Big),\mathcal{N}(0,1)
	\Big)\leq I_{1}+I_{2}+I_{3}+Cn^{-\delta/2},\end{align}
where
\begin{align*}
	I_{1}&=\sup_{x\in X}\sup_{z\in\mathbb{R}}
	\bigl|
	\mathbb{P}(S_{n,x}^{[a]}+U\leq z)
	-\mathbb{P}(\widetilde{S}_{n,m}^{[a]}(x)+U\leq z)
	\bigr|,\\
	I_{2}&=\sup_{x\in X}\sup_{z\in\mathbb{R}}
	\bigl|
	\mathbb{P}(\widetilde{S}_{n,m}^{[a]}(x)+U\leq z)
	-\mathbb{P}(\widetilde{S}_{n,m}^{[a]}(W_{0})+U\leq z)
	\bigr|,\\
	I_{3}&=\sup_{z\in\mathbb{R}}
	\bigl|
	\mathbb{P}(S_{n}^{[a]}+U\leq z)
	-\mathbb{P}(\widetilde{S}_{n,m}^{[a]}(W_{0})+U\leq z)
	\bigr|.
\end{align*}

We next apply \Cref{cor-cdfcomparison} to derive upper bounds for
$I_{1}$, $I_{2}$, and $I_{3}$ separately.

\smallskip
\noindent
\emph{(i) Upper bound for $|I_{1}|.$}
For $k>m$, the state $A_{k-m}x$ is independent of
$g_{k-m+1},\ldots,g_{k}$. With a similar argument as that leading to 
\eqref{eq-ai-03}, we have
\[
\sup_{x\in X}\mathbb{E}\bigl|
H^{[a]}(g_{k},A_{k-1}x)-X_{k,m}^{[a]}
\bigr|
\leq\delta_{1,\infty}(m).
\]
As in the proof of \Cref{prop-localization}, we replace
$H^{[a]}(g_{k},A_{k-1}x)$ by $X_{k,m}^{[a]}$
for $k=m+1,\ldots,n$, block by block from right to left, while
the first $m$ terms remain unchanged.
Use the same $q$ as in \eqref{eq-ub-03}, and set
\[
h':=\biggl\lceil\frac{n-m}{q}\biggr\rceil,
\qquad
b'_{r}:=m+\biggl\lfloor\frac{r(n-m)}{q}\biggr\rfloor
\quad(0\leq r\leq q).
\]
The intervals
$\mathcal{I}'_{r}
:=\{b'_{r-1}+1,\ldots,b'_{r}\}$,
$1\leq r\leq q$, partition $\{m+1,\ldots,n\}$, and each has length
$h'$ or $h'-1$.
Using $n-m\geq m^{2}/2$, the same calculation as in
\Cref{prop-localization} gives, for sufficiently large $m$,
$
h'\geq512m,$
$ h'/\alpha\leq Cm,$
$ qh'\leq2n.$
For $0\leq r\leq q$, put
\[
T'_{r}(x)
:=\sum_{k=1}^{b'_{r}}H^{[a]}(g_{k},A_{k-1}x)
+\sum_{k=b'_{r}+1}^{n}X_{k,m}^{[a]}.
\]
Then $T'_{q}(x)=S_{n,x}^{[a]}$ and
$T'_{0}(x)=\widetilde{S}_{n,m}^{[a]}(x)$.
For $1\leq r\leq q$, the difference
\[
D'_{r}(x):=T'_{r}(x)-T'_{r-1}(x)
=\sum_{k\in\mathcal{I}'_{r}}
\bigl\{
H^{[a]}(g_{k},A_{k-1}x)-X_{k,m}^{[a]}
\bigr\}
\]
is $\mathscr{F}_{b'_{r}}$-measurable, where
$\mathscr{F}_{j}=\sigma(W_{0},g_1,\ldots,g_{j})$,
and satisfies
$
\sup_{x\in X}\mathbb{E}|D'_{r}(x)|
\leq h'\delta_{1,\infty}(m).
$
As in the comparison of $T_{q}+U$ and $T_{q-1}+U$
in the proof of \Cref{prop-localization}, we have
\begin{equation}\label{eq-va-block-last}
	\sup_{x\in X}\sup_{z\in\mathbb{R}}
	\bigl|
	\mathbb{P}(T'_{q}(x)+U\leq z)
	-\mathbb{P}(T'_{q-1}(x)+U\leq z)
	\bigr|
	\leq\frac{h'\delta_{1,\infty}(m)}{2\alpha}.
\end{equation}
For $1\leq r<q$, let
$H'_{r}:=\lfloor(n-b'_{r})/2\rfloor$.
The bounds on $h'$ imply
\[
H'_{r}\geq64m,
\qquad H'_{r}+2(m-1)\leq n-b'_{r},
\qquad H'_{r}\geq\frac{(q-r)h'}{4}.
\]
We apply \Cref{cor-cdfcomparison} with
$
H^{\dagger}=H'_{r},$
$\mathcal{J}^{\dagger}
=\{b'_{r}+m,\ldots,b'_{r}+m+H'_{r}-1\},$
and with $\mathscr{R}^{\dagger}$ trivial.
For this choice, use $\mathscr{G}_{\mathcal{J}^{\dagger}}$,
$V_{\mathcal{J}^{\dagger}}$, and $\mathscr{H}$
as in \Cref{lem-smootheddeletion}.
The center and both boundary intervals lie in
$\{b'_{r}+1,\ldots,n\}$, so the matrices $g_{b'_{r}+1},\ldots,g_{b'_{r}+H'_{r}+2m-2}$ are
independent of $\mathscr{F}_{b'_{r}}$.
The  identity \eqref{eq-af-03} gives
\[
\begin{aligned}
	T'_{r}(x)-V_{\mathcal{J}^{\dagger}}
	&=\sum_{k=1}^{b'_{r}}H^{[a]}(g_{k},A_{k-1}x)
	+\sum_{k=b'_{r}+1}^{b'_{r}+m-1}X_{k,m}^{[a]}\\
	&\quad+\sum_{k=b'_{r}+H'_{r}+2m-1}^{n}X_{k,m}^{[a]}+\mathbb{E}\biggl[
	\sum_{k=b'_{r}+m}^{b'_{r}+H'_{r}+2m-2}X_{k,m}^{[a]}
	\,\biggm|\,\mathscr{G}_{\mathcal{J}^{\dagger}}
	\biggr].
\end{aligned}
\]
For fixed $x$, the first two sums depend only on
$g_1,\ldots,g_{b'_r+m-1}$, and the third sum depends only on
$g_{b'_r+H'_r+m},\ldots,g_n$. The last term is
$\mathscr{G}_{\mathcal{J}^{\dagger}}$-measurable,
hence $\mathscr{H}$-measurable.
Thus $T'_{r}(x)-V_{\mathcal{J}^{\dagger}}$ is
$\mathscr{H}$-measurable.
Moreover, $D'_{r}(x)$ is $\mathscr{H}$-measurable and
independent of $\mathscr{G}_{\mathcal{J}^{\dagger}}$.
Applying  \Cref{cor-cdfcomparison} with
$
T=T'_{r}(x)-V_{\mathcal{J}^{\dagger}},$ $D=D'_{r}(x),$
yields
\begin{equation}\label{eq-va-block-step}
	\begin{aligned}
		&\sup_{x\in X}\sup_{z\in\mathbb{R}}
		\bigl|
		\mathbb{P}(T'_{r}(x)+U\leq z)
		-\mathbb{P}(T'_{r-1}(x)+U\leq z)
		\bigr|\\
		&\qquad\leq
		Ch'\delta_{1,\infty}(m)
		\biggl\{
		(H'_{r})^{-1/2}
		+\frac{1}{\alpha}\biggl(\frac{m}{H'_{r}}\biggr)^{2}
		\biggr\}.
	\end{aligned}
\end{equation}
As in the derivation of \eqref{eq-ai-06}, we apply
the triangle inequality, using \eqref{eq-va-block-last}
for $r=q$ and summing \eqref{eq-va-block-step}
over $r=1,\ldots,q-1$.
Using $h'/\alpha\leq Cm$, $qh'\leq2n$, and $h'\geq m$,
we obtain
\begin{align}\label{eq-va-01}
	I_{1}\leq Cm\delta_{1,\infty}(m)\leq Cn^{-\delta/2}.
\end{align}
\smallskip
\noindent
\emph{(ii) Upper bound for $|I_{3}|.$} 
We use the same argument as in the estimate of $I_{1}$, with
$W_{0}$ in place of $x$: the replacement errors remain
$\mathscr{F}_{b'_{r}}$-measurable, and their first absolute moments
satisfy the same bound, since $W_{0}$ is independent of the sequence $(g_{k})_{k\geq1}$.
Recalling that $S_{n}^{[a]}=\sum_{k=1}^{n}H^{[a]}(g_{k},W_{k-1})$,
we therefore  have
\begin{align}\label{eq-va-02}
I_{3}=\sup_{z\in\mathbb{R}}
\bigl|
\mathbb{P}(S_{n}^{[a]}+U\leq z)
-\mathbb{P}(\widetilde{S}_{n,m}^{[a]}(W_{0})+U\leq z)
\bigr|
\leq Cm\delta_{1,\infty}(m)\leq Cn^{-\delta/2}.
\end{align}
\smallskip
\noindent
\emph{(iii) Upper bound for $|I_{2}|.$} 
Let
\begin{align}\label{eq-defi-dx}
\begin{aligned}
	D(x)
	&:=\widetilde{S}_{n,m}^{[a]}(x)
	-\widetilde{S}_{n,m}^{[a]}(W_{0})=\sum_{k=1}^{m}
	\bigl\{
	H^{[a]}(g_{k},A_{k-1}x)
	-H^{[a]}(g_{k},A_{k-1}W_{0})
	\bigr\}.
\end{aligned}
\end{align}
It is $\mathscr{F}_{m}$-measurable.
By the Lipschitz property of $\tau^{[a]}$,
the definition of $\delta_{1,\infty}(k)$, and \Cref{lem-1}, we have
$
\sup_{x\in X}\mathbb{E}|D(x)|
\leq\sum_{k=1}^{m}\delta_{1,\infty}(k)
\leq C.
$
Choose
\[
H^{\dagger}:=\biggl\lfloor\frac{n-m}{2}\biggr\rfloor,
\qquad
\mathcal{J}^{\dagger}:=\{2m,\ldots,2m+H^{\dagger}-1\}.
\]
To apply \Cref{cor-cdfcomparison}, take
$\mathscr{R}^{\dagger}$ to be trivial and use the corresponding
$\mathscr{G}_{\mathcal{J}^{\dagger}}$,
$V_{\mathcal{J}^{\dagger}}$, and $\mathscr{H}$
from \Cref{lem-smootheddeletion}.
For sufficiently large $n$,
$H^{\dagger}\geq64m$ and
$H^{\dagger}+2(m-1)\leq n-m$.
Thus $\mathcal{J}^{\dagger}$ and its adjacent boundary intervals
\[
\{m+1,\ldots,2m-1\},
\qquad
\{2m+H^{\dagger},\ldots,H^{\dagger}+3m-2\}
\]
lie in $\{m+1,\ldots,n\}$.
In particular, $D(x)$ is $\mathscr{H}$-measurable and
independent of $\mathscr{G}_{\mathcal{J}^{\dagger}}$.
Using \eqref{eq-af-03} again, we have
\[
\begin{aligned}
	\widetilde{S}_{n,m}^{[a]}(x)-V_{\mathcal{J}^{\dagger}}
	&=\sum_{k=1}^{m}H^{[a]}(g_{k},A_{k-1}x)
	+\sum_{k=m+1}^{2m-1}X_{k,m}^{[a]}\\
	&\quad+\sum_{k=H^{\dagger}+3m-1}^{n}X_{k,m}^{[a]}+\mathbb{E}\biggl[
	\sum_{k=2m}^{H^{\dagger}+3m-2}X_{k,m}^{[a]}
	\,\biggm|\,\mathscr{G}_{\mathcal{J}^{\dagger}}
	\biggr].
\end{aligned}
\]
Every term on the right is $\mathscr{H}$-measurable,
by the same  argument as above.
Applying \Cref{cor-cdfcomparison} with
$
T=\widetilde{S}_{n,m}^{[a]}(x)-V_{\mathcal{J}^{\dagger}},$
 $D=D(x)$
yields
\begin{align}\label{eq-va-03}
\begin{aligned}
	I_{2}&=\sup_{x\in X}\sup_{z\in\mathbb{R}}
	\bigl|
	\mathbb{P}(\widetilde{S}_{n,m}^{[a]}(x)+U\leq z)
	-\mathbb{P}(\widetilde{S}_{n,m}^{[a]}(W_{0})+U\leq z)
	\bigr|\\
	&\quad\leq
	C\biggl\{
	(H^{\dagger})^{-1/2}
	+\frac{1}{\alpha}\biggl(\frac{m}{H^{\dagger}}\biggr)^{2}
	\biggr\}
	\leq Cn^{-1/2},
\end{aligned}
\end{align}
where the last inequality follows from  $H^{\dagger}\asymp n\asymp m^{2}$,  $\delta\leq1$,
and
\[
\frac{1}{\alpha}\biggl(\frac{m}{H^{\dagger}}\biggr)^{2}
\leq Cm^{\delta-3}
\leq Cm^{-1}
\leq Cn^{-1/2}.
\]

Combining  \eqref{eq-va-01}, \eqref{eq-va-02} and \eqref{eq-va-03}, we obtain \eqref{eq-uu-03}. Hence \eqref{eq-tt-02} holds.
\subsection{Proof of  \protect\eqref{eq-tt-03}}
Let $\overline{S}_{n}^{[a]}
=\int_{X}S_{n,x}^{[a]}\,\nu(\mathrm{d}x).$ As in the proof of \eqref{eq-tt-02}, it suffices to prove that
\begin{align}\label{eq-uu-04}
	d_{\mathrm{K}}\Big(
	\mathcal{L}\Big(
	\overline{S}_{n}^{[a]}/(s\sqrt{n})
	\Big),\mathcal{N}(0,1)
	\Big)
	\leq Cn^{-\delta/2}.
\end{align}
We also work with sufficiently large $n$. Using the notation $q$, $h'$, $b'_{r}$, $\mathcal{I}'_{r}$,
$T'_{r}(x)$, and $D'_{r}(x)$ from \Cref{sec-5.2},
we define
\[
\overline{T}'_{r}
:=\int_{X}T'_{r}(x)\,\nu(\mathrm{d}x)
\quad(0\leq r\leq q).
\]
Then we have
$
\overline{T}'_{q}=\overline{S}_{n}^{[a]},$
$\overline{T}'_{0}
=\int_{X}\widetilde{S}_{n,m}^{[a]}(x)\,\nu(\mathrm{d}x).$
For $1\leq r\leq q$, put
\[
\overline{D}'_{r}
:=\overline{T}'_{r}-\overline{T}'_{r-1}
=\int_{X}D'_{r}(x)\,\nu(\mathrm{d}x).
\]
This random variable is $\mathscr{F}_{b'_{r}}$-measurable.
The triangle inequality, Tonelli's theorem, and the bound
for $D'_{r}(x)$ established in \Cref{sec-5.2} give
\[
\mathbb{E}|\overline{D}'_{r}|
\leq\int_{X}\mathbb{E}|D'_{r}(x)|\,\nu(\mathrm{d}x)
\leq |\mathcal{I}'_{r}|\delta_{1,\infty}(m)
\leq h'\delta_{1,\infty}(m).
\]

For each $1\leq r<q$, use the same choices of
$H^{\dagger}$, $\mathcal{J}^{\dagger}$, and the trivial
$\mathscr{R}^{\dagger}$ as in the comparison of
$T'_{r}(x)$ and $T'_{r-1}(x)$ in \Cref{sec-5.2}.
Since $V_{\mathcal{J}^{\dagger}}$ does not depend on $x$,
\[
\overline{T}'_{r}-V_{\mathcal{J}^{\dagger}}
=\int_{X}\bigl(T'_{r}(x)-V_{\mathcal{J}^{\dagger}}\bigr)
\,\nu(\mathrm{d}x)
\]
is $\mathscr{H}$-measurable. Moreover, $\overline{D}'_{r}$ is
$\mathscr{H}$-measurable and independent of
$\mathscr{G}_{\mathcal{J}^{\dagger}}$, since it is
$\mathscr{F}_{b'_{r}}$-measurable and both boundary
intervals contain only indices greater than $b'_{r}$.
Thus \Cref{cor-cdfcomparison} applies with
$
T=\overline{T}'_{r}-V_{\mathcal{J}^{\dagger}},$
$D=\overline{D}'_{r}.$
For $r=q$, smoothing by $U$ alone gives the bound
$\mathbb{E}|\overline{D}'_{q}|/(2\alpha)$.
Summing the bounds from \Cref{cor-cdfcomparison}
for $1\leq r<q$ and the bound above for $r=q$,
as in the derivation of \eqref{eq-ai-06}, we obtain
\begin{equation}\label{eq-va-04}
	\sup_{z\in\mathbb{R}}
	\bigl|
	\mathbb{P}(\overline{S}_{n}^{[a]}+U\leq z)
	-\mathbb{P}(\overline{T}'_{0}+U\leq z)
	\bigr|
	\leq Cm\delta_{1,\infty}(m).
\end{equation}

We next compare $\overline{T}'_{0}$ with
$\widetilde{S}_{n,m}^{[a]}(W_{0})$.
Using the definition of $D(x)$ in \eqref{eq-defi-dx}, we have
$
\overline{D}
:=\overline{T}'_{0}-\widetilde{S}_{n,m}^{[a]}(W_{0})
=\int_{X}D(x)\,\nu(\mathrm{d}x).
$
It is $\mathscr{F}_{m}$-measurable, and
\[
\mathbb{E}|\overline{D}|
\leq\int_{X}\mathbb{E}|D(x)|\,\nu(\mathrm{d}x)
\leq\sum_{k=1}^{m}\delta_{1,\infty}(k)
\leq C.
\]
Use the same $H^{\dagger}$, $\mathcal{J}^{\dagger}$,
and trivial $\mathscr{R}^{\dagger}$ as in the comparison
of $\widetilde{S}_{n,m}^{[a]}(x)$ and
$\widetilde{S}_{n,m}^{[a]}(W_{0})$ in \Cref{sec-5.2}.
Then $\overline{D}$ is $\mathscr{H}$-measurable and
independent of $\mathscr{G}_{\mathcal{J}^{\dagger}}$.
Furthermore,
\[
\overline{T}'_{0}-V_{\mathcal{J}^{\dagger}}
=\int_{X}\bigl(
\widetilde{S}_{n,m}^{[a]}(x)-V_{\mathcal{J}^{\dagger}}
\bigr)\,\nu(\mathrm{d}x)
\]
is $\mathscr{H}$-measurable.
Applying \Cref{cor-cdfcomparison} with
$T=\overline{T}'_{0}-V_{\mathcal{J}^{\dagger}}$
and $D=\overline{D}$, and using
$\mathbb{E}|\overline{D}|\leq C$ together with
the last inequality in \eqref{eq-va-03}, we obtain
\begin{equation}\label{eq-va-05}
	\sup_{z\in\mathbb{R}}
	\bigl|
	\mathbb{P}(\overline{T}'_{0}+U\leq z)
	-\mathbb{P}(\widetilde{S}_{n,m}^{[a]}(W_{0})+U\leq z)
	\bigr|
	\leq Cn^{-1/2}.
\end{equation}

Combining \eqref{eq-va-04}, \eqref{eq-va-05},
and \eqref{eq-va-02} with \Cref{prop-stationaryclipped},
we obtain
\[
\begin{aligned}
	&d_{\mathrm{K}}\biggl(
	\mathcal{L}\biggl(\frac{\overline{S}_{n}^{[a]}+U}{s\sqrt{n}}\biggr),
	\mathcal{N}(0,1)
	\biggr)\leq
	C\bigl\{n^{-\delta/2}+m\delta_{1,\infty}(m)+n^{-1/2}\bigr\}
	\leq Cn^{-\delta/2}.
\end{aligned}
\]
Finally, using $|U|\leq\alpha$ and arguing as in
\Cref{prop-Stein}, we obtain
\[
\begin{aligned}
	&d_{\mathrm{K}}\biggl(
	\mathcal{L}\biggl(\frac{\overline{S}_{n}^{[a]}}{s\sqrt{n}}\biggr),
	\mathcal{N}(0,1)
	\biggr)\leq
	d_{\mathrm{K}}\biggl(
	\mathcal{L}\biggl(\frac{\overline{S}_{n}^{[a]}+U}{s\sqrt{n}}\biggr),
	\mathcal{N}(0,1)
	\biggr)
	+\frac{\alpha}{s\sqrt{2\pi n}}
	\leq Cn^{-\delta/2}.
\end{aligned}
\]
This proves \eqref{eq-uu-04}.

\subsection{Proof of  \protect\eqref{eq-tt-04}}
By \cite[Proposition~4.5 and Lemma~4.7(i)]{BQ2016}
(see also \cite[Section~8.1, p.~1862]{CDJ2017}),
there exists a constant $C_{\nu}<\infty$ such that,
for every $g\in\mathrm{GL}(d,\mathbb{R})$,
\begin{equation}\label{eq-aj-02}
	0\leq
	\log\lVert g\rVert
	-\int_{X}\sigma(g,x)\,\nu(\mathrm{d}x)
	\leq C_{\nu}.
\end{equation}
Applying \eqref{eq-aj-02} with $g=A_{n}$, we have
\[
\overline{S}_{n}
\leq\log\lVert A_{n}\rVert-n\lambda_{\mu}
\leq\overline{S}_{n}+C_{\nu}.
\]
Thus, for every $z\in\mathbb{R}$,
\[
\begin{aligned}
	\mathbb{P}\Bigl(
	\frac{\overline{S}_{n}}{s\sqrt{n}}
	\leq z-\frac{C_{\nu}}{s\sqrt{n}}
	\Bigr)
	&\leq
	\mathbb{P}\Bigl(
	\frac{\log\lVert A_{n}\rVert-n\lambda_{\mu}}
	{s\sqrt{n}}\leq z
	\Bigr)\leq
	\mathbb{P}\Bigl(
	\frac{\overline{S}_{n}}{s\sqrt{n}}\leq z+\frac{C_{\nu}}{s\sqrt{n}}
	\Bigr).
\end{aligned}
\]
The normal approximation for $\overline{S}_{n}$
and the bound on the normal density imply
\[
\begin{aligned}
	&d_{\mathrm{K}}\biggl(
	\mathcal{L}\biggl(
	\frac{\log\lVert A_{n}\rVert-n\lambda_{\mu}}
	{s\sqrt{n}}
	\biggr),\mathcal{N}(0,1)
	\biggr)\leq
	Cn^{-\delta/2}
	+\frac{C_{\nu}}{s\sqrt{2\pi n}}
	\leq Cn^{-\delta/2}.
\end{aligned}
\]

%

\appendix
\section{\texorpdfstring{\protect Proofs of auxiliary lemmas}{Proofs of auxiliary lemmas}}\label{sec-app-a}
This appendix contains the proofs of \Cref{lem-A1,lem-absmod}.

\begin{proof}[Proof of \Cref{lem-A1}]
	By Lemma~6 of \cite{CDJ2017}, applied with moment order
	$p=2+\delta$, there exists $\ell>0$ such that
	\[
	\sum_{k=1}^{\infty}k^{\delta}
	\max_{k\leq j\leq 2k}\sup_{x,y\in X,\,x\ne y}
	\mathbb{P}\bigl(\log\rho(A_{j-1}x,A_{j-1}y)
	\geq-\ell k\bigr)<\infty.
	\]
Taking \(k=j=r+1\) and \(c=\ell/2\), we have
$
-cr\ge -\ell(r+1)$ and
$r^\delta\le (r+1)^\delta,
$
which yields \eqref{eq-ab-05}.

	To prove \eqref{eq-ab-06}, choose unit
	representatives $v,w$ for $x,y$ and change the sign of $w$, if
	necessary, so that $\langle v,w\rangle\geq 0$. Then
	\[
	\lVert v-w\rVert^{2}=2\bigl(1-\langle v,w\rangle\bigr)
	\leq 2\bigl(1-\langle v,w\rangle^{2}\bigr)
	=2\rho(x,y)^{2}.
	\]
	For positive $u,v$, integrating the derivative of the logarithm
	between $u$ and $v$ gives
	$|\log u-\log v|\leq|u-v|/\min(u,v)$.
	Together with the reverse triangle inequality and
	$\lVert gv\rVert,\lVert gw\rVert\geq\lVert g^{-1}\rVert^{-1}$,
	this yields
	\begin{align*}
		\big|\log\lVert gv\rVert-\log\lVert gw\rVert\big|
		&\leq\lVert g^{-1}\rVert\,
		|\lVert gv\rVert-\lVert gw\rVert|\\
		&\leq\lVert g^{-1}\rVert\lVert g\rVert\lVert v-w\rVert
		\leq\sqrt{2}N(g)^{2}\rho(x,y).
	\end{align*}
For the other bound in \eqref{eq-ab-06}, we have
\[
|\sigma(g,x)-\sigma(g,y)|
\leq |\sigma(g,x)|+|\sigma(g,y)|
\leq 2\log N(g).
\]
By the sine formula and the inequality
$
	\operatorname{dist}(gv,\mathbb{R}gw)
	\leq
	\lVert g\rVert\operatorname{dist}(v,\mathbb{R}w),
$
we have
\begin{align*}
	\rho(gx,gy)
	&=
	\frac{\operatorname{dist}(gv,\mathbb{R}gw)}
	{\lVert gv\rVert}
	\leq
	\frac{\lVert g\rVert
		\operatorname{dist}(v,\mathbb{R}w)}
	{\lVert gv\rVert}\\
	&\leq
	\lVert g\rVert\lVert g^{-1}\rVert\rho(x,y)
	\leq
	N(g)^2\rho(x,y),
\end{align*}
which proves \eqref{eq-ab-07}.

\end{proof}

To prove \Cref{lem-absmod}, we first establish a weighted
summability estimate for sums of nonnegative i.i.d. random
variables.
We will apply the estimate with $L_j=\log N(g_j)$.

\begin{lemma}\label{lem-completeconv}
	Let $(L_{j})_{j\geq 1}$ be independent and identically distributed
	nonnegative random variables with
	$\mathbb{E}L_{1}^{2+\delta}<\infty$, and put
	$T_{n}=\sum_{j=1}^{n}L_{j}$. Then
	\begin{equation}
		\sum_{n=1}^{\infty}n^{\delta-1}
		\mathbb{E}\bigl[T_{n}\mathbf{1}_{\{T_{n}>n(\mathbb{E}L_{1}+1)\}}\bigr]
		<\infty.
		\label{eq-ab-16}
	\end{equation}
\end{lemma}
\begin{proof}
	Decompose the expectation in \eqref{eq-ab-16} as $I_1+I_2$, where
	\begin{align*}
	I_{1}&=\mathbb{E}\bigl[T_{n}\mathbf{1}_{\{T_{n}>n(\mathbb{E}L_{1}+1)\}}
	\mathbf{1}_{\{\max_{j\leq n}L_{j}>n\}}\bigr],\\
	I_{2}&=\mathbb{E}\bigl[T_{n}\mathbf{1}_{\{T_{n}>n(\mathbb{E}L_{1}+1)\}}
	\mathbf{1}_{\{\max_{j\leq n}L_{j}\leq n\}}\bigr].
	\end{align*}
By independence, nonnegativity, and the fact that the random variables are identically distributed, we have
	\begin{align}	\label{eq-ab-17}
		\begin{aligned}
					I_{1}&\leq \mathbb{E}\bigl[T_{n}\mathbf{1}_{\{\max_{j\leq n}L_{j}>n\}}\bigr]
			\leq\sum_{i=1}^{n}\sum_{j=1}^{n}
			\mathbb{E}\bigl[L_{i}\mathbf{1}_{\{L_{j}>n\}}\bigr]
		\\
			&=n\mathbb{E}\bigl[L_{1}\mathbf{1}_{\{L_{1}>n\}}\bigr]
			+n(n-1)\mathbb{E}L_{1}\,\mathbb{P}(L_{1}>n).
		\end{aligned}
	\end{align}
 By Tonelli's theorem, we have
	\begin{align*}
		\sum_{n=1}^{\infty}n^{\delta}
		\mathbb{E}\bigl[L_{1}\mathbf{1}_{\{L_{1}>n\}}\bigr]
		&=\mathbb{E}\Bigl[L_{1}\sum_{n=1}^{\infty}n^{\delta}\mathbf{1}_{\{L_{1}>n\}}\Bigr]
		\leq C\mathbb{E}L_{1}^{2+\delta},\\
		\sum_{n=1}^{\infty}n^{1+\delta}\mathbb{P}(L_{1}>n)
		&=\mathbb{E}\Bigl[\sum_{n=1}^{\infty}n^{1+\delta}\mathbf{1}_{\{L_{1}>n\}}\Bigr]
		\leq C\mathbb{E}L_{1}^{2+\delta}.
	\end{align*}
 Thus \eqref{eq-ab-17}, multiplied by
	$n^{\delta-1}$, is summable.
	
	We next deal with $I_{2}.$
For convenience, let
	\[
	\xi_{n,j}=L_{j}\mathbf{1}_{\{L_{j}\leq n\}}
	-\mathbb{E}\bigl[L_{1}\mathbf{1}_{\{L_{1}\leq n\}}\bigr],
	\qquad Q_{n}=\sum_{j=1}^{n}\xi_{n,j}.
	\]
On the event $\{\max_{1\leq j\leq n}L_j\leq n\}$, we have
	$T_{n}=Q_{n}+n\mathbb{E}[L_{1}\mathbf{1}_{\{L_{1}\leq n\}}]$.
	If in addition $T_{n}>n(\mathbb{E}L_{1}+1)$, then $Q_{n}>n$.
	Therefore 
	\begin{align}
	I_{2}\leq 	n\mathbb{E}L_{1}\,\mathbb{P}(Q_{n}>n)
		+\mathbb{E}\bigl[Q_{n}\mathbf{1}_{\{Q_{n}>n\}}\bigr]
		&\leq(\mathbb{E}L_{1}+1)n^{-5}\mathbb{E}Q_{n}^{6}.
		\label{eq-ab-18}
	\end{align}
Since	the variables $\xi_{n,j}$ are independent and centered, it follows that
	\begin{align}\label{eq-ab-19}
	\begin{aligned}
		\mathbb{E}Q_{n}^{6}
={}&n\mathbb{E}\xi_{n,1}^{6}
+15n(n-1)\mathbb{E}\xi_{n,1}^{4}\mathbb{E}\xi_{n,1}^{2}
\\
&+10n(n-1)\bigl(\mathbb{E}\xi_{n,1}^{3}\bigr)^{2}
+15n(n-1)(n-2)\bigl(\mathbb{E}\xi_{n,1}^{2}\bigr)^{3}.
	\end{aligned}
	\end{align}
By the
	Cauchy--Schwarz inequality, we have
$
|\mathbb{E}\xi_{n,1}^{3}|^{2}
	\leq\mathbb{E}\xi_{n,1}^{4}\mathbb{E}\xi_{n,1}^{2},$
	$\mathbb{E}\xi_{n,1}^{4}
	\leq (\mathbb{E}\xi_{n,1}^{6}
	\mathbb{E}\xi_{n,1}^{2})^{1/2}.
$
	Consequently, we have
	\begin{align*}
		n^{2}\mathbb{E}\xi_{n,1}^{4}\mathbb{E}\xi_{n,1}^{2}
		&\leq
		\bigl(n\mathbb{E}\xi_{n,1}^{6}\bigr)^{1/2}
		\bigl(n^{3}(\mathbb{E}\xi_{n,1}^{2})^{3}\bigr)^{1/2}\leq\frac{n}{2}\mathbb{E}\xi_{n,1}^{6}
		+\frac{n^{3}}{2}(\mathbb{E}\xi_{n,1}^{2})^{3},
	\end{align*}
	where the last inequality follows from  the basic inequality $uv\leq(u^{2}+v^{2})/2$.
Substituting  the above upper bounds into \eqref{eq-ab-19}, we have 
\begin{align}\label{eq-ab-19a}
	\mathbb{E}Q_{n}^{6}
\leq Cn\mathbb{E}\xi_{n,1}^{6}
+Cn^{3}(\mathbb{E}\xi_{n,1}^{2})^{3}\leq Cn\E[L_{1}^{6}\mathbf{1}_{\{L_{1}\leq n\}}]+Cn^{3}.
\end{align}
	Combining  \eqref{eq-ab-18} and \eqref{eq-ab-19a}, we have
	\begin{align*}
		\begin{aligned}
					\sum_{n=1}^{\infty}n^{\delta-1} I_{2}&\leq C	\sum_{n=1}^{\infty}n^{\delta-5}\E[L_{1}^{6}\mathbf{1}_{\{L_{1}\leq n\}}]+
		C\sum_{n=1}^{\infty}n^{\delta-3}\\
		&\leq   C\E\Big[L_{1}^{6}\sum_{n=1}^{\infty}n^{\delta-5}\mathbf{1}_{\{L_{1}\leq n\}}\Big]+C\\
		&\leq C\E L_{1}^{2+\delta}+C\leq C,
		\end{aligned}
	\end{align*}
	where we used
	$\sum_{n\geq t}n^{\delta-5}\leq C t^{\delta-4}$ for $t\geq1$
	and $t^6\leq t^{2+\delta}$ for $0\leq t\leq1$.
	This proves \eqref{eq-ab-16}.
\end{proof}

Now, we are ready to prove   \Cref{lem-absmod}.

\begin{proof}[Proof of \Cref{lem-absmod}]
	We first prove the first inequality of  \eqref{eq-ac-03}.
	Fix a sufficiently large constant $C_{0}$, independent of $a,m,j,r$.
	We first consider the integers satisfying $r\leq C_{0}(j+1)$.
	By \eqref{eq-ac-01}, we have
	$
	\omega_{j,a}(r)\leq 2\delta_{1,\infty}(j+1).
	$
	Since $0<\delta\leq 1$,  it follows that
	$
	\sum_{1\leq r\leq C_{0}(j+1)}r^{\delta-1}
	\leq C(j+1)^{\delta}.
	$
	Hence
	\[
	\sum_{j=0}^{\infty}\sum_{1\leq r\leq C_{0}(j+1)}r^{\delta-1} \omega_{j,a}(r)\leq C\sum_{j=0}^{\infty}(j+1)^{\delta}\delta_{1,\infty}(j+1)<\infty.
	\]
	We next consider $r>C_{0}(j+1)$. For convenience, we  write
	$L_{i}=\log N(g_{i})$ and
	$T_{j}=\sum_{i=1}^{j}L_{i}$.  These are nonnegative i.i.d. random variables and
	their partial sums.  It follows from the definitions of  $P^{j}$ and $\phi^{[a]}$ that
	\begin{align*}
		&P^{j}\phi^{[a]}(x)
		-P^{j}\phi^{[a]}(y)=
		\mathbb{E}\bigl[
		\tau^{[a]}\bigl(X_{j+1,x}\bigr)
		-\tau^{[a]}\bigl(X_{j+1,y}\bigr)
		\bigr].
	\end{align*}
For $\rho(x,y)\leq e^{-r}$, on the event
	$T_{j+1}\leq r/4$, using the inequality 	$\rho(gx,gy)
	\leq
	N(g)^{2}\rho(x,y)$, we have
	\[
	\rho(A_{j}x,A_{j}y)
	\leq \exp\biggl(2\sum_{i=1}^{j}L_{i}\biggr)\rho(x,y)\leq \exp\biggl(2\sum_{i=1}^{j}L_{i}\biggr)e^{-r}.
	\]
	Combining this inequality with \Cref{lem-A1}(ii) and the $1$-Lipschitz property of $\tau^{[a]}$, we have
	\begin{align}\label{eq-ac-4.1}
		\begin{aligned}
		\bigl|\tau^{[a]}(X_{j+1,x})-\tau^{[a]}(X_{j+1,y})\bigr|
			&\leq |X_{j+1,x}-X_{j+1,y}|\\
			&\leq CN(g_{j+1})^{2}\rho(A_{j}x,A_{j}y)\\
			&\leq C\exp\bigl(2T_{j+1}-r\bigr)\leq C e^{-r/2}.
		\end{aligned}
	\end{align}
	On the complementary event, the same Lipschitz property and $|\sigma(g,x)|\leq \log N(g)$ give
	\[
	\bigl|\tau^{[a]}(X_{j+1,x})-\tau^{[a]}(X_{j+1,y})\bigr|\leq2L_{j+1}.
	\]
Splitting the expectation over these two events and taking the supremum over $\rho(x,y)\leq e^{-r}$, we have
	\begin{equation}\label{eq-ac-05}
		\omega_{j,a}(r)
		\leq C e^{-r/2}
		+C\mathbb{E}\bigl[L_{j+1}\mathbf{1}_{\{T_{j+1}>r/4\}}\bigr].
	\end{equation}
	We sum this estimate with a general weight $r^{\gamma-1}$, where
	$0<\gamma\leq 1$.  The exponential part has a finite   double sum:
	\[
	\sum_{j=0}^{\infty}
	\sum_{r>C_{0}(j+1)}r^{\gamma-1}e^{-r/2}
	\leq \sum_{r=1}^{\infty}r e^{-r/2}<\infty.
	\]
	For the other part,  since for any $t>0$,
	$
	\sum_{C_{0}(j+1)<r<4t}r^{\gamma-1}
	\leq C t^{\gamma}
	\mathbf{1}_{\{t>C_{0}(j+1)/4\}},
	$
	it follows that the sum at a fixed $j$ is at most
	\[
	C\mathbb{E}\bigl[
	L_{j+1}T_{j+1}^{\gamma}
	\mathbf{1}_{\{T_{j+1}>C_{0}(j+1)/4\}}
	\bigr].
	\]
For $j=0$, this is bounded by $C\mathbb{E}L_{1}^{1+\gamma}<\infty$.
For $j\geq1$, choose $c>\mathbb{E} L_{1}+1$ and then choose $C_{0}$ so large that
	$C_{0}/4\geq 2c$.  If both $L_{j+1}\leq cj$ and $T_{j}\leq cj$,
	then $T_{j+1}\leq 2cj<C_{0}(j+1)/4$.  Therefore,
	\[
	\{T_{j+1}>C_{0}(j+1)/4\}
	\subseteq \{L_{j+1}>cj\}\cup\{T_{j}>cj\}.
	\]
	Subadditivity gives
	$(T_{j}+L_{j+1})^{\gamma}\leq T_{j}^{\gamma}+L_{j+1}^{\gamma}$.
	Expanding over the two events,  and then using independence and
	$\mathbb{E} T_{j}^{\gamma}\leq(\mathbb{E} T_{j})^{\gamma}
	=j^{\gamma}(\mathbb{E} L_{1})^{\gamma}$, we have
	\begin{align}	\label{eq-ac-06}
		\begin{aligned}
		\sum_{ r> C_{0}(j+1)}r^{\gamma-1}\mathbb{E}\bigl[L_{j+1}\mathbf{1}_{\{T_{j+1}>r/4\}}\bigr]&\leq C\biggl\{
\mathbb{E}\bigl[L_{1}^{1+\gamma}\mathbf{1}_{\{L_{1}>cj\}}\bigr]
+j^{\gamma}\mathbb{E}\bigl[L_{1}\mathbf{1}_{\{L_{1}>cj\}}\bigr]
\\
&\quad+\mathbb{E} L_{1}^{1+\gamma}\,\mathbb{P}(T_{j}>cj)
+\mathbb{E} L_{1}\,
\mathbb{E}\bigl[T_{j}^{\gamma}\mathbf{1}_{\{T_{j}>cj\}}\bigr]
\biggr\}.
		\end{aligned}
	\end{align}
	Take first $\gamma=\delta$ and sum over all $j\geq 1$.
	For the first term, Tonelli's theorem gives
	\begin{align*}
		\sum_{j=1}^{\infty}
		\mathbb{E}\bigl[L_{1}^{1+\delta}\mathbf{1}_{\{L_{1}>cj\}}\bigr]
		&=\mathbb{E}\biggl[
		L_{1}^{1+\delta}\sum_{j=1}^{\infty}
		\mathbf{1}_{\{j<L_{1}/c\}}
		\biggr]\leq C\mathbb{E} L_{1}^{2+\delta}.
	\end{align*}
	For the second term, using the inequality
	$\sum_{1\leq j<t}j^{\delta}\leq C t^{1+\delta}$, we have
	\begin{align*}
		\sum_{j=1}^{\infty}j^{\delta}
		\mathbb{E}\bigl[L_{1}\mathbf{1}_{\{L_{1}>cj\}}\bigr]
		&=\mathbb{E}\biggl[L_{1}\sum_{1\leq j<L_{1}/c}j^{\delta}\biggr]
		\leq C\mathbb{E} L_{1}^{2+\delta}.
	\end{align*}
	For the probability term, combining
	$
	\mathbb{P}(T_{j}>cj)
	\leq
	\mathbb{E}[T_{j}\mathbf{1}_{\{T_{j}>cj\}}]/(cj)
	$ and
	\Cref{lem-completeconv}, we have
	\begin{equation}\label{eq-ac-prob-summable}
		\sum_{j=1}^{\infty}\mathbb{P}(T_{j}>cj)
		\leq C\sum_{j=1}^{\infty}j^{\delta-1}
		\mathbb{E}\bigl[
		T_{j}\mathbf{1}_{\{T_{j}\geq cj\}}
		\bigr]
		<\infty.
	\end{equation}
	Finally, on $T_{j}>cj$, the fact that
	$\delta-1\leq 0$ implies
	$T_{j}^{\delta}
	=T_{j}T_{j}^{\delta-1}
	\leq (cj)^{\delta-1}T_{j}$.
	Combining this with Lemma~\ref{lem-completeconv},
	we see that the fourth term on the right of
	\eqref{eq-ac-06}, with $\gamma=\delta$,
	is summable over $j\geq 1$.
	This proves the first inequality of \eqref{eq-ac-03}.

	For the second inequality of \eqref{eq-ac-03}, using the inequality $\omega_{j,a}(r)\leq 2\delta_{1,\infty}(j+1)$, we have
	\begin{align*}
		\sum_{j=0}^{m-2}
		\sum_{1\leq r\leq C_{0}(j+1)}\omega_{j,a}(r)
		&\leq C\sum_{j=0}^{m-2}(j+1)\delta_{1,\infty}(j+1)\\
		&\leq C m^{1-\delta}
		\sum_{j=0}^{m-2}(j+1)^{\delta}\delta_{1,\infty}(j+1)\\
		&\leq C m^{1-\delta},
	\end{align*}
	where the second inequality follows from
	$(j+1)^{1-\delta}\leq m^{1-\delta}$.
	For the region $r>C_{0}(j+1)$,  by \eqref{eq-ac-05}, we have
	\begin{align*}
		\sum_{j=0}^{m-2}
		\sum_{ r> C_{0}(j+1)}\omega_{j,a}(r)\leq C+	C\sum_{j=1}^{m-2}\sum_{ r> C_{0}(j+1)}\mathbb{E}\bigl[L_{j+1}\mathbf{1}_{\{T_{j+1}>r/4\}}\bigr].
	\end{align*}
	To bound the second term of the above inequality, we use \eqref{eq-ac-06} with $\gamma=1$. The sums over $j=1,\ldots,m-2$ of the first two terms
	on the right of \eqref{eq-ac-06} are bounded, respectively, by
	\[
	C\mathbb{E}\bigl[L_{1}^{2}\min\{m,L_{1}\}\bigr]
	\quad\hbox{and}\quad
	C\mathbb{E}\bigl[L_{1}\min\{m,L_{1}\}^{2}\bigr].
	\]
Using the inequalities
$t^{2}\min\{m,t\}\leq m^{1-\delta}t^{2+\delta}$
and
$t\min\{m,t\}^{2}\leq m^{1-\delta}t^{2+\delta}$,
we see that both expectations are bounded by $C m^{1-\delta}$.  By \eqref{eq-ac-prob-summable}, the third term
	on the right of \eqref{eq-ac-06}, with $\gamma=1$,
	is summable over $j\geq 1$.
	For the last term, we have
	\begin{align*}
		\sum_{j=1}^{m-2}
		\mathbb{E}\bigl[T_{j}\mathbf{1}_{\{T_{j}>cj\}}\bigr]
		&=\sum_{j=1}^{m-2}j^{1-\delta}
		j^{\delta-1}\mathbb{E}\bigl[T_{j}\mathbf{1}_{\{T_{j}>cj\}}\bigr]\\
		&\leq m^{1-\delta}
		\sum_{j=1}^{\infty}j^{\delta-1}
		\mathbb{E}\bigl[T_{j}\mathbf{1}_{\{T_{j}>cj\}}\bigr]\\
		&\leq C m^{1-\delta}.
	\end{align*}
	This proves the second inequality of  \eqref{eq-ac-03}.
	
	It remains to prove \eqref{eq-ac-04}.  Note that  the function
	$r\mapsto\omega_{j,a}(r)$ is nonincreasing.  For an integer $r\geq 2$, we have
	\begin{align*}
		C
		&\geq\sum_{t=\lceil r/2\rceil}^{r}
		t^{\delta-1}\sum_{j=0}^{\infty}\omega_{j,a}(t)\geq\biggl(\sum_{t=\lceil r/2\rceil}^{r}t^{\delta-1}\biggr)
		\sum_{j=0}^{\infty}\omega_{j,a}(r)\geq c r^{\delta}
		\sum_{j=0}^{\infty}\omega_{j,a}(r).
	\end{align*}
	The case $r=1$ follows directly from the first inequality in
	\eqref{eq-ac-03}.  For real $r\geq 1$, apply the integer
	bound at $\lfloor r\rfloor$ and use
	$\lfloor r\rfloor\geq r/2$.  This proves the same bound for real
	arguments as well.
\end{proof}

\end{document}